\documentclass[a4paper,12pt]{article}
\usepackage{amsthm,amsfonts,amssymb,euscript}
\usepackage{latexsym, multicol, fancybox}
\usepackage{graphicx}
\usepackage{color}
\usepackage{amsmath, amsthm, amssymb, bm}
\usepackage{epstopdf}
\usepackage{caption}
\usepackage{psfrag}
\usepackage{mathrsfs}
\usepackage{xcolor}
\usepackage{comment}
\usepackage{authblk}
\usepackage{indentfirst}
\usepackage[colorlinks=true,linkcolor=blue,citecolor=blue]{hyperref}
\DeclareFontFamily{U}{mathx}{\hyphenchar\font45}
\DeclareFontShape{U}{mathx}{m}{n}{
      <5> <6> <7> <8> <9> <10>
      <10.95> <12> <14.4> <17.28> <20.74> <24.88>
      mathx10
      }{}
\DeclareSymbolFont{mathx}{U}{mathx}{m}{n}
\DeclareFontSubstitution{U}{mathx}{m}{n}
\DeclareMathAccent{\widecheck}{0}{mathx}{"71}

\def\ombc{\widecheck{\omb}}

\def\ombcheck{\widecheck{\omb}}
\def\rhocheck{\widecheck{\rho}}
\def\mucheck{\widecheck{\mu}}

\def\bc{\widecheck{b}}
\def\Gb{\underline{G}}
\newtheorem{theorem}{Theorem}[section]
\newtheorem{lemma}[theorem]{Lemma}
\newtheorem{proposition}[theorem]{Proposition}
\newtheorem{corollary}[theorem]{Corollary}
\newtheorem{definition}[theorem]{Definition}
\newtheorem{remark}[theorem]{Remark}

\numberwithin{equation}{section}
\newcommand{\bea}{\begin{eqnarray}}
\newcommand{\eea}{\end{eqnarray}}
\def\beaa{\begin{eqnarray*}}
\def\eeaa{\end{eqnarray*}}
\def\ba{\begin{array}}
\def\ea{\end{array}}
\def\be#1{\begin{equation} \label{#1}}
\def\beq{\begin{equation}}
\def\eeq{\end{equation}}

\def\lab{\label}
\def\les{\lesssim}

\def\c{\cdot}

\DeclareMathOperator{\lot}{l.o.t.}
\def\dual{{\,^\star \mkern-2mu}}
\def\tr{\mbox{tr}}

\newcommand{\hch}{\hat{\chi}}
\newcommand{\hchb}{\underline{\hat{\chi}}}
\def\ov{\overline}
\renewcommand{\div}{\sdiv}
\newcommand{\curl}{\scurl}
\DeclareMathOperator{\sdiv}{div}
\DeclareMathOperator{\scurl}{curl}
\def\nab{\nabla}
\def\lap{\Delta}
\def\pr{\partial}
\def\dkb{\,\mathfrak{d}\mkern-9mu /}
\def\dk{\mathfrak{d}}
\def\dkt{{\widetilde{\dk}}}

\def\dddS{\ddd^{\,\S}}

\def\muS{\mu^\S}

\def\ddd{ \,  d \hspace{-2.4pt} \mkern-6mu /\,}

\def\ddsSone{{\ddd_1^{\,\S,\star}}}
\def\ddsStwo{{\ddd_2^{\,\S,\star}}}

\def\rhod{ \,^\star  \hspace{-2.2pt} \rho}

\def\ovu{\overset{\circ}{ u}}

\def\ovs{\overset{\circ}{ s}}

\def\ovr{\overset{\circ}{ r}}
\def\ovI{\overset{\circ}{ I}}
\def\ovla{\protect\overset{\circ}{ \la}\,}
\def\ovb{\protect\overset{\circ}{b}\,}
\def\ovS{\overset{\circ}{ \S}}

\def\epg{\protect\overset{\circ}{\ep}}
\def\ug{\overset{\circ}{u}}
\def\sg{\overset{\circ}{s}}
\def\rg{\overset{\circ}{r}}

\def\mg{\overset{\circ}{m}}

\def\ovu{\overset{\circ}{ u}}

\def\ovs{\overset{\circ}{ s}}

\def\ovr{{\overset{\circ}{ r}\,}}
\def\ovm{{\overset{\circ}{ m}\,}}
\def\ovS{\overset{\circ}{ S}}

\def\ovg{\overset{\circ}{g}}

  \def\muc{\widecheck{\mu}}

    \def\hb{\underline{h}}

\def\BS{{B^\S}}
\def\BbS{{\Bb^\S}}
\def\DS{{D^\S}}

\def\etaS{\eta^\S}
\def\etabS{\etab^\S}
\def\kaS{\ka^\S}
\def\kabS{\kab^\S}
\def\omS{\om^\S}
\def\ombS{\omb^\S}
\def\zeS{\ze^\S}

\def\bS{\b^\S}

\def\rhoS{\rho^\S}

\def\UpS{{\Up^\S}}
\def\rS{{r^\S}}
\def\uS{{u^\S}}
\def\nabS{\nab^\S}
\def\xibS{\xib^\S}
\def\GagS{{\Ga_g^\S}}
\def\GabS{{\Ga_b^\S}}

\def\divS{{\div^\S}}
\def\chihS{{\chih^\S}}
\def\chibhS{{\chibh^\S}}
\def\mS{{m^\S}}
\def\bS{{b^\S}}
\def\bcS{\bc^\S}

\def\zS{{z^\S}}
\def\zcS{{\widecheck{z}^\S}}
\def\OmbS{{\Omb^\S}}
\def\OmbcS{{\Ombc^\S}}
\def\Abr{{\breve{A}}}
\def\Bbr{{\breve{B}}}

\def\a{\alpha}
\def\b{\beta}
\def\ga{\gamma}
\def\Ga{\Gamma}
\def\de{\delta}
\def\De{\Delta}
\def\ep{\epsilon}
\def\la{\lambda}
\def\La{\Lambda}

\def\Si{\Sigma}
\def\om{\omega}
\def\Om{\Omega}

\def\rhod{\dual \rho}
\def\th{{\theta}}

\def\ka{\kappa}
\def\ze{\zeta}
\def\Up{\Upsilon}
\def\ks{\kappa}

\def\vep{\varepsilon}
\def\vphi{{\varphi}}

\def\vsi{\varsigma}

\renewcommand{\aa}{\protect\underline{\a}}
\newcommand{\bb}{\protect\underline{\b}}
\def\omb{\protect\underline{\om}}

\def\Omb{\underline{\Omega}}
\def\ub{{\underline{u}} }
\newcommand{\chib}{\protect\underline{\chi}}
\newcommand{\xib}{\protect\underline{\xi}}
\newcommand{\etab}{\protect\underline{\eta}}
\def\ksb{\underline{\kappa}}
\def\kab{\protect\underline{\kappa}}
\def\Nb{\underline{N}}
\def\epgf{\epg^\f12}
\DeclareMathOperator{\good}{Good}
\def\AA{{\mathcal A}}

\def\CC{{\mathcal C}}

\def\EE{{\mathcal E}}

\def\MM{{\mathcal M}}
\def\NN{{\mathcal N}}

\def\RR{{\mathcal R}}
\def\SS{{\mathcal S}}
\def\TT{{\mathcal T}}
\def\UU{{\mathcal U}}
\def\YY{{\mathcal Y}}
\def\XX{{\mathcal X}}

\def\D{{\bf D}}

\def\R{{\bf R}}

\def\S{{\bf S}}

\def\g{{\bf g}}

\def\Bb{\underline{B}}
\def\Eb{\underline{E}}
\def\Fb{\underline F}

\def\fb{\protect\underline{f}}
\def\CCb{\underline{\CC}}
\def\Cb{{\underline{C}}}

\def\Bbbr{\breve{\Bb}}
\def\BBbr{\breve{B}}
\def\DDbr{\breve{D}}
\def\Labr{\breve{\La}}
\def\Labbr{\breve{\Lab}}
\def\psibr{\breve{\psi}}
\def\Nbr{\breve{N}}
\def\Nbbr{\breve{\Nb}}
\def\Mbr{\breve{M}}
\def\Psibr{\breve{\Psi}}
\def\Pbr{\breve{P}}

\def\RRR{{\Bbb R}}
\def\SSS{{\mathbb{S}}}

\def\hk{\mathfrak{h}}

\def\rhoc{\widecheck{\rho}}
\def\Gac{\check \Gamma}

\def\kac{\widecheck \ka}
\def\kabc{\widecheck{\underline{\ka}}}

\def\Ombc{\widecheck{\underline{\Omega}}}

\def\yc{\widecheck{y}}
\def\zc{\widecheck{z}}
\def\chih{\widehat{\chi}}
\def\chibh{\widehat{\chib}}
\def\trch{\tr \chi}
\def\trchb{\tr \chib}
\def\hot{\widehat{\otimes}}
\def\c{\cdot}
\def \f12{\frac 1 2 }
\def\ov{\overline}
\def\err{\mbox{Err}}

\def\err{\mbox{Err}}

\def\Lab{{\protect\underline{\La}}}

\def\epg{\protect\overset{\circ}{\ep}}
\def\dg{\overset{\circ}{\de}}
\def\undB{\underline{B}}

\def\lapzero{{\overset{\circ}{ \lap}}}

\def\Jp{J^{(p)}}
\def\JpS{{J^{(\S, p)}}}
\def\Cbp{\Cb^{(p)}}
\def\Mp{M^{(p)}}
\def\CbpS{\Cb^{(\S, p)}}
\def\MpS{M^{(\S, p)}}

\def\kadot{\dot{\ka}}
\def\kabdot{\dot{\kab}}
\def\mudot{\dot{\mu}}
\def\Cbdot{\dot{\Cb}}
   \def\Mdot{\dot{M}}
   \def\Cbpdot{{\dot{\Cb}^{(p)}}}
   \def\Mpdot{{\dot{M}^{(p)}}}

\newcommand{\Mext}{\,{}^{(ext)}\mathcal{M}}

\newcommand{\Mint}{\,{}^{(int)}\mathcal{M}}

\def\Jt{{\widetilde{J}}}
\def\Jh{{\widehat{J}}}
\def\St{{\widetilde{\S}}}

\begin{document}
\title{Construction of GCM hypersurfaces in perturbations of Kerr}
\author{Dawei Shen\footnote{Email adress: dawei.shen@polytechnique.edu\\ \indent\hspace{0.175cm} Laboratoire Jacques-Louis Lions, Sorbonne Universit\'e, 75252 Paris, France}}
\maketitle
{\bf Abstract.} \textit{This is a follow-up of \cite{KS:Kerr1} on the general covariant modulated (GCM) procedure in perturbations of Kerr. In this paper, we construct GCM hypersurfaces, which play a central role in extending GCM admissible spacetimes in \cite{KS:main} where decay estimates are derived in the context of nonlinear stability of Kerr family for $|a|\ll m$. As in \cite{KS}, the central idea of the construction of GCM hypersurfaces is to concatenate a $1$--parameter family of GCM spheres of \cite{KS:Kerr1} by solving an ODE system. The goal of this paper is to get rid of the symmetry restrictions in the GCM procedure introduced in \cite{KS} and thus remove an essential obstruction in extending the results to a full stability proof of the Kerr family.}\\ \\
{\bf Keywords.} \textit{Frame transformations, GCM spheres, GCM hypersurfaces, Kerr stability.}
\tableofcontents
\section{Introduction}
\subsection{Stability of Kerr conjecture}
In this paper, we construct general covariant modulated (GCM) hypersurfaces, which play a central role in the proof of the nonlinear stability of Kerr family for $|a|\ll m$. \\ \\
{\bf Conjecture} (Stability of Kerr conjecture).\,\,{\it  Vacuum initial data sets,
   sufficiently close to Kerr initial data, have a maximal development with complete
   future null infinity\footnote{This means, roughly, that observers  which are  far away
    from the black hole  may live forever. } and with domain of outer communication which
   approaches  (globally)  a nearby Kerr solution.} \\ \\
The paper  builds on the strategy  laid out in \cite{KS} in the context of the nonlinear stability of Schwarzschild for axially symmetric polarized  perturbations. The central new idea of \cite{KS} was the introduction and construction of GCM spheres and GCM hypersurfaces on which specific geometric quantities take Schwarzschildian values. This was made possible by taking into account the full  general covariance of the Einstein vacuum equations.  The  construction, however,  also  made essential use of the polarization assumption.\\ \\
The goal of this, and its companion papers \cite{KS:Kerr1} and \cite{KS:Kerr2}, is to get rid of the symmetry restriction in the GCM procedure and thus  remove an  essential obstruction  in extending the  result in \cite{KS} to a full stability proof of the Kerr family.


\subsection{Stability of Schwarzschild in the polarized case}


\subsubsection{GCM admissible spacetimes in \texorpdfstring{\cite{KS}}{}}


In \cite{KS}, Klainerman and Szeftel proved the nonlinear stability of the Schwarzschild space under axially symmetric polarized perturbations. The  final spacetime in \cite{KS} was constructed as the limit of a continuous family of finite GCM admissible spacetimes as represented in Figure \ref{fig1-introd} below, whose future boundaries consist of the union $\AA\cup \CCb_* \cup \CC_* \cup \Si_*$ where $\AA$ and $ \Si_*$ are spacelike,  $\CCb_*$ is incoming null, and $\CC_*$ outgoing null.  The boundary $\AA$ is chosen so that, in the limit when $\MM$ converges to the final state, it is included inside the black hole region of the limit spacetime. The spacetime $\MM$ also contains a timelike hypersurface  $\TT$  which divides $\MM$ into  an exterior region we call $\Mext$ and an  interior one $\Mint$. Both $\Mext$ and $\Mint$  are foliated by $2$--surfaces as follows.
\begin{figure}[h!]
\centering
\includegraphics[scale=0.5]{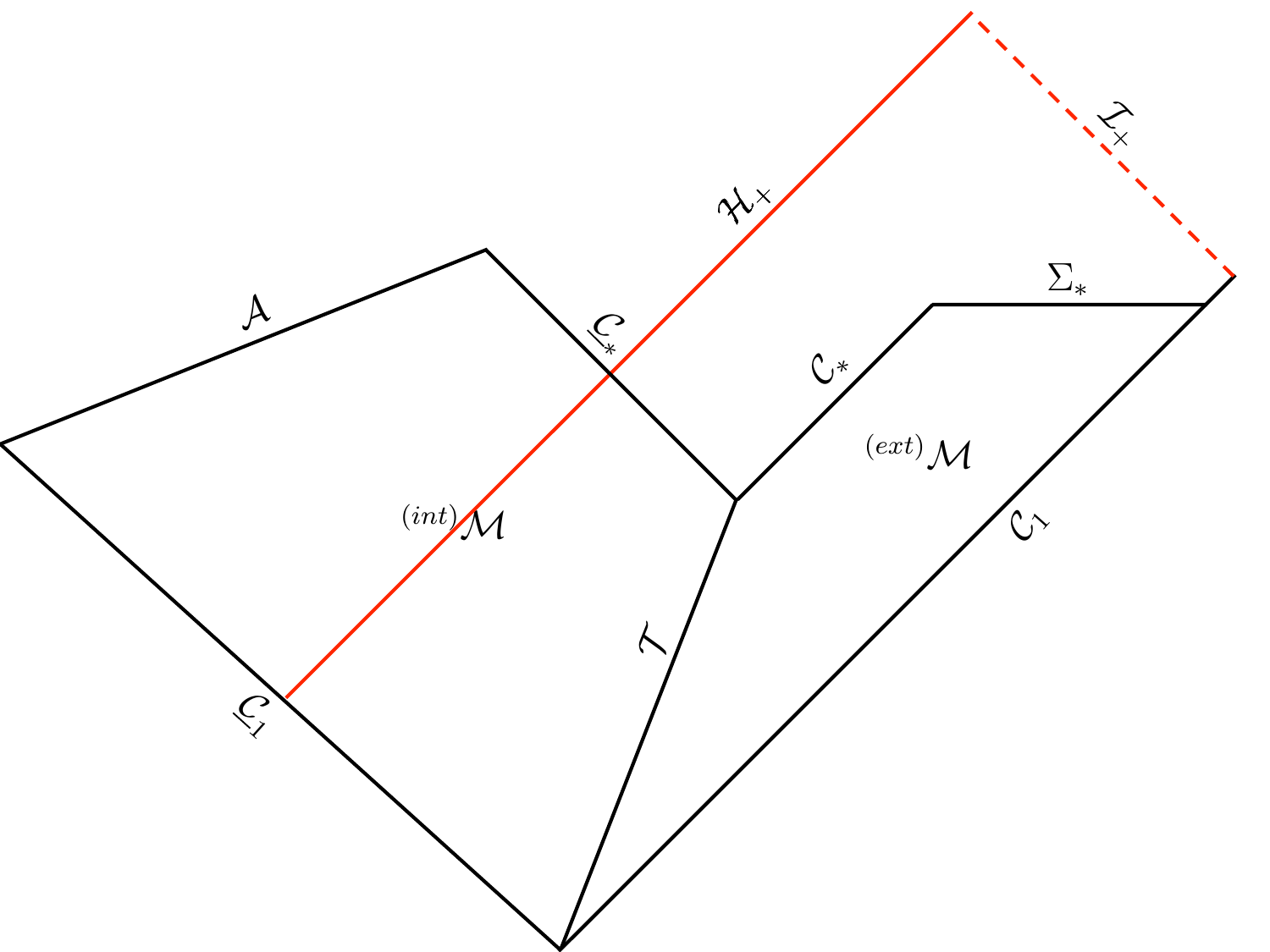}
\caption{The GCM admissible space-time $\mathcal{M}$ of \cite{KS}}
\label{fig1-introd}
\end{figure}
\begin{enumerate}
\item[(i)]  The  far region  $\Mext$ is foliated by a geodesic foliation $S(u, s)$ induced by an outgoing  optical function $u$  initialized on  $\Si_*$ with $s$ the affine parameter along the null geodesic generators of $\Mext$.  We denote by $r=r(u, s)$ the area radius of $S(u,s)$. On the boundary $\Si_*$ of $\Mext  $ we also assume that $r$ is sufficiently large.
\item[(ii)]  The  near region $\Mint $ is foliated by a geodesic foliation induced by an incoming optical function $\ub$ initialized at $\TT$ such that its level sets on $\TT$ coincide with  those of $u$.
\end{enumerate}
To prove convergence to the final state one has to establish precise decay estimates  for all Ricci  and curvature coefficients    decomposed relative  to the null geodesic  frames associated  to the foliations in $\Mext$ and $\Mint$. The decay properties of both  Ricci and curvature  coefficients in $\Mext$  depend heavily on the choice of the boundary $\Si_*$  as well as on the choice of  the cuts of the optical function $u$ on it. As such, the central idea of \cite{KS}  was  the  introduction and construction of GCM hypersurfaces on which specific geometric quantities take Schwarzschildian values.
\subsubsection{The role played by GCM admissible spacetimes}
As mentioned  above the  final spacetime was constructed as the limit of a continuous family of finite GCM admissible  spacetimes. At every stage one assumes that all Ricci and curvature coefficients of a fixed  GCM admissible spacetime $\MM$  verify precise  bootstrap assumptions.  One  makes use of the GCM admissibility properties of $\Si_*$ and the smallness of the initial conditions to show that all the bounds of the Ricci and curvature coefficients of $\MM$ depend only on the size of the initial data and thus, in particular, improve the bootstrap assumptions. This allows to extend the spacetime to a larger one $\MM'$ in which the bootstrap assumptions are still valid. To make sure that the extended spacetime is admissible, one has to construct a new GCM hypersurface $\widetilde{\Si}_*$ in $\MM'\setminus\MM$  and use it to define a new extended GCM admissible spacetime $\widetilde{\MM}$.

The goal of the present work is to extend the construction of $\Si_*$ of Section 9.8 of \cite{KS}, for polarized perturbations of Schwarzschild, to the case of general perturbations of Kerr.
\subsection{Review of the main results of \texorpdfstring{\cite{KS:Kerr1}}{}}
\label{reviewsection}
The main building block of our GCM hypersurface are the GCM spheres constructed in \cite{KS:Kerr1} which we now review.
\subsubsection{Background space}
As in \cite{KS:Kerr1}, we consider spacetime regions $\RR$ foliated by a geodesic foliation $S(u, s)$ induced by an outgoing optical function $u$ with $s$ a properly normalized affine parameter along the null geodesic generators of $L=-\g^{\a\b}\pr_\b  u\pr_\a $ where $\g$ is the spacetime metric. We denote by $r=r(u,s)$ the area radius of $S(u,s)$ and  let $(e_3, e_4, e_1, e_2)$ be an adapted null frame with $e_4$ proportional to $L$ and $e_1, e_2$ tangent to spheres $S=S(u, s)$, see Section \ref{sec:backgroundspacetime}.
The main assumptions made in \cite{KS:Kerr1}   were that    the Ricci   and curvature coefficients,  relative to the  adapted   null frame,  have the  same asymptotics in powers of $r$ as  in  Schwarzschild space.  Note that these  assumptions  hold true in the far region of Kerr and is  expected to hold true for  the far region of realistic perturbations of Kerr.   The actual size of the perturbation from   Kerr   is measured with respect to a small parameter $\epg>0$, see  sections  \ref{subsubsect:regionRR1} and \ref{subsubsect:regionRR2} for  precise definitions.
\subsubsection{Null frame transformation}
In general, two null frames $(e_3, e_4, e_1, e_2)$ and $(e_3', e_4', e_1', e_2')$ are related by a frame transformation of the following form:\footnote{See Lemma \ref{Lemma:Generalframetransf} for a precise statement.}
 \begin{align}\label{eq:Generalframetransf-intro}
 \begin{split}
  e_4'&=\la\left(e_4 + f^b  e_b +\frac 1 4 |f|^2  e_3\right),\\
  e_a'&= \left(\de_{ab} +\frac{1}{2}\fb_af_b\right) e_b +\frac 1 2  \fb_a  e_4 +\left(\frac 1 2 f_a +\frac{1}{8}|f|^2\fb_a\right)   e_3,\\
 e_3'&=\la^{-1}\left( \left(1+\frac{1}{2}f\c\fb  +\frac{1}{16} |f|^2  |\fb|^2\right) e_3 + \left(\fb^b+\frac 1 4 |\fb|^2f^b\right) e_b  + \frac 1 4 |\fb|^2 e_4 \right),
 \end{split}
 \end{align}
where the scalar $\la$ and the 1-forms $f$ and $\fb$ are called the transition coefficients from $(e_3,e_4,e_1,e_2)$ to $(e_3',e_4',e_1',e_2')$.
\subsubsection{Basis of \texorpdfstring{$\ell=1$}{} modes}
We introduce the following generalization of the $\ell=1$ spherical harmonics of the standard sphere\footnote{Recall that on the standard sphere $\SSS^2$, in spherical coordinates $(\th, \vphi)$,  these are $J^{(0, \SSS^2)}=\cos\th$, $J^{(+,\SSS^2)}=\sin\th\cos\vphi$, $J^{(-,\SSS^2)}=\sin\th\sin\vphi$.}.
\begin{definition}\label{jpdefintro}
On a sphere $S$, an $\epg$--approximated basis of $\ell=1$ modes is a triplet of functions $\Jp$ on $S$ verifying
 \begin{align}
\begin{split} \label{def:Jpsphericalharmonicsintro}
  (r^2\De+2) \Jp  &= O(\epg),\qquad p=0,+,-,\\
\frac{1}{|S|} \int_{S}\Jp J^{(q)} &=\frac{1}{3}\de_{pq} +O(\epg),\qquad p,q=0,+,-,\\
\frac{1}{|S|}\int_{S}\Jp&=O(\epg),\qquad p=0,+,-,
\end{split}
\end{align}
where $\epg>0$ is a sufficiently small constant.
\end{definition}
\begin{remark}
    For simplicity, throughout this paper, $J^{(p)}$ is called a basis of $\ell=1$ modes.
\end{remark}
Assuming  the existence of such a basis $\Jp$, $p\in\big\{ -, 0, +\big\}$, we   define, for a scalar function $h$,
\bea\label{defl=1intro}
( h)^S_{\ell=1} &:=&\left\{\int_{S} h \Jp, \quad  p=-, 0, +\right\}.
\eea
A scalar function $h$ is said to be supported on $\ell\le 1$ modes, i.e.  $(f)^S_{\ell\ge 2}=0$, if there exist constants $A_0, B_{-}, B_0, B_{+} $ such that
\bea
h=A_0+B_{-}J^{(-)}+B_0J^{( 0)}+B_{+}J^{(+)}.
\eea
\subsubsection{Definition of GCM spheres}
The null expansions $\ka:=\trch$ and $\kab:=\trchb$ relative to the adapted null frame $(e_3, e_4, e_1, e_2)$ are defined by
\begin{equation*}
    \trch:=\g^{ab}\chi_{ab},\qquad \trchb:=\g^{ab}\chib_{ab},
\end{equation*}
where
\begin{equation*}
    \chi_{ab}:=\g\left(\D_{e_a}e_4,e_b\right),\qquad \chib_{ab}:=\g\left(\D_{e_a}e_3,e_b\right).
\end{equation*}
The mass aspect function $\mu$ is defined by 
\begin{equation*}
    \mu:=-\div\ze-\rho+\f12\hch\cdot\hchb,
\end{equation*}
where the shears $\hch$, $\hchb$, the torsion $\ze$ and the curvature components $\rho$ are defined by
\begin{align*}
    \hch_{ab}&:=\chi_{ab}-\f12 \de_{ab} \ka,\qquad\quad\;\hchb_{ab}:=\chib_{ab}-\f12\de_{ab}\kab,\\
    \ze_a&:=\f12\g\left(\D_{e_a}e_4,e_3\right),\quad\qquad \rho:=\frac 1 4  \R(e_3,e_4,e_3,e_4).
\end{align*}
In an outgoing geodesic foliation of Schwarzschild spacetime, we have:
\bea\label{Introd:GCMspheres1}
 \ka=\frac{2}{r},\qquad\quad \kab=-\frac{2\Up}{r}, \quad\qquad \mu=\frac{2m}{r^3},
\eea
where $\Up=1-\frac{2m}{r}$ and $r$, $m$ denote the area radius and Hawking mass of $S$, i.e.
\begin{equation}
    r:=\sqrt{\frac{|S|}{4\pi}},\qquad\qquad \frac{2m}{r}:=1+\frac{1}{16\pi}\int_\S\ka\kab.
\end{equation}
The idea to construct GCM spheres is to mimic the condition \eqref{Introd:GCMspheres1} in the perturbed spacetimes. More precisely, the GCM spheres are topological spheres $\S$ embedded in $\RR$ endowed with a null frame $(e_3^\S, e_4^\S, e_1^\S, e_2^\S)$ adapted to $\S$ (i.e. $e_1^\S,e_2^\S$ tangent to $\S$), relative to which the null expansions $\ka^\S=\trch^\S$, $\kab^\S=\trchb^\S$ and mass aspect function $\mu^\S$ satisfy:
\bea\label{Introd:GCMspheres2}
\ka^\S -\frac{2}{r^\S} =0, \qquad \left(\kab^\S+\frac{2\Up^\S}{r^\S}\right)_{\ell\ge 2}=0, \qquad \left(\mu^\S-\frac{2m^\S}{(r^\S)^3} \right)_{\ell\ge 2}=0,
\eea
where $\rS$ and $\mS$ denote the area radius and Hawking mass of $\S$.
\subsubsection{Deformations of spheres and frame transformations}
The construction of GCM spheres in \cite{KS:Kerr1} was obtained by deforming a given sphere $\ovS=S(\ovu, \ovs)$ of the background foliation of $\RR$. An   $O(\dg)$  deformation of $\ovS$ is defined by a map $\Psi:\ovS\to \S \subset \RR$ of the form
\bea\label{eq:definitionofthedeformationmapPsiinintroduction}
 \Psi(\ovu,\ovs,y^1,y^2)=\left(\ovu+U(y^1, y^2),\ovs+S(y^1, y^2),y^1,y^2\right)
\eea
with $(U, S)$ smooth functions on $\ovS$, vanishing at a fixed point of $\ovS$,  of size  proportional to the small constant  $\dg$ and $(y^1,y^2)$ are spherical coordinates on $\ovS$.
Given such a deformation we identify, at any point on $\S$, two important null frames.
\begin{enumerate}
\item The null frame $(e_3, e_4, e_1, e_2)$ of the background foliation of $\RR$.
\item A null frame  $( e^\S_3, e^\S_4, e^\S_1, e^\S_2)$ obtained from \eqref{eq:Generalframetransf-intro} adapted to $\S$, (i.e. $e_1^\S$, $e_2^\S$ tangent to $\S$).
\end{enumerate}
 \begin{remark}
     Throughout this paper, we denote $(f,\fb,\la)$ the transition coefficients from the background frame $(e_3,e_4,e_1,e_2)$ of $\RR$ to the null frame $(e_3^\S,e_4^\S,e_1^\S,e_2^\S)$ adapted to $\S$.
 \end{remark}
\subsubsection{GCM spheres with \texorpdfstring{$\ell=1$}{} modes in \texorpdfstring{\cite{KS:Kerr1}}{}}
Here is a short version of the main result in \cite{KS:Kerr1}.
\begin{theorem}[Existence of GCM spheres in \cite{KS:Kerr1}]
\label{Theorem:ExistenceGCMS1-intro}
Let $\RR$ be fixed spacetime region, endowed with an outgoing geodesic foliation $S(u, s)$, verifying specific asymptotic assumptions\footnote{Compatible with small perturbations of Kerr.} expressed  in terms  of two parameters $0<\dg\leq \epg$.  In particular we assume that the  GCM quantities  of the background spheres in $\RR$, i.e.
\bea\label{Introd:GCMspheres3}
 \ka-\frac 2 r , \qquad \left( \kab+\frac{2\Up}{r}\right)_{\ell\ge 2 }, \qquad \left( \mu- \frac{2m}{r^3} \right)_{\ell\ge 2},
\eea
are small with respect to the parameter $\dg$. Let  $\ovS=S(\ovu, \ovs)$ be  a fixed sphere of the foliation with  $\rg$ and $\mg$ denoting respectively its area radius and  Hawking mass, with $\rg$ sufficiently large.
 Then, for any fixed triplets   $\La, \Lab \in \RRR^3$  verifying
\bea
|\La|,\,  |\Lab|  &\les & \dg,
\eea
 there exists a unique sphere $\S=\S(\La, \Lab)$, together with a null frame $(e_3^\S,e_4^\S,e_1^\S,e_2^\S)$, which is GCM, i.e. $\S$ is a deformation of $\ovS$, such that\footnote{Note that the GCM conditions \eqref{Introd:GCM spheres2-again} require a choice of $\ell=1$ modes on $\S$, see Remark  \ref{remark:noncan-modes-Intro}.}
 \bea
  \label{Introd:GCM spheres2-again}
 \ka^\S -\frac{2}{r^\S} =0, \qquad \left( \kab^\S+\frac{2\Up^\S}{r^\S}\right)_{\ell\ge 2 }=0, \qquad \left( \mu^\S - \frac{2m^\S}{(r^\S)^3} \right)_{\ell\ge 2}=0,
  \eea
  and
  \bea
  \label{Introd:GCM spheres-LaLab}
  (\div^\S f)_{\ell=1}=\La, \qquad (\div^\S\fb)_{\ell=1}=\Lab,
  \eea
  where $(f, \fb, \la)$ denote  the transition coefficients of the transformation \eqref{eq:Generalframetransf-intro}  from the background frame of $\RR$ to the frame adapted to $\S$.
    \end{theorem}
\begin{remark}
\lab{remark:noncan-modes-Intro}
 The conditions  \eqref{Introd:GCMspheres3},  \eqref{Introd:GCM spheres2-again} and \eqref{Introd:GCM spheres-LaLab} depend  on the  definition of $\ell=1$ modes  respectively on $\ovS$ and $\S$.  In \cite{KS:Kerr1}, once a choice   of $\ell=1$ modes on $\ovS$ is made, it is then extended to $\S$ using the background foliation.  As a consequence, the  GCM spheres of Theorem \ref{Theorem:ExistenceGCMS1-intro}  depend on the particular choice  of $\ell=1$ modes on $\ovS$.
\end{remark}
\subsection{First version of the main theorem}
The goal of this paper is to construct hypersurfaces which are suitable concatenations of the spheres of Theorem \ref{Theorem:ExistenceGCMS1-intro}. We give below a simplified version of our main theorem, see Theorem \ref{mthm} for the precise version.
\begin{theorem}[Existence of GCM hypersurfaces, first version]
\label{intromthm}
Let  $\RR$ be  fixed spacetime region, endowed with an outgoing geodesic foliation  $S(u, s)$, verifying same assumptions as Theorem \ref{Theorem:ExistenceGCMS1-intro}. Assume in addition that $e_3(\Jp)$, $(\div\eta)_{\ell=1}$, $(\div\xib)_{\ell=1}$, $r-s$ and $e_3(r)-e_3(s)$ are small with respect to the parameter $\dg$.\\\\
Let $\S_0$ be a fixed sphere included in the region $\RR$, let a pair of triplets $\La_0,\Lab_0\in\mathbb{R}^3$ such that
\beaa 
|\La_0|,\, |\Lab_0|\les \dg,
\eeaa 
and let $J^{(\S_0,p)}$ a basis of $\ell=1$ modes on $\S_0$, such that we have on $\S_0$
\beaa
\ka^{\S_0}-\frac{2}{r^{\S_0}} =0, \qquad \left( \kab^{\S_0}+\frac{2\Up^{\S_0}}{r^{\S_0}}\right)_{\ell\ge 2 }=0,\qquad \left( \mu^{\S_0} - \frac{2m^{\S_0}}{(r^{\S_0})^3}\right)_{\ell\ge 2}=0,
\eeaa
and
\beaa
(\div^{\S_0} f)_{\ell=1}=\La_0, \qquad (\div^{\S_0}\fb)_{\ell=1}=\Lab_0,
\eeaa
where $(f, \fb)$ denote the transition coefficients of the transformation \eqref{eq:Generalframetransf-intro}  from the background frame of $\RR$ to the frame adapted to $\S_0$.\\ \\
Then, there exists a unique, local, smooth, spacelike hypersurface $\Sigma_0$ passing through $\S_0$, a scalar function $u^\S$ defined on $\Sigma_0$, whose level surfaces are topological spheres denoted by $\S$, a smooth collection of triplets of constants $\La^\S,\Lab^\S$ and a triplet of functions $\JpS$ defined on $\Si_0$ verifying,
\beaa
\La^{\S_0}=\La_0,\qquad \Lab^{\S_0}=\Lab_0,\qquad \JpS\big|_{\S_0}=J^{(\S_0,p)},
\eeaa
such that the following conditions are verified:
\begin{enumerate}
\item The surfaces $\S$ of constant $u^\S$, together with a null frame $(e_3^\S,e_4^\S,e_1^\S,e_2^\S)$, verify 
 \bea
  \label{Introd:GCM spheres2-againagain}
 \ka^\S -\frac{2}{r^\S} =0, \qquad \left( \kab^\S+\frac{2\Up^\S}{r^\S}\right)_{\ell\ge 2 }=0, \qquad \left( \mu^\S - \frac{2m^\S}{(r^\S)^3} \right)_{\ell\ge 2}=0,
  \eea
  and
  \bea
  \label{Introd:GCM spheres-LaLabagain}
  (\div^\S f)_{\ell=1}=\La^\S, \qquad (\div^\S\fb)_{\ell=1}=\Lab^\S,
  \eea
for the triplets of constants $\La^\S,\Lab^\S$ and $\ell=1$ modes $\JpS$.
\item The following transversality conditions
\bea\label{transnew}
    \xi^\S=0,\qquad \omS=0,\qquad\etabS=-\zeS,
\eea
and
\bea \label{e4ue4rintro}
e_4^\S(\uS)=0,\qquad e_4^\S(\rS)=1
\eea
are assumed on $\Si_0$.
\item We have, for some constant $c_0$,
\bea\label{uSrSc}
u^\S+r^\S=c_0,\quad \mbox{along}\quad \Sigma_0.
\eea
\item Let $\nu^\S$ be the unique vectorfield tangent to the hypersurface $\Sigma_0$, normal to $\S$, and normalized by $\g(\nu^\S,e_4^\S)=-2$. Let $\bS$ be the unique scalar function on $\Si_0$ such that $\nu^\S$ is given by
\bea\label{defbS}
    \nu^\S=e_3^\S+\bS e_4^\S.
\eea
The following normalization condition holds true
\bea\label{norma}
    \ov{\bS}=-1-\frac{2m_{(0)}}{r^\S},
\eea
where $\ov{\bS}$ is the average value of $\bS$ over $\S$ and $m_{(0)}$ is a constant.
\item We have the following identities on $\Si_0$:
\bea\label{l=1etaSxibS}
(\div^\S \etaS)_{\ell=1}=0,\qquad (\div^\S \xib^\S)_{\ell=1}=0.
\eea
\item The transition functions $(f,\fb,\la)$, area radius $\rS$ and Hawking mass $\mS$ verify appropriate estimates.
\end{enumerate}
\end{theorem}
\begin{remark}
Theorem \ref{intromthm} is the generalization of Theorem 9.52 in \cite{KS} in the absence of symmetry. It plays a central role in the proof of Theorem M6 and M7 in \cite{KS:main}, see sections 8.4 and 8.5 in \cite{KS:main}.
\end{remark}
\begin{remark}
We provide below more explanations for the statements 1-5 in Theorem \ref{intromthm}:
\begin{enumerate}
    \item Since we concatenate a family of GCM spheres $\S(\La^\S,\Lab^\S)$ emanating from $\S_0$ to construct the GCM hypersurfaces $\Si_0$, by Theorem \ref{Theorem:ExistenceGCMS1-intro}, we have automatically \eqref{Introd:GCM spheres2-againagain} and \eqref{Introd:GCM spheres-LaLabagain} on every $\S$.
    \item The transversality conditions \eqref{transnew} and \eqref{e4ue4rintro} are consistent with a local extension by an outgoing geodesic foliation initialized on $\Si_0$, see item 1 in Remark \ref{rem4.5}. The role of these transversality conditions is to make sense of $\etaS$ and $\xibS$ on $\Si_0$, see item 5 below.
    \item $\uS$ should be chosen to be constant on the GCM spheres foliating $\Si_0$. The choice \eqref{uSrSc} is simple and fulfills this condition but other choices making $\Si_0$ spacelike are possible.
    \item The value $\ov{\bS}$ is free and should be prescribed. Note that the choice \eqref{norma} coincides with the value for the hypersurface $\{u+r=c_0\}$ in Schwarzschild spacetime.
    \item In \eqref{l=1etaSxibS}, $\etaS$ and $\xibS$ are defined intrinsically on $\Si_0$ by
    \begin{equation*}
    \etaS_a=\f12\g\left(\D_{\nu^\S} e_4^\S, e_a^\S\right),\qquad \xibS_a=\f12\g\left(\D_{\nu^\S}e_3^\S,e_a^\S\right)+\bS\zeS_a.
    \end{equation*}
    These definitions are consistent with the standard ones provided $\Si_0$ satisfies \eqref{transnew} which is equivalent to extending $\Si_0$ locally by an outgoing geodesic foliation, see item 2 in Remark \ref{rem4.5}. In the sequel, \eqref{l=1etaSxibS} will be enforced thanks to a special choice of $\La^\S$ and $\Lab^\S$, see Section \ref{sketchofproof} for more explanations.
\end{enumerate}
\end{remark}
\begin{remark}
    As in Section 9.8 of \cite{KS}, $\Si_0$ is chosen to be spacelike. One may wonder whether $\Si_0$ could be chosen to be null\footnote{In the context of the stability of Minkowski, the last slice in the original proof by Christodoulou and Klainerman in \cite{Ch-Kl} is spacelike, while it is null in the proof by Klainerman and Nicol\`o in \cite{KN} in the case of the exterior of an outgoing null cone.}. The reason for choosing it to be spacelike is that it allows more flexibility: all spheres foliating $\Si_0$ in Theorem \ref{intromthm} are GCM spheres, while only one could be a GCM sphere on a null hypersurface.
\end{remark}
\subsection{Sketch of the proof of the main theorem}\label{sketchofproof}
The idea of the proof is to construct $\Si_0$ as a $1$--parameter union of GCM spheres.\\ \\
\noindent{\bf Step 1.} For every background sphere $S(u,s)$ in $\RR$, every pair of triplets $(\La,\Lab)$ and every triplet of functions $\widetilde{J}^{(p)}$ satisfying
\beaa 
\sum_{p=0,+,-}\|\Jp-\widetilde{J}^{(p)}\|_{\hk_{s_{max}}(S(u,s))}\les r\dg,
\eeaa
where $\Jp$ is the $\ell=1$ modes on $S(u,s)$ and $\hk_{s_{max}}(S(u,s))$ denotes the Sobolev space on $S(u,s)$ of order $s_{max}$, we construct by Theorem \ref{Theorem:ExistenceGCMS1-intro} a unique GCM sphere $\S[u,s,\La,\Lab,\widetilde{J}^{(p)}]$, as a deformation of $S(u,s)$ with $\ell=1$ modes in the definition of \eqref{Introd:GCM spheres2-again} and \eqref{Introd:GCM spheres-LaLab} computed w.r.t. $\Jt^{(p)}$. In particular, \eqref{Introd:GCM spheres2-again} and \eqref{Introd:GCM spheres-LaLab} are verified and we have $\S_0=\S[\ovu,\ovs,\La_0,\Lab_0,\Jt^{(p)}]$, provided we choose $\widetilde{J}^{(p)}\big|_{\S_0}=J^{(\S_0,p)}$. \\ \\
\noindent{\bf Step 2.} Given $(\Psi(s),\La(s),\Lab(s))$ such that
\beaa
\Psi(\ovs)=\ovu,\qquad \La(\ovs)=\La_0,\qquad \Lab(\ovs)=\Lab_0,
\eeaa 
We construct, relying on Step 1 and a Banach fixed-point argument, see Theorem \ref{generalcoro}, a family of basis of $\ell=1$ modes $\Jt(s)$ and of GCM spheres $\S[\Psi(s),s,\La(s),\Lab(s),\Jt(s)]$ verifying
\beaa
\nu^\S (\Jt(s))=0 \quad \mbox{ on }\Si,\qquad \Jt^{(p)}(\ovs)=J^{(\S_0,p)},
\eeaa
where the hypersurface $\Si$ is given by
\beaa 
\Si=\bigcup_{s}\S(s)=\bigcup_{s}\S[\Psi(s),s,\La(s),\Lab(s),\Jt(s)],
\eeaa 
and where $\nu^\S$ is the unique vectorfield tangent to $\Si$ with $\g(\nu^\S,e_4^\S)=-2$ and normal to $\S(s)$, see Figure \ref{Sigmaconstruction} for a geometric description.
\begin{figure}[h!]
\centering
\includegraphics[scale=0.5]{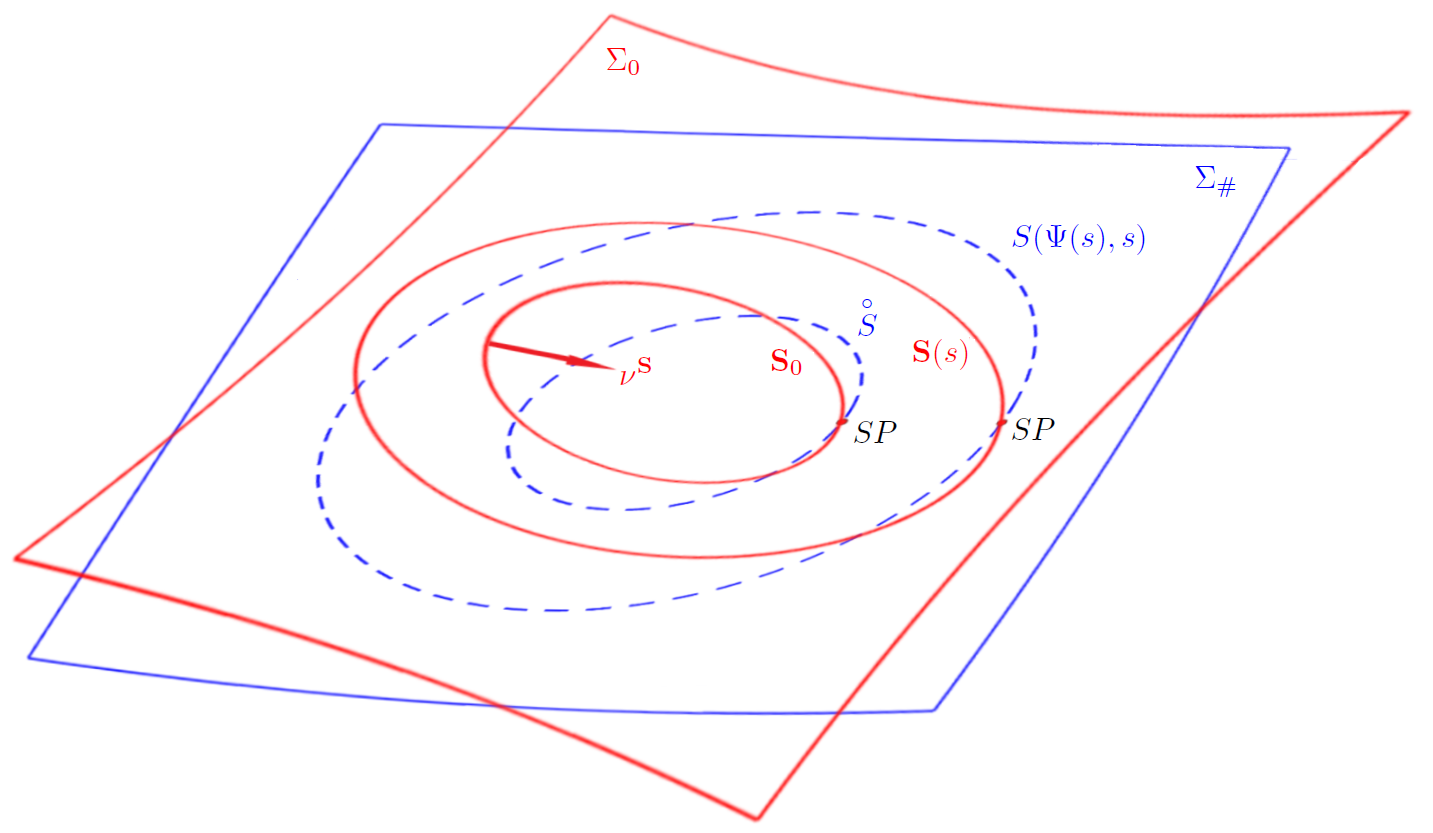}
\caption{The GCM hypersurface $\Si_0$ as a deformation of $\Si_\#=\{u=\Psi(s)\}$}
\label{Sigmaconstruction}
\end{figure}
\\ \\
\noindent{\bf Step 3.} We then derive for $(\Psi,\La,\Lab)$ an ODE system of the following type:
\begin{align}
\begin{split}\label{firstsystem}
\frac{1}{\Psi'(s)}\La'(s)=&(\div^\S\eta^\S)_{\ell=1}-\f12 r^{-1} \La(s) -\f12 r^{-1} \Lab(s)+\lot,\\
\frac{1}{\Psi'(s)}\Lab'(s)=&(\div^\S\xib^\S)_{\ell=1}+\lot,\\
\Psi'(s)=&-1-\f12\left(\ov{\bS}+1+\frac{2\mS}{\rS}\right)+\lot,
\end{split}
\end{align}
where $\lot$ denote lower order terms, see Section \ref{sec4.3} for the precise statement. \\ \\
\noindent{\bf Step 4.} We look for a special choice $(\Psibr(s),\Labr(s),\Labbr(s))$ such that the additional GCM conditions \eqref{norma} and \eqref{l=1etaSxibS} are verified. These conditions lead, in view of Step 3, to an ODE system for $(\Psibr(s),\Labr(s),\Labbr(s))$, with prescribed initial conditions at $\ovs$ which allows us to uniquely determine the desired hypersurface $\Si_0$.
\begin{remark}
The proof of Theorem \ref{intromthm} is largely analogous to that of Theorem 9.52 in \cite{KS}. Below, we compare the proof in this paper and that in Section 9.8 of \cite{KS}.
\begin{itemize}
    \item In {\bf Step 1} and {\bf Step 2}, we show that, in general, one can choose the approximate basis of $\ell=1$ modes so that they are transported along the normal direction to the GCM spheres $\S(s)$ on $\Si$. This contrasts with \cite{KS} where the basis of $\ell=1$ modes is fixed by the polarized symmetry.
    \item Once the choice of $\ell=1$ modes is made, in {\bf Step 3}, we derive the system of ODEs \eqref{firstsystem}. Note that the coefficients of linear terms of $\La(s)$ and $\Lab(s)$ on the R.H.S. of \eqref{firstsystem} are different from that of (9.8.74) in \cite{KS}, which is due to the different choice of $\ell=1$ modes.\footnote{More precisely, the basis of $\ell=1$ modes in \cite{KS}, fixed by polarized symmetry, is not transported along the vectorfield $\nu^\S$.}
    \item {\bf Steps 1-3} are significantly more involved than the corresponding in \cite{KS} due the absence of symmetry, while {\bf Step 4} is similar to that in \cite{KS}.
\end{itemize}
\end{remark}
\subsection{Structure of the paper}
The structure of the paper is as follows:
\begin{itemize}
\item In Section \ref{geometricsetup}, we review the geometric set-up of \cite{KS:Kerr1}.
\item In Section \ref{GCMz}, we review the main result of \cite{KS:Kerr1} on the construction of GCM spheres.
\item In Section \ref{mainsection}, we prove our main theorem concerning the construction of GCM hypersurfaces.
\end{itemize}
\section{Geometric set-up}\label{geometricsetup}
\subsection{Ricci and curvature coefficients}\label{gardef}
As in \cite{KS:Kerr1}, we consider given a vacuum spacetime $\RR$ with metric $\g$
endowed with an outgoing geodesic foliation  by  spheres $S(u, s)$   of fixed $(u, s) $, where $u$  is an outgoing optical function\footnote{That is  $\g^{\a\b}\pr_\a u\pr_\b u  =0$.}  with $L=-\g^{\a\b} \pr_\b u \pr_\a  $ its  null geodesic generator and $s$ chosen such that
\beaa
L(s) =\frac{1}{\vsi}, \qquad L(\vsi)=0.
\eeaa
Let $e_4=\vsi L $ and $e_3$  the unique null vectorfield orthogonal to  $S(u,s)$   and such that $\g(e_3, e_4)=-2$. We then let $(e_1, e_2)$ be an orthogonal basis of the tangent space of $S(u, s)$.
The  corresponding  connection coefficients relative to the null frame $(e_3, e_4, e_1, e_2)$ are denoted by $\chi, \chib, \xi, \xib, \om, \omb, \eta, \etab, \ze$ and the null components of the curvature tensor by $\a, \aa, \b, \bb, \rho,\rhod$. For the convenience of the reader  we recall their definition below:
 \begin{align}
 \begin{split}\label{defga}
\chib_{ab}&=\g(\D_ae_3, e_b),\qquad \,\,\chi_{ab}=\g(\D_ae_4, e_b),\\
\xib_a&=\frac 1 2 \g(\D_3 e_3 , e_a),\qquad \xi_a=\frac 1 2 \g(\D_4 e_4, e_a),\\
\omb&=\frac 1 4 \g(\D_3e_3 , e_4),\qquad\, \om=\frac 1 4 \g(\D_4 e_4, e_3), \\
\etab_a&=\frac 1 2 \g(\D_4 e_3, e_a),\qquad \eta_a=\frac 1 2 \g(\D_3 e_4, e_a),\\
 \ze_a&=\frac 1 2 \g(\D_{e_a}e_4,  e_3),
 \end{split}
\end{align}
and
\begin{align}
\begin{split}\label{defr}
\a_{ab} &=\R(e_a, e_4, e_b, e_4) , \qquad \,\,\,\aa_{ab} =\R(e_a, e_3, e_b, e_3), \\
\b_{a} &=\frac 1 2 \R(e_a, e_4, e_3, e_4), \qquad \bb_{a}=\frac 1 2 \R(e_a, e_3, e_3, e_4),\\
\rho&= \frac 1 4  \R(e_3, e_4, e_3, e_4), \qquad\, \rhod =\frac{1}{4}\dual\R( e_3, e_4, e_3, e_4),
\end{split}
\end{align}
where $\dual\R$ denotes the Hodge dual of $\R$. The null second fundamental forms $\chi, \chib$ are further  decomposed in their traces $\ka=\trch$ and $\kab=\trchb$, and traceless parts $\chih$ and $\chibh$:
\begin{align}
\begin{split}\label{expansionetc}
\ka&:=\trch=\de^{ab}\chi_{ab},\qquad \,\hch_{ab}:=\chi_{ab}-\f12\de_{ab}\ka,\\
\kab&:=\trchb=\de^{ab}\chib_{ab},\qquad \hchb_{ab}:=\chib_{ab}-\f12\de_{ab}\kab.
\end{split}
\end{align}
We define the horizontal covariant operator $\nab$ as follows:
\beaa
\nab_X Y&:=&\D_X Y-\f12\chib(X,Y)e_4-\f12\chi(X,Y)e_3.
\eeaa
We also define $\nab_4 X$ and $\nab_3 X$ to be the horizontal projections:
\beaa
\nab_4 X&:=&\D_4 X-\f12\g(X,\D_4e_3)e_4-\f12\g(X,\D_4e_4)e_3,\\
\nab_3 X&:=&\D_3 X-\f12\g(X,\D_3e_3)e_3-\f12\g(X,\D_3e_4)e_4.
\eeaa
As a direct consequence of \eqref{defga}, we have the Ricci formulas:
\begin{align}
\begin{split}\label{ricciformulas}
    \D_a e_b&=\nab_a e_b+\f12\chi_{ab} e_3+\f12\chib_{ab}e_4,\\
    \D_a e_3&=\chib_{ab}e_b+\ze_a e_3,\\
    \D_a e_4&=\chi_{ab}e_b-\ze_a e_4,\\
    \D_3 e_a&=\nab_3 e_a+\eta_a e_3+\xib_a e_4,\\
    \D_4 e_a&=\nab_4 e_a+\etab_a e_4+\xi_a e_4,\\
    \D_3 e_3&=-2\omb e_3+2\xib_b e_b,\\ 
    \D_3 e_4&=2\omb e_4+2\eta_b e_b,\\
    \D_4 e_4&=-2\om e_4+2\xi_b e_b,\\
    \D_4 e_3&=2\om e_3+2\etab_b e_b.
\end{split}
\end{align}
We recall, see Lemma 2.4 in \cite{KS:Kerr1}, that the geodesic nature of the foliation implies
\beaa
\om=\xi=0,  \qquad  \etab = -\ze, \qquad  \vsi=\frac{2}{e_3(u)}.
\eeaa
We denote
\bea\label{defz}
\Omb:=e_3(r),\qquad z:=e_3(u)=\frac{2}{\vsi}.
\eea
The area radius $ r=r(u, s)$ is defined   such  that the volume of $S$ is given by $4\pi r^2$. The Hawking mass  $m=m(u, s)$ of $S=S(u, s) $ is given by the formula
 \bea\label{eq:definitionoftheHawkingmass}
\frac{2m}{r}=1+\frac{1}{16\pi}\int_{S_{}}\ka\kab.
\eea
The Gauss curvature of $S$ is denoted by $K$. It verifies the Gauss equation
\bea\label{eq:Gaussequation}
K=-\rho -\frac{1}{4} \ka\kab +\frac{1}{2}\chih\c\chibh.
\eea
The mass aspect function $\mu$  is defined by
\bea
\label{def:massaspectfunctions.general}
\mu &:=& - \div \ze -\rho+\frac 1 2  \chih\c \chibh.
\eea
As in \cite{KS:Kerr1}, we define the renormalized quantities\footnote{Renormalized quantities are obtained by subtracting their Schwarzschild values.}
\begin{align}
\begin{split}\label{renor}
\widecheck{\trch} &:= \trch -\frac{2}{r}, \qquad
\widecheck{\trchb}:= \trchb +\frac{2\Up}{r},\qquad
\widecheck{\omb} := \omb -\frac{m}{r^2},\\
\widecheck{K} &:= K -\frac{1}{r^2},\qquad\,\,\,\,\,\,\widecheck{\rho} := \rho +\frac{2m}{r^3},\qquad\,\,\,\,\,\widecheck{\mu} := \mu -\frac{2m}{r^3}, \\
\widecheck{\Omb}& :=\Omb+\Up, \qquad\quad\,\,\,\,
\widecheck{\varsigma} := \varsigma-1, \qquad\qquad\zc:=z-2,
\end{split}
\end{align}
where
\beaa
\Up :=1-\frac{2m}{r},
\eeaa
and the sets
\begin{align}
\label{definition:Ga_gGa_b}
\begin{split}
\Ga_g &:= \Bigg\{\kac,\,\,\chih,\,\,\ze,\,\, \kabc,\,\, r\widecheck{\mu},\,\,  r\widecheck{\rho}, \,\, r\dual\rho, \,\, r\b, \,\, r\a, \,\, r\widecheck{K}\Bigg\},\\
\Ga_b &:= \Bigg\{\eta, \,\,\chibh, \,\, \ombc, \,\, \xib,\,\,  r\bb,\,\,\aa,\,\, r^{-1}\widecheck{\Omb}, \,\,r^{-1}\widecheck{\varsigma},\,\,r^{-1}\zc,\,\, r^{-1}(e_3(r)+\Up\big)\Bigg\}.
\end{split}
\end{align}
\begin{remark}
    The renormalized Ricci coefficients are divided into two groups $\Ga_g$ and $\Ga_b$ according to their $r$--weights. See Remark \ref{Gammagb} for more details.
\end{remark}
\subsection{Basic equations for the renormalized quantities}
We recall some of the null structure equations and Bianchi identities for the renormalized quantities, see Proposition 5.1.17 in \cite{KS:main}.
\begin{proposition}\label{nullstructure}
In an outgoing geodesic foliation, the following null structure equations hold true 
\begin{align*}
    \nab_4\kac+\frac{2}{r}\kac&=\Ga_g\c\Ga_g,\\
    \nab_4\hch+\frac{2}{r}\hch&=-\a+\Ga_g\c\Ga_g,\\
    \nab_4\ze+\frac{2}{r}\ze&=-\b+\Ga_g\c\Ga_g,\\
    \nab_4\kabc+\frac{1}{r}\kabc&=-2\div\ze+2\rhoc+\frac{\Up}{r}\kac+\Ga_b\c\Ga_g,\\
    \nab_4\hchb+\frac{1}{r}\hchb&=\frac{\Up}{r}\hch-\nab\hot\ze+\Ga_b\c\Ga_g,
\end{align*}
and
\begin{align*}
    \nab_3\kac-\frac{\Up}{r}\ka&=2\div\eta+2\rhoc-\frac{1}{r}\kabc+\frac{4}{r}\ombc+\frac{2m}{r^2}\kac+\frac{2}{r^2}\yc+\Ga_b\c\Ga_b,\\
    \nab_3\kabc-\frac{2\Up}{r}\kabc&=2\div\xib+\frac{4\Up}{r}\ombc-\frac{2m}{r^2}\kabc-\left(\frac{2}{r^2}-\frac{8m}{r^3}\right)\yc+\Ga_b\c\Ga_b,\\
    \nab_3\hchb-\frac{2\Up}{r}\hchb&=-\aa-\frac{2m}{r^2}\hchb+\nab\hot\xib+\Ga_b\c\Ga_b,\\
    \nab_3\ze-\frac{\Up}{r}\ze&=-\bb-2\nab\ombc+\frac{\Up}{r}(\eta+\ze)+\frac{1}{r}\xib+\frac{2m}{r}(\ze-\eta)+\Ga_b\c\Ga_b,\\
    \nab_3\hch-\frac{\Up}{r}\hch&=\nab\hot\eta-\frac{1}{r}\hchb+\frac{2m}{r^2}\hch+\Ga_b\c\Ga_b.
\end{align*}
Also,
\begin{align*}
    \curl\eta&={^*\rho}+\Ga_b\c\Ga_g,\\
    \curl\xib&=\Ga_b\c\Ga_b,\\
    \muc&=-\div\ze-\rhoc+\Ga_b\c\Ga_g.
\end{align*}
The Bianchi identities are given by
\begin{align*}
    \nab_3\a-\frac{\Up}{r}\a&=\nab\hot\b+\frac{4m}{r^2}\a+\frac{6m}{r^3}\hch+\Ga_b\c(\a,\b)+r^{-1}\Ga_g\c\Ga_g,\\
    \nab_4\b+\frac{4}{r}\b&=-\div\a+r^{-1}\Ga_g\c\Ga_g,\\
    \nab_3\b-\frac{2\Up}{r}\b&=(\nab\rho+{^*\nab}{^*\rho})+\frac{2m}{r^2}\b-\frac{6m}{r^3}\eta+r^{-1}\Ga_b\c\Ga_g,\\
    \nab_4\rhoc+\frac{3}{r}\rhoc&=\div\b+\frac{3m}{r^3}\kac+r^{-1}\Ga_b\c\Ga_g,\\
    \nab_3\rhoc-\frac{3\Up}{r}&=-\div\bb+\frac{3m}{r^3}\kabc-\frac{6m}{r^4}\Ombc-\f12\hch\c\aa+r^{-1}\Ga_b\c\Ga_b,\\
    \nab_4{^*\rho}+\frac{3}{r}{^*\rho}&=-\curl\b+r^{-1}\Ga_b\c\Ga_g,\\
    \nab_3{^*\rho}-\frac{3\Up}{r}{^*\rho}&=-\curl\bb-\f12\hch\c{^*\aa}+r^{-1}\Ga_b\c\Ga_b.
\end{align*}
\end{proposition}
\begin{proof}
The proof follows immediately from the standard null structure equations and Bianchi identities (see Propositions 7.3.2 and 7.4.1 in \cite{Ch-Kl}) and the definition \eqref{renor} of the renormalized quantities, see Proposition 5.1.17 in \cite{KS:main} for more details.
\end{proof}
\subsection{Commutation lemmas}
We recall the following commutation lemmas.
\begin{lemma}\label{comm}
For any tensor on a sphere $S$, the following commutation formulas hold true
\begin{align*}
[\nab_3,\nab_a]f&=-\f12\trchb\nab_a f+(\eta_a-\ze_a)\nab_3 f-\hchb_{ab}\nab_b f+\xib_a\nab_4 f+(\Fb[f])_a,\\
[\nab_4,\nab_a]f&=-\f12\trch\nab_a f+(\etab_a+\ze_a)\nab_4 f-\hch_{ab}\nab_b f+\xi_a\nab_3f+(F[f])_a,
\end{align*}
where the tensors $F[f]$ and $\Fb[f]$ have the following schematic form
\beaa 
F[f]=(\b,\chi\c\etab,\chib\c\xi)\c f,\qquad \Fb[f]=(\bb,\chib\c\eta,\chi\c\xib)\c f.
\eeaa 
\end{lemma}
\begin{proof}
See Lemma 7.3.3 in \cite{Ch-Kl}.
\end{proof}
\begin{lemma}\label{commfor}
The following commutation formulas hold true for any tensor $f$ on a sphere $S$:
\begin{align*}
    [\nab_3,\nab]f&=\frac{\Up}{r}\nab f+\Ga_b\cdot\nab_3 f+r^{-1}\Ga_b\cdot\dk^{\leq 1}f,\\
    [\nab_4,\nab]f&=-\frac{1}{r}\nab f+r^{-1}\Ga_g\cdot\dk^{\leq 1}f,\\
    [\nab_\nu,\nab]f&=\frac{2}{r}\nab f+\Ga_b\cdot\nab_\nu f+r^{-1}\Ga_b\cdot\dk^{\leq 1}f,\\
    [\nab_\nu,\De]f&=\frac{4}{r}\nab f+r^{-1}\dkb^{\leq 1}\left(\Ga_b\c\nab_\nu f+r^{-1}\Ga_b\c\dk f\right),
\end{align*}
where
\begin{equation}\label{dkdkbdf}
    \dk\in \{\nab_3,\, r\nab_4,\,\dkb\},\qquad \dkb:=r\nab,\qquad \dk^{\leq 1}f:=(f,\dk f).
\end{equation}
\end{lemma}
\begin{proof}
See Lemma 5.1.19 in \cite{KS:main}.
\end{proof}
\subsection{Hodge operators}
We recall the following Hodge operators acting on $2$--surface $S$ (see Chapter 2 in \cite{Ch-Kl}):
\begin{definition}
We define the following Hodge operators:
\begin{enumerate}
    \item The operator $\ddd_1$ takes any $1$--form $f$ into the pair of functions $(\div f,\curl f)$.
    \item The operator $\ddd_2$ takes any symmetric traceless $2$--tensor $f$ into the $1$--form $\div f$.
    \item The operator $\ddd_1^*$ takes any pair of scalars $(h,{^*h})$ into the $1$--form $-\nab h+{^*\nab}{^*h}$.
    \item The operator $\ddd_2^*$ takes any $1$--form $f$ into the symmetric traceless $2$--tensor $-\f12\nab\hot f$.
\end{enumerate}
\end{definition}
On can easily check that $\ddd_k^*$ is the formal adjoint on $L^2(S)$ of $\ddd_k$ for $k=1,2$. Moreover,
\bea
\begin{split}
\ddd_1^*\ddd_1&=-\De_1+K,\qquad\;\,\,\ddd_1\ddd_1^*=-\De_0,\\
\ddd_2^*\ddd_2&=-\f12\De_2+K,\qquad \ddd_2\ddd_2^*=-\f12(\De_1+K),
\end{split}
\eea
where $\Delta_k$, $k=0,1,2$ is the Laplace operator on scalars, $1$--forms and symmetric traceless $2$--tensors.
\subsection{Definition of \texorpdfstring{$\ell=1$}{} modes}
\begin{definition}\label{jpdef}
On a sphere $S$, a (suitable $\epg$--approximation of) basis of $\ell=1$ modes is a triplet of functions $\Jp$ on $S$ satisfying\footnote{The properties  \eqref{def:Jpsphericalharmonics} of the scalar functions $\Jp$ are motivated by the fact that the $\ell=1$ spherical harmonics  on the standard sphere  $\SSS^2$,  given by  $J^{(0, \SSS^2)}=\cos\th,\, J^{(+, \SSS^2)}=\sin\th\cos\vphi, \,  J^{(-, \SSS^2)}=\sin\th\sin\vphi$,
satisfy  \eqref{def:Jpsphericalharmonics}  with $\epg=0$. Note also  that on $\SSS^2$, there holds
\beaa
\int_{\mathbb{S}^2}(\cos\th)^2=\int_{\mathbb{S}^2}(\sin\th\cos\vphi)^2=\int_{\mathbb{S}^2}(\sin\th\sin\vphi)^2=\frac{4\pi}{3}, \qquad |\SSS^2|=4\pi.
\eeaa} 
 \begin{align}
\begin{split} \label{def:Jpsphericalharmonics}
  (r^2\De+2) \Jp  &= O(\epg),\qquad p=0,+,-,\\
\frac{1}{|S|} \int_{S}\Jp J^{(q)} &=\frac{1}{3}\de_{pq} +O(\epg),\qquad p,q=0,+,-,\\
\frac{1}{|S|}\int_{S}\Jp&=O(\epg),\qquad p=0,+,-,
\end{split}
\end{align}
where $\epg>0$ is a sufficiently small constant.
\end{definition}
\begin{remark}
    For simplicity, throughout this paper, $J^{(p)}$ is called a basis of $\ell=1$ modes.
\end{remark}
\begin{proposition}\label{jpprop}
For a basis of $\ell=1$ modes $\Jp$ defined in Definition \ref{jpdef}, we have
\bea\label{estjpepg}
\ddd_2^*\ddd_1^*\Jp&=&r^{-2}O(\epg).
\eea    
\end{proposition}
\begin{proof}
See Lemma 5.2.8 in \cite{KS:main}.
\end{proof}
\begin{definition}\label{defl=1}
Given a scalar function $f$   defined on a sphere $S$, we define the $\ell=1$ modes of $f$ by   the triplet
\beaa
(f)_{\ell=1} :=\left\{ \int_S f \Jp,\qquad p=0,+,-\right\}.
\eeaa
\end{definition}
\subsection{Elliptic estimates}
For a tensor $f$ on a sphere $S$, we define the following norms for any integer $k\geq 0$
\bea
\begin{split}
\Vert f \Vert_{\hk_k(S)}&:=&\sum_{j=0}^k\Vert\dkb^j f \Vert_{L^2(S)},\\
\Vert f \Vert_{\infty,k}&:=&\sum_{j=0}^k\Vert\dkb^j f \Vert_{L^\infty(S)},
\end{split}
\eea
Under the assumptions
\bea\label{gaggabepr}
\Vert \Ga_g \Vert_{\infty,k}\leq \ep r^{-2},\qquad \Vert \Ga_b \Vert_{\infty,k}\leq \ep r^{-1},
\eea
where $\ep>0$ is sufficiently small, we have the following coercive properties of the operators $\ddd_1$, $\ddd_2$, $\ddd^*_1$.
\begin{lemma}\label{elliptic1}
Consider a sphere $S$ such that its Gauss curvature $K$ satisfies
\begin{equation}\label{Gaussepsilon}
    \left\|K-\frac{1}{r^2}\right\|_{\hk_{s_{max}-1}(S)}\les\epg,
\end{equation}
Then, the following estimates hold for all integer $k\leq s_{max}$:
\begin{enumerate}
    \item If $f$ is a $1$--form
    \bea
    \Vert f\Vert_{\hk_{k+1}(S)}\les r\Vert\ddd_1 f\Vert_{\hk_k(S)}.
    \eea 
    \item If $v$ is a symmetric traceless $2$-tensor
    \bea 
    \Vert v\Vert_{\hk_{k+1}(S)}\les r\Vert \ddd_2 v\Vert_{\hk_k(S)}.
    \eea 
    \item If $(h,{^*h})$ is a pair of scalars
    \bea
    \Vert (h-\ov{h},{^*h}-\ov{^*h})\Vert_{\hk_{k+1}(S)}\les r\Vert\ddd_1^*(h,{^*h})\Vert_{\hk_k(S)}.
    \eea
\end{enumerate}
\end{lemma}
\begin{proof}
See Lemma 2.20 in \cite{KS:Kerr1}.
\end{proof}
\begin{lemma}\label{elliptic2}
On a fixed sphere $S$ for which \eqref{Gaussepsilon} holds, we have for any $1$--form $f$ and integer $0\leq k\leq s_{max}$:
\bea\label{fcontrol2.10}
\Vert f\Vert_{\hk_{k+1}(S)}&\les& r\Vert \ddd_2^* f\Vert_{\hk_k(S)}+|(f)_{\ell=1}|,
\eea
where the $\ell=1$ modes are given by Definition \ref{defl=1}.
\end{lemma}
\begin{proof}
See Lemma 2.19 in \cite{KS:Kerr1} for the case $k=0$. Assume that \eqref{fcontrol2.10} holds for $1\leq k\leq n-1$, we have from standard elliptic regularity that
\beaa
\Vert f\Vert_{\hk_{n+1}(S)} &\les & r^2\Vert \Delta f\Vert_{\hk_{n-1}(S)}\\
&\les & r^3\|\ddd_2^*\Delta f\|_{\hk_{n-2}(S)}+r^2|(\Delta f)_{\ell=1}|\\
&\les & r^3\|\Delta\ddd_2^*f\|_{\hk_{n-2}(S)}+r^3\|[\Delta,\ddd_2^*]f\|_{\hk_{n-2}(S)}+r^2\left|\int_S (\Delta f) J^{(p)}\right|\\
&\les & r\|\ddd_2^*f\|_{\hk_n(S)}+\|f\|_{\hk_n(S)}+r^2\left|\int_S f (\Delta J^{(p)})\right|\\
&\les & r\|\ddd_2^*f\|_{\hk_n(S)}+|(f)_{\ell=1}|,
\eeaa
where we used \eqref{fcontrol2.10} in the case of $k=n-1$ and \eqref{def:Jpsphericalharmonics} at the last step. By induction, we deduce \eqref{fcontrol2.10} for $k=n$. This concludes the proof of Lemma \ref{elliptic2}.
\end{proof}
The following corollary will be useful in Section \ref{sec4.2}.
\begin{corollary}\label{d2d2}
On a fixed sphere $S$ for which \eqref{Gaussepsilon} holds for $\epg$ small enough, we have for any symmetric traceless $2$--tensor $v$ and integer $0\leq k\leq s_{max}-1$:
\begin{align}
\Vert v\Vert_{\hk_{k+2}(S)} &\les r^2\Vert \ddd_2^*\ddd_2 v\Vert_{\hk_{k}(S)},\label{d2d2eq}\\
\Vert v\Vert_{\hk_{k+2}(S)}&\les r^2\left\Vert\left(\ddd_2^*\ddd_2+\frac{1}{r^2}\right) v\right\Vert_{\hk_{k}(S)}.\label{d2d2eqr2}
\end{align}
\end{corollary}
\begin{proof}
By Lemma \ref{elliptic2}, we have
\beaa 
\Vert \ddd_2 v \Vert_{\hk_{k+1}(S)} &\les& r\Vert \ddd_2^* \ddd_2 v\Vert_{\hk_k(S)}+|(\ddd_1\ddd_2 v)_{\ell=1}|\\
&\les&r\Vert \ddd_2^* \ddd_2 v\Vert_{\hk_k(S)}+\left|\int_S(\ddd_1\ddd_2 v)\Jp\right|\\
&\les&r\Vert \ddd_2^* \ddd_2 v\Vert_{\hk_k(S)}+\left|\int_S v(\ddd_2^*\ddd_1^*\Jp)\right|\\
&\les&r\Vert \ddd_2^* \ddd_2 v\Vert_{\hk_k(S)}+\epg\Vert v\Vert_{\infty,k},
\eeaa 
where we have used integration by parts and \eqref{estjpepg}.
Then, we apply Lemma \ref{elliptic1} to conclude
\beaa 
\Vert v\Vert_{\hk_{k+2}(S)} \les r \Vert \ddd_2v\Vert_{\hk_{k+1}(S)}\les r^2\Vert \ddd_2^*\ddd_2 v\Vert_{\hk_k(S)}+\epg r\Vert v\Vert_{\infty,k},
\eeaa
which implies, together with Sobolev and $\epg>0$ small enough,
\beaa 
\Vert v\Vert_{\hk_{k+2}(S)} &\les& r^2\Vert \ddd_2^*\ddd_2 v\Vert_{\hk_{k}(S)}.
\eeaa 
This concludes \eqref{d2d2eq}.

Since $\ddd_2^*\ddd_2$ is a positive operator on $L^2(S)$, the case $k=0$ of \eqref{d2d2eqr2} follows immediately from \eqref{d2d2eq}. Assuming that \eqref{d2d2eqr2} holds true for $0\leq k\leq n-1$, we have
\beaa
\|v\|_{\hk_{n+2}(S)}&\les& r^2\|\ddd_2^* \ddd_2 v\|_{\hk_{n}(S)}\\
&\les& r^2\left\|\left(\ddd_2^* \ddd_2+\frac{1}{r^2}\right) v\right\|_{\hk_n(S)}+ \|v\|_{\hk_n(S)}\\
&\les& r^2\left\|\left(\ddd_2^* \ddd_2+\frac{1}{r^2}\right) v\right\|_{\hk_n(S)},
\eeaa
where we used \eqref{d2d2eqr2} for $k=n-2$. By induction, we obtain \eqref{d2d2eqr2} for all $0\leq k\leq s_{max}-1$. This concludes the proof of Corollary \ref{d2d2}.
\end{proof}
\subsection{General frame transformations}
We recall below Lemma 3.1 in \cite{KS:Kerr1}.
\begin{lemma}
\label{Lemma:Generalframetransf}
Given a null frame $(e_3, e_4, e_1, e_2)$, a general null transformation from the null frame  $(e_3, e_4, e_1, e_2)$  to another null frame $(e_3', e_4', e_1', e_2')$ can be written in the form,
\bea
 \label{eq:Generalframetransf}
 \begin{split}
  e_4'&=\la\left(e_4 + f^b  e_b +\frac 1 4 |f|^2  e_3\right),\\
  e_a'&= \left(\de_{ab} +\frac{1}{2}\fb_af_b\right) e_b +\frac 1 2  \fb_a  e_4 +\left(\frac 1 2 f_a +\frac{1}{8}|f|^2\fb_a\right)   e_3,\qquad a=1,2,\\
 e_3'&=\la^{-1}\left( \left(1+\frac{1}{2}f\c\fb  +\frac{1}{16} |f|^2  |\fb|^2\right) e_3 + \left(\fb^b+\frac 1 4 |\fb|^2f^b\right) e_b  + \frac 1 4 |\fb|^2 e_4 \right),
 \end{split}
\eea
  where $\la$ is a scalar, $f$ and $\fb$ are horizontal 1-forms on the sphere generated by $(e_1',e_2')$. The dot product and magnitude  $|\c |$ are taken with respect to the standard Euclidean norm of $\RRR^2$.\footnote{Here $f$ and $\fb$ are $1$--forms defined on a sphere, we use an orthonomal basis of the sphere to define the dot product and magnitude w.r.t. the standard Euclidean norm of $\RRR^2$.} We call $(f, \fb, \la)$ the transition coefficients of the change of frame. We denote $$F:=(f, \fb, \ovla),\qquad \ovla:=\la-1.$$
  \end{lemma}
We recall some of the identities of Proposition 3.3 in \cite{KS:Kerr1}.
\begin{proposition}
\label{Prop:transformation-formulas-generalcasewithoutassumptions}
Under a general transformation of type \eqref{eq:Generalframetransf}, the Ricci coefficients  transform as follows:
\begin{itemize}
\item The transformation formula for $\xi$ is given by
\bea
\begin{split}
\la^{-2}\xi' &= \xi +\frac{1}{2}\nab_{\la^{-1}e_4'}f+\frac{1}{4}\trch f+\om f +\err(\xi,\xi'),\\
\err(\xi,\xi') &= \frac{1}{2}f\c\chih+\frac{1}{4}|f|^2\eta+\frac{1}{2}(f\c \ze)\,f -\frac{1}{4}|f|^2\etab \\
&+ \la^{-2}\left( \frac{1}{2}(f\c\xi')\,\fb+ \frac{1}{2}(f\c\fb)\,\xi'   \right)  +\lot
       \end{split}
\eea

\item The transformation formula for $\xib$ is given by
\bea
\begin{split}
\la^2\xib' &= \xib + \frac{1}{2}\la\nab_3'\fb +    \omb\,\fb + \frac{1}{4}\trchb\,\fb +\err(\xib, \xib'),\\
\err(\xib, \xib') &=   \frac{1}{2}\fb\c\chibh - \frac{1}{2}(\fb\c\ze)\fb +  \frac 1 4 |\fb|^2\etab  -\frac 1 4 |\fb|^2\eta'+\lot
\end{split}
\eea
\item The transformation formula for $\trch$ is given by
\bea 
\begin{split}\label{transtrch}
    \la^{-1}\trch'&=\trch+\div'f+f\c\eta+f\c\ze+\err(\trch',\trch),\\
    \err(\trch',\trch)&=\fb\c\xi+\frac{1}{4}\fb\c f\trch+\om(f\c\fb)-\omb|f|^2\\
    &-\frac{1}{4}|f|^2\trchb-\frac{1}{4}(f\c\fb)\la^{-1}\trch'+\lot
\end{split}
\eea 
\item   The transformation formula for $\ze$ is given by
\bea
\begin{split}
\ze' &=\ze-\nab'(\log\la)-\frac{1}{4}\trchb f+\om\fb-\omb f+\frac{1}{4}\fb\trch+\err(\ze,\ze'),\\
\err(\ze, \ze') &=-\f12\hchb\c f+\f12(f\c\ze)\fb-\f12(f\c\etab)\fb+\frac{1}{4}\fb(f\c\eta)+\frac{1}{4}\fb(f\c\ze)\\
&+\frac{1}{4}{^*\fb}(f\wedge\eta)+\frac{1}{4}{^*\fb}(f\wedge\ze)+\frac{1}{4}\fb\div'f+\frac{1}{4}{^*\fb}\curl'f+\f12\la^{-1}\fb\c\hch'\\
&-\frac{1}{16}(f\c\fb)\fb\la^{-1}\trch'+\frac{1}{16}{^*\fb}\la^{-1}(f\wedge\fb)\trch'+\lot
\end{split}
\eea
\item The transformation formula for $\eta$ is given by
\bea
\begin{split}
\eta' &= \eta +\frac{1}{2}\la \nab_3'f+\frac{1}{4}\fb\trch-\omb\, f +\err(\eta, \eta'),\\
\err(\eta, \eta') &= \frac{1}{2}(f\c\fb)\eta +\frac{1}{2}\fb\c\chih
+\frac{1}{2}f(\fb\c\ze)  -  (\fb\c f)\eta'+ \frac{1}{2}\fb (f\c\eta') +\lot
\end{split}
\eea
\item The transformation formula for $\etab$ is given by
\bea
\begin{split}
\etab' &= \etab +\frac{1}{2}\nab_{\la^{-1}e_4'}\fb +\frac{1}{4}\trchb f -\om\fb +\err(\etab, \etab'),\\
\err(\etab, \etab') &=  \frac{1}{2}f\c\chibh + \frac{1}{2}(f\c\eta)\fb-\frac 1 4  (f\c\ze)\fb  -\frac 1 4 |\fb|^2\la^{-2}\xi'+\lot
\end{split}
\eea

\item   The transformation formula for $\om$ is given by
\bea
\begin{split}
\la^{-1}\om' &=  \om -\frac{1}{2}\la^{-1}e_4'(\log\la)+\frac{1}{2}f\c(\ze-\etab) +\err(\om, \om'),\\
\err(\om, \om') &=   -\frac{1}{4}|f|^2\omb - \frac{1}{8}\trchb |f|^2+\frac{1}{2}\la^{-2}\fb\c\xi' +\lot
\end{split}
\eea
\item  The transformation formula for $\omb$ is given by
\bea
\begin{split}
\la\omb' &= \omb+\frac{1}{2}\la e_3'(\log\la)  -\frac{1}{2}\fb\c\ze -\frac{1}{2}\fb\c\eta +\err(\omb,\omb'),\\
\err(\omb,\omb') &= f\c\fb\,\omb-\frac{1}{4} |\fb|^2\om  +\frac{1}{2}f\c\xib + \frac{1}{8}(f\c\fb)\trchb-\frac{1}{8}|\fb|^2\trch \\
&-\frac{1}{4}\la\fb\c\nab_3'f+\frac{1}{2}(\fb\c f)(\fb\c\eta')-\frac{1}{4}|\fb|^2 (f\c\eta')+\lot
\end{split}
\eea
\end{itemize}
where, for the transformation formulas of the Ricci coefficients above, $\lot$ denote expressions of the type
\beaa
         \lot&=O((f,\fb)^3)\Ga +O((f,\fb)^2) \Gac
\eeaa
containing no derivatives of $f$, $\fb$, $\Ga$ and $\Gac$.
\end{proposition}
\begin{proof}
See Appendix A of \cite{KS:Kerr1}.
\end{proof}
\subsection{Background spacetime}
\label{sec:backgroundspacetime}
\subsubsection{Adapted coordinates}\label{adaptedcoordinates}
Recall that we consider given a vacuum spacetime $\RR$ and endowed with of $(u,s)$ foliation, see Section \ref{gardef}. A coordinate system $(u, s, y^1, y^2)$ on $\RR$ is said to be adapted to an outgoing geodesic foliation as above if $e_4(y^1)= e_4(y^2)=0$.  In that case the spacetime metric can be written in the form, see Lemma 2.6 in \cite{KS:Kerr1},
\bea
\lab{spacetimemetric-y-coordinates}
\g &=& -2\vsi duds+\vsi^2\Omb du^2+g_{ab}\big(dy^a-\vsi\undB^a du\big)\big( dy^b-\vsi\undB^b du\big),
\eea
where
\bea
\Omb=e_3(s), \qquad \undB^a =\frac{1}{2} e_3(y^a), \qquad  g_{ab}=\g(\pr_{y^a}, \pr_{y^b}).
\eea
Relative to these coordinates
\bea\label{eq:decompositionofnullframeoncoordinatesframeforbackgroundfoliation}
\pr_s=e_4, \qquad \pr_u = \vsi\left(\frac{1}{2}e_3-\frac{1}{2}\Omb e_4-\undB^a\pr_{y^a}\right), \qquad \pr_{y^a}= \sum_{c=1,2} Y_{(a)}^c e_c, \quad   a=1,2,
\eea
with coefficients  $ Y_{(a)}^b $ verifying
\beaa
g_{ab}=\sum_{c=1, 2} Y_{(a)}^c Y_{(b)}^c.
\eeaa
As a direct consequence of \eqref{eq:decompositionofnullframeoncoordinatesframeforbackgroundfoliation}, we have\footnote{Recall that $z:=\frac{2}{\vsi}$ is defined in \eqref{defz}.}
\bea\label{suc}
e_4=\pr_s,\qquad e_3=z\pr_u+\Omb\pr_s+2\Bb^a\pr_{y^a},\qquad e_c=X_{(c)}^a\pr_{y^a},
\eea 
where $X_{(a)}^c$ is defined by
\beaa 
\sum_{c=1,2} X_{(c)}^aY_{(a)}^c=1.
\eeaa 
As in \cite{KS:Kerr1},  we assume  that $\RR $ is covered by two coordinate systems, i.e. $\RR=\RR_N\cup \RR_S$,
such that:
\begin{enumerate}
\item The North coordinate chart $\RR_N$ is given by the coordinates
$(u, s, y_{N}^1, y_{N}^2)$ with $(y^1_{N})^2+(y^2_{N})^2<2$ while the South coordinate chart $\RR_S$ is given by the coordinates
$(u, s, y_{S}^1, y_{S}^2)$ with $(y^1_{S})^2+(y^2_{S})^2<2$.
\item The two coordinate charts intersect in the open equatorial region
$\RR_{Eq}:=\RR_N\cap \RR_S$ in which both coordinate systems are defined.
 \item  In $\RR_{Eq} $   the transition functions  between the two coordinate  systems are given by  the smooth  functions $ \varphi_{SN}$ and $\varphi_{NS}= \varphi_{SN}^{-1} $.
 \end{enumerate}
The metric coefficients for the two coordinate systems are given by
 \beaa
\g &=& - 2\vsi du ds + \vsi^2\Omb  du^2 +g^{N}_{ab}\big( dy_N^a- \vsi \undB_{N}^a du\big) \big( dy_N^b-\vsi \undB_N^b du\big),\\
\g &=& - 2\vsi du ds + \vsi^2\Omb  du^2 +g^{S}_{ab}\big( dy_S^a- \vsi \undB_{S}^a du\big) \big( dy_S^b-\vsi \undB_S^b du\big),
\eeaa
where
\beaa
\Omb=e_3(s), \qquad \undB_N^a =\frac{1}{2} e_3(y_N^a), \qquad \undB_S^a =\frac{1}{2} e_3(y_S^a).
\eeaa
For a $S(u,s)$--tangent tensor $f$, we consider the following norms
  \bea
  \label{Norms-spacetimefoliation-GSMS}
  \begin{split}
  \| f\|_{\infty} :&=\| f\|_{L^\infty(S)}, \qquad  \| f\|_{2} :=\| f\|_{L^2(S)}, \\
  \|f\|_{\infty,k} &= \sum_{i=0}^k \|\dk^i f\|_{\infty },  \qquad
\|f\|_{2,k}=\sum_{i=0}^k \|\dk^i f\|_{2},
\end{split}
  \eea
  where $\dk^i$ stands for any   combination  of length $i$ of operators  of the form
   $e_3, r e_4, r\nab $.\\ \\
Finally, we recall the following lemma for geodesic foliations.
\begin{lemma}\label{nonpo}
For the background geodesic foliation of $\RR$ and any scalar function $h$ defined on $\RR$, we have
\bea
\begin{split}
e_4\left(\int_S h \right)&=\int_S \left( e_4 (h)+ \ka h\right),\\
e_3\left(\int_S h \right)&=z\int_S\left( z^{-1}e_3(h)-z^{-1}\Omb e_4(h)+ z^{-1}\kab h-z^{-1}\Omb\ka h\right)+\Omb\int_S \left(e_4(h)+\ka h\right).
\end{split}
\eea
In particular, we have for the area radius $r$:
\bea\label{nonpoeq}
e_4(r)=\frac{r}{2}\ov{\ka},\qquad e_3(r)=\frac{r}{2}\left(z\ov{z^{-1}\kab}-z\ov{z^{-1}\Omb\ka}+\Omb\ov{\ka}\right).
\eea
\end{lemma}
\begin{proof}
Firstly, recall from \eqref{eq:decompositionofnullframeoncoordinatesframeforbackgroundfoliation} that we have
\bea\label{e4|S|}
\begin{split}
e _4\left(\int_S h\right)=\pr_s\left(\int_S h\right)=\int_S \pr_s h+\g^{ab}\g(\D_a\pr_s,\pr_b)h=\int_S \left(e_4(h)+\ka h\right)
\end{split}
\eea
as stated. Next, recall from \eqref{eq:decompositionofnullframeoncoordinatesframeforbackgroundfoliation} that
\beaa 
\pr_u\left(\int_S h\right)&=&\int_S \pr_u h+ \g^{ab}\g(\D_a \pr_u ,\pr_b)h\\
&=&\int_S z^{-1}e_3(h)-z^{-1}\Omb e_4(h)-\frac{2}{z}\Bb^c\pr_{c}(h)+\g^{ab}\g(\D_a (z^{-1}e_3),\pr_b)h\\
&-&\int_S \g^{ab}\g(\D_a (z^{-1}\Omb e_4),\pr_b)h+\g^{ab}\g\left(\D_a\left(\frac{2}{z}\Bb^c\pr_c\right),\pr_b\right)h\\
&=&\int_S z^{-1}e_3(h)-z^{-1}\Omb e_4(h)+ z^{-1}\kab h-z^{-1}\Omb\ka h \\
&-&\int_S \frac{2}{z}\Bb^c\g^{ab}\g(\D_a\pr_c,\pr_b)h+\nab_c\left(\frac{2}{z}\Bb^c h\right).
\eeaa
Recall that the divergence theorem implies that
\beaa 
\int_S \nab_c\left(\frac{2}{z}\Bb^c h\right)=0.
\eeaa
On the other hand
\beaa
\g^{ab}\g(\D_a\pr_c,\pr_b)&=&\g^{ab}\g(\D_c\pr_a,\pr_b)=\f12\g^{ab}\g(\D_c\pr_a,\pr_b)+\f12\g^{ab}\g(\D_c\pr_b,\pr_a)\\
&=&\f12\g^{ab}\D_c(\g_{ab})=\f12\D_c(\g^{ab}\g_{ab})=0.
\eeaa 
Thus, we obtain
\beaa 
\pr_u\left(\int_S h\right)&=&\int_S z^{-1}e_3(h)-z^{-1}\Omb e_4(h)+ z^{-1}\kab h-z^{-1}\Omb\ka h.
\eeaa 
Together with \eqref{eq:decompositionofnullframeoncoordinatesframeforbackgroundfoliation} and \eqref{e4|S|}, we infer
\beaa
e_3\left(\int_S h\right)&=&z\pr_u\left(\int_S h\right)+\Omb e_4\left(\int_S h\right)\\
&=&z\int_S\left( z^{-1}e_3(h)-z^{-1}\Omb e_4(h)+ z^{-1}\kab h-z^{-1}\Omb\ka h\right)+\Omb\int_S \left(e_4(h)+\ka h\right)
\eeaa 
as stated. Taking $h=1$, we obtain \eqref{nonpoeq}. This concludes the proof of Lemma \ref{nonpo}.
\end{proof}
\subsubsection{Background spacetime region \texorpdfstring{$\RR$}{}}\label{subsubsect:regionRR1}
In the following definition, we specify the background spacetime region $\RR$.
 \begin{definition}
\label{defintion:regionRRovr}
Let $m_0>0$ a constant. Let $\epg >0$ a sufficiently small constant, and let  $(\ug, \sg, \rg)$ three real numbers with $\rg$ sufficiently large so that
\bea\lab{eq:rangeofrgandepsilon}
\epg\ll m_0, \qquad\qquad\rg\gg m_0.
\eea
We define $\RR$ to be the region
\bea
\lab{definition:RR(dg,epg)}
\RR:=\left\{|u-\ug|\leq\epg,\quad |s-\sg|\leq  \epg \right\},
\eea
such that  assumptions {\bf A1-A4} below  with constant $\epg$  on  the background foliation of $\RR$,   are verified.
\end{definition}
\subsubsection{Main assumptions for \texorpdfstring{$\RR$}{}}
\label{subsubsect:regionRR2}
Given an integer $s_{max}\geq 5$, we assume\footnote{In  view of \eqref{eq:assumtioninRRforGagandGabofbackgroundfoliation}, we will often replace $\Ga_g$ by $r^{-1} \Ga_b$.} the following:
\begin{enumerate}
\item[\bf A1.]
For  $k\le s_{max}$
\bea\lab{eq:assumtioninRRforGagandGabofbackgroundfoliation}
\begin{split}
\| \Ga_g\|_{k, \infty}&\leq  \epg  r^{-2},\\
\| \Ga_b\|_{k, \infty}&\leq  \epg  r^{-1}.
\end{split}
\eea
\item[\bf A2.]  The Hawking mass $m=m(u,s)$ of  $S(u, s)$ verifies
\bea\lab{eq:assumtionsonthegivenusfoliationforGCMprocedure:Hawkingmass}
\sup_{\RR}\left|\frac{m}{m_0}-1\right| &\leq& \epg.
\eea
\item[\bf A3.]
In the  region of their respective validity\footnote{That is  the quantities on the left verify the  same estimates as those for $\Ga_b$, respectively $\Ga_g$.}   we have
\bea
 \undB_N^a,\,\, \undB_S^a \in r^{-1}\Ga_b, \qquad  Z_N^a,\,\, Z_S^a \in \Ga_b,\qquad r^{-2} \widecheck{g}^{N}_{ab},  \,\, r^{-2} \widecheck{g}^{S}_{ab} \in r\Ga_g,
 \eea
 where
 \beaa
 \widecheck{g} ^{N}\!_{ab} &=& g^N_{ab}-\frac{4r^2}{(1+(y^1_{N})^2+(y^2_{N})^2) }\de_{ab},\\
 \widecheck{g}^{S}\!_{ab} &=& g^S_{ab}-\frac{4r^2}{(1+(y^1_{S})^2+(y^2_{S})^2) } \de_{ab},\\
  Z^c &=& \Bb^aY^c_{(a)}.
 \eeaa
\item[\bf  A4.] We assume  the existence of a smooth family of  scalar functions $\Jp:\RR\to\RRR$, for $p=0,+,-$,   verifying the following properties
\begin{enumerate}
\item On the sphere $\ovS$ of the background foliation, there holds \eqref{def:Jpsphericalharmonics} with $\ep=\epg$, i.e.
 \begin{align}\label{eq:Jpsphericalharmonics}
\begin{split}
  \Big((\rg)^2\lapzero+2\Big) \Jp  &= O(\epg),\qquad p=0,+,-,\\
\frac{1}{|\ovS|} \int_{\ovS}  \Jp J^{(q)} &=  \frac{1}{3}\de_{pq} +O(\epg),\qquad p,q=0,+,-,\\
\frac{1}{|\ovS|}  \int_{\ovS}\Jp &=O(\epg),\qquad p=0,+,-.
\end{split}
\end{align}
\item We extend $\Jp$ from $\ovS$ to $\RR$ by $\pr_s\Jp=\pr_u\Jp=0$, i.e.
\bea\label{eq:extensionofJpfromovStoRR}
\Jp(u,s,y^1,y^2)=\Jp(\ug, \sg, y^1, y^2).
\eea
\end{enumerate}
\end{enumerate}
\begin{remark}
We note that the assumptions {\bf A1}, {\bf A2}, {\bf A3}, {\bf A4}, are expected to be valid in the far regions, i.e. $r$ large, of a perturbed Kerr. In particular, they hold in far regions of Kerr, see Lemma 2.10 in \cite{KS:Kerr1}.
\end{remark}
\subsection{Deformation of surfaces in \texorpdfstring{$\RR$}{}}
\label{sec:defofmationofsurfacesinRR}
 \begin{definition}
 \label{definition:Deformations}
 We say that    $\S$ is a deformation of $ \ovS$ if there exist  smooth  scalar functions $U, S$ defined on $\ovS$ and a map $\Psi:\ovS\to\S$ verifying, on any coordinate chart $(y^1, y^2) $ of $\ovS$,
   \bea
 \Psi(\ovu, \ovs,  y^1, y^2)=\left( \ovu+ U(y^1, y^2 ), \, \ovs+S(y^1, y^2 ), y^1, y^2  \right).
 \eea
 \end{definition}

\begin{definition}
Given a deformation $\Psi:\ovS\to \S$ we say that a new frame $(e_3^\S, e_4^\S, e_1^\S, e_2^\S)$ on $\S$, obtained from the standard frame $(e_3, e_4, e_1, e_2)$  via the transformation \eqref{eq:Generalframetransf}, is $\S$-adapted if the vectorfields $e_1^\S$, $e_2^\S$ are tangent to $\S$ and the vectorfields $e_3^\S$, $e_4^\S$ are orthogonal to $\S$.
\end{definition}

\begin{definition}
Let $\S\subset\RR$ be a compact $2$-sphere, which is a deformation of a leaf $S(u,s)$ of the background geodesic foliation of $\RR$ and let  $(e_3^\S,e_4^\S,e_1^\S,e_2^\S)$ the null frame adapted to $\S$. Then, we denote
\begin{itemize}
\item by $\chi^\S$, $\chib^\S$, $\ze^\S$,..., the corresponding Ricci coefficients,
\item by $\a^\S$, $\b^\S$, $\rho^\S$, ..., the corresponding curvature coefficients,
\item by $r^\S$, $m^\S$, $K^\S$ and $\mu^\S$ respectively the corresponding area radius, Hawking mass, Gauss curvature  and mass aspect function,
\item by $\nab^\S$ the corresponding covariant derivative.
\end{itemize}
\end{definition}
\begin{definition}
We will work with the following weighted Sobolev norms on $\S$
\bea
   \label{definition:spaceH^k(boldS)}
\| f\|_{\hk_s(\S)} &:=& \sum_{i=0}^s \|( \dkb^\S )^i f\|_{L^2(\S)}, \qquad   \dkb^\S =r^\S \nab^\S.
\eea
\end{definition}
We will need the following lemmas.
\begin{lemma}\label{lemma:comparison-gaS-ga}
 Let $\ovS \subset \RR$.    Let  $\Psi:\ovS\to \S $  be   a  deformation generated by the  functions $(U, S)$ as in Definition \ref{definition:Deformations} and denote by $g^{\S,\#}$  the  pull back of the metric $g^\S$ to $\ovS$. Assume the bound
 \bea
 \label{assumption-UV-dg}
   \| (U, S)\|_{L^\infty(\ovS)} +r ^{-1} \big\|(U, S)\big\|_{\hk_{s_{max}+1}(\ovS)}   &\les&  \dg.
 \eea
  Then
  \begin{enumerate}
 \item We have
  \bea\label{eq:compairisionofpulledbackmetricsfordeformations}
  \big\|  g^{\S, \#} -\ovg\big\|_{L^\infty} +r^{-1} \big\|  g^{\S, \#} -\ovg\big\|_{\hk_{s_{max}}(\ovS)}\les \dg  r.
  \eea

\item  For any  tensor $h$  on $\RR$
 \bea
 \label{eq:Prop:comparison2}
 \|h\|_{\hk_s(\S)} \les  r \sup_{\RR}\Big(|\dkb^{\leq s}h|+\dg|\dk^{\leq s}h|\Big), \qquad 0\leq s \leq s_{max}.
 \eea
\item If $V\in \hk_s(\S)$ and $V^\#$ is its pull-back by $\Psi$, we have for all $0\leq s\leq s_{max}$,
 \bea
 \label{eq:Prop:comparison1}
 \|V\|_{\hk_s(\S)}= \|V^\#\|_{\hk_s(\ovS,\, g^{\S,\#})} = \| V^\#\|_{\hk_s(\ovS, \ovg)}\big(1+O(r^{-1} \dg)\big).
 \eea
 \item As a corollary of \eqref{eq:Prop:comparison1} (choosing $V=1$ and $s=0$), we deduce
\bea
 \frac{r^\S}{\ovr}= 1 + O(r ^{-1}  \dg )
 \eea
 where $r^\S$ is the area radius of $\S$ and $\ovr$ that of $\ovS$.
 \item We also have
 \bea\label{2.50}
 |m-\mS|=O(\dg).
 \eea
 \end{enumerate}
 \end{lemma}
\begin{proof}
See Lemma 5.8, Proposition 5.10 and Corollary 5.17 in \cite{KS:Kerr1}.
\end{proof}
We also have the following lemma.
\begin{lemma}
\label{Lemma:coparison-forintegrals}
Under the same assumptions as in Lemma \ref{lemma:comparison-gaS-ga},  the following estimate holds  for   a scalar function $F$ defined on $\RR$,
\beaa
\left|\int_\S F -\int_{\ovS} F\right| &\les& \dg r \,\left(    \sup_{\RR}|F|+r\sup_{\RR}\big(|\pr_uF|+|\pr_sF|\big)\right).
\eeaa
\end{lemma}
\begin{proof}
See Corollary 5.9 in \cite{KS:Kerr1}.
\end{proof}
We introduce the following schematic presentation of the error terms which appear in various calculations below.
\begin{definition}
\label{Definition:errorterms-prime}
We denote by $\err_k$, $k=1,2$, error terms\footnote{Note  however that  the precise error terms differ in each particular case and that we only emphasize here their general structure.} which can be written schematically in the form,
\bea
\begin{split}
r\err_1&= F\c  (r\Ga_b)+   F^2   + F\c  (r \nabS) F= F\c  (r\Ga_b)+ F\c  (r \nabS)^{\le 1} F,\\
r^2 \err_2 &= ( r\nabS)^{\le 1}( r\err_1)+F\c  r \dk \Ga_b,
\end{split}
\eea
where
\beaa
F&:=& (f,\fb,\ovla),\qquad \ovla=\la-1.
\eeaa
\end{definition}
We recall the following identities, see Corollary 4.6 in \cite{KS:Kerr1}.
\begin{proposition}\label{GCMsystemequation}
The following relations hold true for any adapted frame $(e_3^\S,e_4^\S,e_1^\S,e_2^\S)$ to a given sphere $\S$ connected to the reference frame $(e_3,e_4,e_1,e_2)$ by the transition coefficients $(f,\fb,\la)$,
\bea
\begin{split}\label{Generalizedsystem}
\curl ^\S f  &= -\err_1[\curl^\S  f ],\\
\curl^\S \fb  &= -\err_1[\curl^\S  \fb ],\\
\div^\S f + \ka \ovla -\frac{2}{(r^\S)^2}\ovb &= \ka^\S-\frac{2}{r^\S} -\left(\ka-\frac{2}{r}\right) -\err_1[\div^\S f ] -\frac{2(r-r^\S)^2}{r(r^\S)^2},\\
\div^\S\fb - \kab \ovla +\frac{2}{(r^\S)^2}\ovb &= \kab^\S+\frac{2}{r^\S} -\left(\kab+\frac{2}{r}\right) -\err_1[\div^\S \fb ] +\frac{2(r-r^\S)^2}{r(r^\S)^2},\\
\Delta^\S\ovla + V\ovla &=\mu^\S-\mu -\left(\omb +\frac 1 4 \kab \right) \big(\ka^\S-\ka \big)+\left(\om +\frac 1 4 \ka \right) \big(\kab^\S-\kab \big)+\err_2[ \lap^\S\ovla],\\
\Delta^\S\ovb &= \frac{1}{2}\div^\S\left(\fb - \Up f +\err_1[\Delta^\S \ovb ] \right),
\end{split}
\eea
where
\beaa
\ovb&:=& r-\rS,\\
V&:=&-\left(\f12\ka\kab+\ka\omb+\kab\om\right),
\eeaa
and the error terms $\err_1[\curl^\S f]$, $\err_1[\curl^\S \fb]$, $\err_1[\sdiv^\S f]$, $\err_1[\sdiv^\S \fb]$ and $\err_1[\Delta^\S \ovb]$ are consistent with $\err_1$ in Definition \ref{Definition:errorterms-prime} whereas the error term $\err_2[\Delta^\S\ovla]$ is consistent with $\err_2$ in Definition \ref{Definition:errorterms-prime}.
\end{proposition}
\begin{proof}
    The first two identities in \eqref{Generalizedsystem} follow from the fact that ${^{(a)}\ka} :=\in^{ab} \chi_{ab}=0$ and ${^{(a)}\kab}:=\in^{ab}\chib_{ab}=0$. The third, fourth and fifth identities in \eqref{Generalizedsystem} follow from the transformation formulae for $\ka,\kab$ and $\ze$. See Lemma 4.3 in \cite{KS:Kerr1} for a detailed proof. The last identity in \eqref{Generalizedsystem} is proven in Corollary 4.6 in \cite{KS:Kerr1}.
\end{proof}
\section{GCM spheres}\label{GCMz}
\subsection{GCM spheres with \texorpdfstring{$\ell=1$}{} modes in \texorpdfstring{\cite{KS:Kerr1}}{}}
We review below Theorem 6.1 of \cite{KS:Kerr1} on existence  and uniqueness of GCM spheres in the context of  an arbitrary choice of a $\ell=1$ modes on $\ovS$, denoted $\Jp$,  which verify the assumptions {\bf A4} given in Section \ref{subsubsect:regionRR2}.
\begin{theorem}[GCM spheres with $\ell=1$ modes in \cite{KS:Kerr1}]
\label{Theorem:ExistenceGCMS1}
Let $m_0>0$ a constant.   Let $0<\dg\leq \epg $   two sufficiently   small   constants, and let  $(\ug, \sg, \rg)$ three real numbers with $\rg$ sufficiently large so that
\beaa
\epg\ll m_0, \qquad\qquad  \rg\gg m_0.
\eeaa
Let a fixed  spacetime region $\RR$, as in Definition \ref{defintion:regionRRovr}, together with a $(u, s)$ outgoing geodesic   foliation verifying the  assumptions {\bf A1-A4}.
 Let  $\ovS=S(\ovu, \ovs)$   be  a fixed    sphere  from this foliation  with  $\rg$ and $\mg$ denoting  its area radius and  Hawking mass.
 Assume that  the GCM quantities   $\ka, \kab, \mu$  of the background foliation verify the  following:
\begin{align}\label{eq:thedecompositionofkakabandmuintermsofkadotkabdotandmudotplusmodes}
\begin{split}
\ka&=\frac{2}{r}+\dot{\ka},\\
\kab&=-\frac{2\Up}{r} +  \Cb_0+\sum_p \Cbp \Jp+\dot{\kab},\\
\mu&= \frac{2m}{r^3} + M_0+\sum _p\Mp \Jp+\dot{\mu},
\end{split}
\end{align}
where the scalar functions $\Cb_0=\Cb_0(u,s)$, $\Cbp=\Cbp(u,s)$, $M_0=M_0(u,s)$ and $\Mp=\Mp(u,s)$, defined on the spacetime region $\RR$, depend only on the coordinates $(u,s)$, and where $\dot{\ka}$, $\dot{\kab}$ and $\dot{\mu}$ satisfy the following estimates
\bea\label{oldeq3.2}
\sup_{\RR}\big|\dkb^{\leq s_{max}}(\kadot, \kabdot)|\les r^{-2}\dg,\qquad\qquad
\sup_{\RR}\big|\dkb^{\leq s_{max}}\mudot| \les r^{-3}\dg,
\eea
where $\dkb=r\nab$. Then for any fixed triplets   $\La, \Lab \in \RRR^3$  verifying
\bea\label{eq:assumptionsonLambdaabdLambdabforGCMexistence}
|\La|,\,  |\Lab|  &\les & \dg,
\eea
there exists a unique  GCM sphere $\S=\S(\La, \Lab)$, which is a deformation of $\ovS$,
such that  there exist constants $\Cb^\S_0$, $\CbpS$, $ M^\S_0$, $\MpS$, $p\in\{-,0, +\}$  for which the following GCM conditions are verified
 \begin{align}
\label{def:GCMC}
\begin{split}
\ka^\S&=\frac{2}{r^\S},\\
\kab^\S &=-\frac{2}{r^\S}\Up^\S+  \Cb^\S_0+\sum_p \CbpS \Jp,\\
\mu^\S&= \frac{2m^\S}{(r^\S)^3} +   M^\S_0+\sum _p\MpS \Jp.
\end{split}
\end{align}
Relative to these modes we also have
\bea
\label{GCMS:l=1modesforffb}
(\div^\S f)_{\ell=1}=\La, \qquad   (\div^\S\fb)_{\ell=1}=\Lab.
\eea
The resulting  deformation has the following additional properties:
\begin{enumerate}
\item The triplet $(f,\fb,\ovla)$ verifies
\bea
\label{eq:ThmGCMS1}
\|(f,\fb, \ovla)\|_{\hk_{s_{max}+1(\S)}} &\les &\dg.
\eea
\item The GCM constants  $\Cb^\S_0$, $\CbpS$, $M^\S_0$, $\MpS$, $p\in\{-,0, +\}$ verify
\begin{align}
\label{eq:ThmGCMS2}
\begin{split}
\big| \Cb^\S_0-\ov{\Cb_0}^\S\big|+\big| \CbpS-\ov{\Cbp}^\S\big|&\les r^{-2}\dg,\\
\big| M^\S_0-\ov{M_0}^\S\big|+\big| \MpS-\ov{\Mp}^\S\big|&\les r^{-3}\dg.
\end{split}
\end{align}
\item The volume radius $r^\S$  verifies
\bea
\lab{eq:ThmGCMS3}
\left|\frac{r^\S}{\rg}-1\right|\les  r^{-1} \dg.
\eea
\item  The parameter  functions $U, S$  of the deformation verify
\bea
\label{eq:ThmGCMS4}
 \|( U, S)\|_{\hk_{s_{max}+1}(\ovS)}  &\les& r\dg.
 \eea
\item The Hawking mass  $m^\S$  of $\S$ verifies the estimate
\bea
\label{eq:ThmGCMS5}
 \big|m^\S-\ovm\big|&\les &\dg.
 \eea
 \item The well defined\footnote{Note  that  while  the Ricci coefficients $\ka^\S, \kab^\S,  \chih^\S, \chibh^\S, \ze^\S$ as well as all curvature  components  and   mass aspect function $\mu^\S$     are well defined on $\S$, this in not the case  of $\eta^\S, \etab^\S, \xi^\S, \xib^\S, \om^\S, \omb^\S$ which require  the  derivatives of the frame in the $e_3^\S$ and $e_4^\S$ directions.}
  Ricci and curvature coefficients of $\S$  verify,
\begin{align}\label{eq:ThmGCMS6}
\begin{split}
\| \Ga^\S_g\|_{\hk_{s_{max} }(\S) }&\les  \epg  r^{-1},\\
\| \Ga^\S_b\|_{\hk_{s_{max} }(\S) }&\les  \epg.
\end{split}
\end{align}
\item  The transition parameters $f, \fb, \ovla$ are continuously differentiable  with respect to $\La, \Lab $ and
\begin{align}\label{eq:property7oftheoldGCMtheoremnoncanonicalell=1modes}
\begin{split}
\frac{\pr f }{\pr \La}&=O\big( r^{-1}\big), \quad\;\,  \frac{ \pr f }{\pr \Lab}=O\big(\dg r^{-1} \big), \quad\;\,
\frac{\pr\fb }{\pr \La}=O\big(\dg  r^{-1}\big), \quad\;\,  \frac{\pr\fb}{\pr \Lab}=O\big( r^{-1} \big),\\
\frac{\pr\ovla  }{\pr \La}&=O\big(\dg  r^{-1}\big), \quad \frac{\pr\ovla  }{\pr \Lab}=O\big(\dg  r^{-1}\big).
\end{split}
\end{align}
\item The parameter functions $(U, S)$ of the deformation are continuously differentiable  with respect to $\La, \Lab$
 and
 \bea\label{USdepend}
 \frac{\pr U}{\pr \La}=O(1), \quad \frac{\pr U}{\pr \Lab}=O(1),\quad  \frac{\pr S}{\pr \La}=O(1) , \quad  \frac{\pr S}{\pr \Lab}=O(1).
 \eea
\item Relative to the  coordinate system induced by $\Psi$ the metric $g^{\La, \Lab}$  of $\S=\S(\La, \Lab)$  is continuous with respect to the parameters $\La, \Lab$ and verifies
\bea\label{eq:property9oftheoldGCMtheoremnoncanonicalell=1modes}
\big\| \pr_\La g^\S, \,\pr_{\Lab} g^\S\|_{L^\infty(\S)} &\les O( r^2).
\eea
\end{enumerate}
\end{theorem}
\begin{remark}
The conclusions of Theorem \ref{Theorem:ExistenceGCMS1} still hold if we replace \eqref{oldeq3.2} with the weaker condition\footnote{Note that \eqref{oldeq3.2} implies \eqref{eq:GCM-improved estimate2-again:butwithhksnorminstead} in view of \eqref{eq:Prop:comparison2}.}
\bea
\lab{eq:GCM-improved estimate2-again:butwithhksnorminstead}
\big\|(\kadot, \kabdot)\|_{\hk_{s_{max}}(\S)}\les r^{-1}\dg,\qquad\qquad
\big\|\mudot\|_{\hk_{s_{max}}(\S)} \les r^{-2}\dg,
\eea
for any deformed sphere $\S$ with $(U,S)$ satisfying \eqref{eq:ThmGCMS4}, where \eqref{eq:GCM-improved estimate2-again:butwithhksnorminstead} is uniform w.r.t such spheres. See Remark 6.2 in \cite{KS:Kerr1}.
\end{remark}
\subsection{Linearized GCM equations}
\begin{definition}
\label{definition:GCMSgen-equations}
Let   $\S\subset\RR$  a  smooth  $O(\epg)$-sphere. We say that $F=(f, \fb, \ovla)$  verifies the inhomogeneous linearized GCM system on $\S$ if the following holds true:
 \bea
 \label{GeneralizedGCMsystem}
\begin{split}
\curl ^\S f &= h_1 -\ov{h_1}^\S,\\
\curl^\S \fb&= \underline{h}_1 - \ov{\underline{h}_1}^\S,\\
\div^\S f + \frac{2}{r^\S} \ovla  -\frac{2}{(r^\S)^2}\ovb &=  h_2,
\\
\div^\S\fb + \frac{2}{r^\S} \ovla +\frac{2}{(r^\S)^2}\ovb
&=   \Cbdot_0+\sum_p \Cbpdot \JpS +\underline{h}_2,\\
\left(\Delta^\S+\frac{2}{(r^\S)^2}\right)\ovla  &=  \Mdot_0+\sum _p\Mpdot \JpS+\frac{1}{2r^\S}\left(\Cbdot_0+\sum_p \Cbpdot \JpS\right) +h_3,\\
\Delta^\S\ovb-\frac{1}{2}\div^\S\Big(\fb - f\Big) &= h_4 -\ov{h_4}^\S , \qquad \ov{\ovb}^\S=b_0,
\end{split}
\eea
 for  some choice of  constants $\Cbdot_0, \Mdot_0, \Cbpdot, \Mpdot$, $b_0$, and scalar functions $h_1$, $h_2$, $h_3$, $h_4$, $\underline{h}_1$, $\underline{h}_2$.
 \end{definition}
The following proposition provides a priori estimates for the linearized GCM system \eqref{GeneralizedGCMsystem}, which play a crucial role in the proof of Theorem \ref{Theorem:ExistenceGCMS1} (see the proof of Theorem 6.1 in \cite{KS:Kerr1}).
\begin{proposition}
\label{Thm.GCMSequations-fixedS:contraction}
  Assume  $\S$ is a given  $O(\epg)$-sphere  in $\RR$.  Assume given a solution     $(f, \fb, \ovla, \Cbdot_0, \Mdot_0, \Cbpdot, \Mpdot, \ovb)$  of the system     \eqref{GeneralizedGCMsystem}, and verifying
  \beaa 
  (\divS f)_{\ell=1}=\La,\qquad (\divS f)_{\ell=1}=\Lab.
  \eeaa 
 Then, the following a priori estimates are verified, for $3\leq s\leq s_{max}+1$,
\bea
\begin{split}
&\Big\|\Big(f,\fb, \ovla-\ov{\ovla}^\S\Big)\Big\|_{\hk_{s}(\S)} +\sum_p\Big(r^2|\Cbpdot|+r^3|\Mpdot|\Big)\\  \les &\; r\|(h_1-\ov{h_1}^\S,\, \hb_1-\ov{\hb_1}^\S, \,h_2-\ov{h_2}^\S,\,h_2-\ov{\hb_2}^\S)\|_{\hk_{s-1}(\S)} \\&+r^2\|h_3-\ov{h_3}^\S\|_{\hk_{s-2}(\S)}+r\|h_4-\ov{h_4}^\S\|_{\hk_{s-3}(\S)} +|\La|+|\Lab|,
\end{split}
\eea
and
\bea
\begin{split}
r^2|\Cbdot_0|+r^3|\Mdot_0|+r\Big|\ov{\ovla}^\S\Big|&\les r\|(h_1-\ov{h_1}^\S,\, \hb_1-\ov{\hb_1}^\S, \,h_2-\ov{h_2}^\S,\,h_2-\ov{\hb_2}^\S)\|_{\hk_{s-1}(\S)} \\&+r^2\|h_3-\ov{h_3}^\S\|_{\hk_{s-2}(\S)}+r\|h_4-\ov{h_4}^\S\|_{\hk_{s-3}(\S)} +|\La|+|\Lab|+|b_0|.
\end{split}
\eea
\end{proposition}
\begin{proof}
See Proposition 4.13 in \cite{KS:Kerr1}.
\end{proof}
Proposition \ref{Thm.GCMSequations-fixedS:contraction} will be used in Sections \ref{sec4.4} and \ref{sec4.5}.
\subsection{GCM spheres with more general \texorpdfstring{$\ell=1$}{} modes}
\label{generalGCMsection}
We have the following corollary of Theorem \ref{Theorem:ExistenceGCMS1}.
\begin{corollary}\label{generalGCMspheres}
Under the same assumptions as in Theorem \ref{Theorem:ExistenceGCMS1}, assume moreover that there exists a triplet of scalar functions $\widetilde{J}^{(p)}$ verifying
\bea\label{widetildeJ}
\sum_{p}\left\|\Jp-\widetilde{J}^{(p)}\right\|_{\hk_{s_{max}}(\ovS)}\les r\dg,
\eea
and
\begin{equation}\label{extendJtilde}
    \pr_u \widetilde{J}^{(p)}(u,s,y^1,y^2)=\pr_s \widetilde{J}^{(p)}(u,s,y^1,y^2)=0.
\end{equation}
Then, for any fixed triplets $\La,\Lab\in\mathbb{R}^3$ verifying
\bea\label{lalabdg}
|\La|,\,|\Lab|&\les& \dg,
\eea
there exists a unique GCM sphere $\S=\S(\La,\Lab,\widetilde{J}^{(p)})$, which is a deformation of $\ovS$, such that there exists constants $\Cb_0^\S$, $M_0^\S$, $\CbpS,\MpS$, $p\in\{-,0,+\}$ for which the following GCM conditions are verified
\begin{align}
\begin{split}\label{3.20}
\kaS&= \frac{2}{\rS},\\
\kabS&= -\frac{2}{\rS}\UpS+\Cb_0^\S+\sum_p \CbpS\widetilde{J}^{(p)},\\
\muS&= \frac{2\mS}{(\rS)^3}+M_0^\S+\sum_p \MpS\widetilde{J}^{(p)}.
\end{split}
\end{align}
Moreover,
\bea
(\divS f)_{\ell=1}=\La,\qquad (\divS \fb)_{\ell=1}=\Lab,\label{3.23}
\eea
where $(\divS f)_{\ell=1},\,(\divS \fb)_{\ell=1}$ are defined with respect to $\widetilde{J}^{(p)}$. Finally, the resulting deformation verifies all the properties 1-9 of Theorem \ref{Theorem:ExistenceGCMS1}.
\end{corollary}
\begin{proof}
Notice that \eqref{widetildeJ}, \eqref{extendJtilde} and \eqref{eq:Jpsphericalharmonics} implies that $\widetilde{J}^{(p)}$ satisfies Assumption {\bf A4}. Then, it suffices to apply Theorem \ref{Theorem:ExistenceGCMS1} with $\widetilde{J}^{(p)}$ replacing $\Jp$.
\end{proof}
\section{Construction of GCM hypersurfaces}
\label{mainsection}
We are ready to state the precise version of the main result of this paper.
\begin{theorem}[Existence of GCM hypersurfaces, version 2]\label{mthm}
Let $m_0>0$ a constant. Let $0<\dg\leq\epg$ two sufficiently small constants, and let $(\ug,\sg,\rg)$ three real numbers with $\rg$ sufficiently large so that
\beaa
\epg\ll m_0, \qquad\qquad  \rg\gg m_0.
\eeaa
Let a fixed space-time region $\RR$ as in Definition \ref{defintion:regionRRovr}, together with a $(u,s)$ outgoing geodesic foliation and a basis of $\ell=1$ modes $\Jp$ verifying assumptions {\bf{A1-A4}} in Section \ref{subsubsect:regionRR2}. We further assume that
\bea\label{assjp}
\sup_{\RR} r\left|{\dkt}^{\leq s_{max}+1} e_3(\Jp)\right| &\les& \dg,\qquad s_{max}\geq 5,
\eea
where
\beaa 
\dkt&:=&(e_3-(z+\Omb)e_4,\dkb),\qquad \dkb=r\nab,
\eeaa
denotes the weighted derivatives tangent to the level hypersurfaces of $u+s$. In addition, we assume on $\RR$, the estimates
\bea
\label{eq:GCM-improved estimate2-again}
\sup_{\RR}\big|\dkt^{\leq s_{max}+1}(\kadot, \kabdot)|\les r^{-2}\dg,\qquad\qquad
\sup_{\RR}\big|\dkt^{\leq s_{max}+1}\mudot| \les r^{-3}\dg,
\eea
and
\bea\label{assuml1}
|(\div\eta)_{\ell=1}|\les\dg,\qquad |(\div \xib)_{\ell=1}|\les \dg.
\eea
We also assume
\bea\label{e3re3s}
|r-s|+|e_3(r)-e_3(s)|&\les&\dg,
\eea
as well as the existence of a constant $m_{(0)}$ such that we have on $\RR$,
\bea\label{vsiOmbass}
\left|\ov{z+\Omb}-1-\frac{2m_{(0)}}{r}\right|\les \dg,
\eea
where $\ov{z+\Omb}$ denotes the average of $z+\Omb$ on the sphere $S(u,s)$.\\ \\
Let $\S_0$ be a fixed sphere included in the region $\RR$, let a pair of triplets $\La_0,\Lab_0\in \mathbb{R}^3$ such that
\bea \label{4.6}
|\La_0|, |\Lab_0|\les\dg,
\eea 
and let $J^{(\S_0,p)}$ be a basis of $\ell=1$ modes on $\S_0$ satisfying
\bea\label{csjs}
\Vert J^{(\S_0,p)}-\Jp\Vert_{\hk_{s_{max}+1}(\S_0)}\les r\dg.
\eea
Assume that we have on $\S_0$
\begin{align}
    \begin{split}
        \ka^{\S_0}&=\frac{2}{r^{\S_0}},\\
        \kab^{\S_0}&=-\frac{2\Up^{\S_0}}{r^{\S_0}}+\Cb_0^{\S_0}+\sum_p\Cb^{(\S_0,p)}J^{(\S_0,p)},\\
        \mu^{\S_0}&=\frac{2m^{\S_0}}{r^{\S_0}}+M_0^{\S_0}+\sum_p M^{(\S_0,p)}J^{(\S_0,p)},
    \end{split}
\end{align}
as well as
\bea
(\div f_0)_{\ell=1}=\La_0,\qquad (\div \fb_0)_{\ell=1}=\Lab_0,
\eea 
with $(f_0,\fb_0)$ corresponding to the coefficients from the background frame to the frame adapted to $\S_0$, and the $\ell=1$ modes being taken w.r.t. the basis $J^{(\S_0,p)}$.

Then, there exists a unique, local, smooth, spacelike hypersurface $\Sigma_0$ passing through $\S_0$, a scalar function $u^\S$ defined on $\Sigma_0$, whose level surfaces are topological spheres denoted by $\S$ with adapted frame $(e_3^\S,e_4^\S,e_1^\S,e_2^\S)$, a smooth collection of constants $\La^\S,\Lab^\S$ and a triplet of functions $\JpS$ defined on $\Si_0$ verifying
\beaa
\La^{\S_0}=\La_0,\qquad \Lab^{\S_0}=\Lab_0,\qquad \JpS|_{\S_0}=J^{(\S_0,p)},
\eeaa
such that the following conditions are verified:
\begin{enumerate}
\item The following GCM conditions hold on $\Si_0$
\begin{align}
\begin{split}\label{GCMcondi}
\kaS =& \frac{2}{\rS}, \\
\kabS =& -\frac{2}{\rS}\UpS+\Cb_0^\S +\sum_p \Cb^{(\S,p)}\JpS ,\\
\muS =& \frac{2\mS}{(\rS)^3}+M_0^\S +\sum_p M^{(\S,p)}\JpS.
\end{split}
\end{align}
\item Denoting $\rS$ to be the area radius of the spheres $\S$ we have, for some constant $c_0$,
\bea
u^\S+r^\S=c_0,\quad \mbox{along}\quad \Sigma_0.
\eea
\item Let $\nu^\S$ be the unique vectorfield tangent to the hypersurface $\Sigma_0$, normal to $\S$, and normalized by $\g(\nu^\S,e_4^\S)=-2$. There exists a unique scalar function $b^\S$ on $\Sigma_0$ such that $\nu^\S$ is given by
    \beaa
    \nu^\S=e_3^\S+b^\S e_4^\S.
    \eeaa
Then, the following normalization condition holds true on every sphere $\S$
    \bea\label{4.7}
    \ov{b^\S}=-1-\frac{2m_{(0)}}{r^\S},
    \eea
    where $\ov{\bS}$ denotes the average of $\bS$ on the sphere $\S$.
\item The triplet of functions $\JpS$ verifies on $\Si_0$\footnote{A basis of $\ell=1$ modes verifying \eqref{nuSJpS=0} can simplify the transport equations along $\nu^\S$ for the $\ell=1$ modes of the transition functions $(f,\fb)$, see Lemma \ref{ODElalab} for more details.}
\bea\label{nuSJpS=0}
\nu^\S (\JpS)=0,\quad p=0,+,-.
\eea
\item The following transversality conditions are assumed\footnote{A priori, the quantities $\xi^\S$, $\om^\S$, $\etab^\S$, $e_4^\S(\rS)$ and $e_4^\S(\uS)$ do not make sense on $\Si_0$. The conditions \eqref{transversalityru} and \eqref{transversalitye4re3u} are in fact consistent with a local extension of $\Si_0$ by an outgoing geodesic foliation, see Remark \ref{rem4.5}.}
\bea\label{transversalityru}
\xi^\S=0,\qquad \omS=0,\qquad \etab^\S=-\zeS,
\eea
and
\bea\label{transversalitye4re3u}
e_4^\S(\rS)=1,\qquad e_4^\S(u^\S)=0.
\eea 
\item In view of \eqref{transversalityru}, the Ricci coefficients $\etaS,\xibS$ are well defined for every $\S\subset\Si_0$. They verify
\bea\label{4.10}
(\div^\S \etaS)_{\ell=1}=0,\qquad (\div^\S \xib^\S)_{\ell=1}=0.
\eea
\item The transition coefficients from the background foliation to that of $\Si_0$ verify
\bea
\Vert (f,\fb,\ovla)\Vert_{\hk_{s_{max}+1}(\S)}+\Vert\dk(f,\fb,\ovla)\Vert_{\hk_{s_{max}}(\S)}\les \dg,
\eea
where 
\beaa 
\dk\in\{\nabS_3,r\nabS,r\nabS_4\}.
\eeaa 
\end{enumerate}
\end{theorem}
We state the following corollary.
\begin{corollary}\lab{coro25}
Let $\RR$ be a fixed spacetime region with a background foliation verifying the assumption of Theorem \ref{mthm}, including \eqref{assjp}-\eqref{vsiOmbass}. Also, assume given a GCM hypersurface $\Si_0\subset \RR$ foliated by surfaces $\S$ such that
\beaa
\kaS &=& \frac{2}{\rS},\\
\kabS &=& -\frac{2\UpS}{\rS}+\Cb_0^\S+\sum_p \CbpS \JpS,\\
\muS&=& \frac{2\mS}{(\rS)^3}+M_0^\S +\sum_p \MpS\JpS,\\
(\divS \etaS)_{\ell=1}&=& (\divS \xibS)_{\ell=1} =0 ,\qquad \nu^\S (\JpS)=0,\quad p=0,+,-,
\eeaa
where the $\ell=1$ modes are defined w.r.t. $\JpS$. Assume moreover
\bea\label{divetadivxibpr}
|\dkt^{\leq s_{max}}(\div\eta)_{\ell=1}|\les\dg,\qquad |\dkt^{\leq s_{max}}(\div\xib)_{\ell=1}|\les\dg,
\eea
and
\bea\label{e3re3saddition}
|\dkt^{\leq s_{max}}(e_3(r)-e_3(s))|&\les&\dg,
\eea
where $\dkt\in\{e_3-(z+\Omb)e_4,\dkb\}.$ Then:
\begin{enumerate}
\item If we assume in addition that for a given sphere $\S_0$ on $\Si_0$, the transition coefficients $(f,\fb,\la)$ from the background foliation to $\S_0$ verify
\bea\label{zctj}
\|(f,\fb,\ovla)\|_{\hk_{s_{max}+1}(\S_0)}&\les& \dg,
\eea
then
\beaa
\|\dk^{\leq s_{max}+1}(f,\fb,\ovla)\|_{L^2(\S_0)}&\les&\dg.
\eeaa
\item If we assume in addition that for a given sphere $\S_0$ on $\Si_0$, the transition coefficients $(f,\fb,\la)$ from the background foliation to $\S_0$ verify
\bea\label{atj}
\| f\|_{\hk_{s_{max}+1}(\S_0)}+(r^{\S_0})^{-1}\|(\fb,\ovla)\|_{\hk_{s_{max}+1}(\S_0)}&\les& \dg,
\eea
then
\beaa
\|\dk^{\leq s_{max}+1}f\|_{L^2(\S_0)}+(r^{\S_0})^{-1}\|\dk^{\leq s_{max}+1}(\fb,\ovla)\|_{L^2(\S_0)}+\|\dk^{\leq s_{max}}\nabS_{\nu^\S}(\fb,\ovla)\|_{L^2(\S_0)} &\les& \dg.
\eeaa
\end{enumerate}
\end{corollary}
\begin{remark}
Theorem \ref{mthm} is the generalization of Theorem 9.52 in \cite{KS} and Corollary \ref{coro25} is the generalization of Corollary 9.53 in \cite{KS} in the absence of symmetry. Theorem \ref{mthm} plays a central role in the proof of Theorems M6 and M7 in \cite{KS:main}, see Section 8.4 and 8.5 in \cite{KS:main}. Note that Theorem \ref{mthm} and Corollary \ref{coro25} are restated as Theorem 8.1.10 and Corollary 8.1.11 in \cite{KS:main}.
\end{remark}
\begin{remark}
The proof of Theorem \ref{mthm} and Corollary \ref{coro25} is similar to the one of Theorem 9.52 in \cite{KS} and Corollary 9.53 in \cite{KS}. The main differences are:
\begin{itemize}
    \item the absence of symmetry in Theorem \ref{mthm} and Corollary \ref{coro25} while \cite{KS} relies on axial polarization,
    \item the dependance of the construction on the choice of the basis of $\ell=1$ modes, while the basis of $\ell=1$ modes in \cite{KS} is fixed by the polarized symmetry.
\end{itemize}
\end{remark}
The proof of Theorem \ref{mthm} is given in Sections \ref{sec4.1}-\ref{sec4.4} and the proof of Corollary \ref{coro25} is given in Section \ref{sec4.5}.
\subsection{Definition of a family of hypersurfaces $\Sigma$}\label{sec4.1}
As stated in Theorem \ref{mthm}, we assume given a spacetime region 
$$\RR=\{|u-\ovu|\leq \epg ,\, |s-\ovs|\leq\epg\}$$ 
endowed with a background foliation such that the conditions {\bf A1-A4} and \eqref{assjp}-\eqref{vsiOmbass} hold true. We also assume given a sphere $\S_0$, triplets $(\La_0, \Lab_0)$ and a basis of $\ell=1$ modes $J^{(p,\S_0)}$ satisfying \eqref{4.6}-\eqref{csjs}. In view of the uniqueness in Corollary \ref{generalGCMspheres}, note that $\S_0$ is in fact the deformation sphere
\beaa
\S_0=\S[\ovu,\ovs,\La_0,\Lab_0,J^{(p)}[\S_0]]
\eeaa
of the given sphere $\ovS = S(\ovu,\ovs)$ of the background foliation provided by Corollary \ref{generalGCMspheres}. Our goal is to construct, in a small neighborhood of $\S_0$, a spacelike hypersurface $\Si_0$ passing through $\S_0$ verifying all the desired properties of Theorem \ref{mthm}. To this end, we first define more general families of hypersurfaces within which we will make a suitable choice of $\Si_0$ in Section \ref{sec4.4}. We start with the following definition.
\begin{definition}\label{SigmaJ}
Under the same assumptions as in Theorem \ref{mthm}, let $\Psi(s)$ be a real valued function, $\La(s)$, $\Lab(s)$ triplet of functions satisfying
\bea
\begin{split}\label{curvecon}
|\La(s)|,\;|\La'(s)|,\;|\Lab(s)|,\;|\Lab'(s)|,\;|\Psi(s)+s-c_0|,\;|\Psi'(s)+1|\les \dg,\qquad s\in \ovI,\\
\La(\ovs)=\La_0,\qquad \Lab(\ovs)=\Lab_0,\qquad \Psi(\ovs)=\ovu,
\end{split}
\eea
where
$$c_0:=\ovu+\ovs,\qquad \ovI:=[\ovs,s_1],$$
with $s_1$ verifying $|s_1-\ovs|\les\epg$. Then, we denote
\beaa 
\Si_\#:=\bigcup_{s\in\ovI}S(\Psi(s),s),
\eeaa 
and $\nu_\#$ is the unique tangent vector to $\Si_\#$ normal to $S(\Psi(s),s)$ with $\g(\nu_\#,e_4)=-2$. By definition, $\Si_\#$ is the hypersurface $\{u=\Psi(s),s\in\ovI\}$.\\ \\
Let a triplet of functions $\Jt$ defined on $\Si_\#$ such that 
\bea \label{JJpdiff}
\sup_{s\in\ovI}\sum_p\|\Jt^{(p)}(s)-\Jp \|_{\hk_{s_{max}}(S(\Psi(s),s))}\les r\dg,\qquad \Jt(s):=\Jt\big|_{S(\Psi(s),s)}.
\eea 
For every $s\in\ovI$, we apply Corollary \ref{generalGCMspheres} to the background sphere $S(\Psi(s),s)$ and the triplet $\Jt(s)$\footnote{For every $s$, we extend $\Jt(s)$ to $\RR$ by \eqref{extendJtilde} in order to apply Corollary \ref{generalGCMspheres}.} to obtain a GCM sphere $\S[\Psi(s),s,\La(s),\Lab(s),\Jt(s)]$. The deformation map is given by:
\begin{align*}
\Phi^{\Jt}(s):S(\Psi(s),s)&\longrightarrow \S[\Psi(s),s,\La(s),\Lab(s),\Jt(s)]\subset \RR, \\
(\Psi(s),s,y^1,y^2)&\longrightarrow (\Psi(s)+U^{\Jt}(\Psi(s),s,y^1,y^2),s+S^{\Jt}(\Psi(s),s,y^1,y^2),y^1,y^2).
\end{align*}
The adapted null frame of $\S[\Psi(s),s,\La(s),\Lab(s),\Jt(s)]$ is denoted by $(e_3^{\Jt},e_4^{\Jt},e_1^{\Jt},e_2^{\Jt})$ and the transition functions are denoted by 
\beaa 
F^{\Jt}(s):=\left(f^{\Jt}(s),\fb^{\Jt}(s),\ovla^{\Jt}(s)\right).
\eeaa
Moreover, we denote
\beaa 
F^{\Jt,\#}(s):=\left(f^{\Jt,\#}(s),\fb^{\Jt,\#}(s),\ovla^{\Jt,\#}(s)\right):=(\Phi^{\Jt}(s))^\#\left(f^{\Jt}(s),\fb^{\Jt}(s),\ovla^{\Jt}(s)\right).
\eeaa 
Then, we define
\bea\label{hyperSiJt}
\Si_{\Jt}&:=&\bigcup_{s\in\ovI}\S[\Psi(s),s,\La(s),\Lab(s),\Jt(s)].
\eea
Let $\nu^{\Jt}$ the unique vectorfield tangent to the hypersurface $\Si_{\Jt}$ such that $\g(\nu^{\Jt},e_4^{\Jt})=-2$ and normal to $\S[\Psi(s),s,\La(s),\Lab(s),\Jt(s)]$. We define 
\beaa
\nu^{\Jt}_{\#}:=\left((\Phi^{\Jt})^{-1}\right)_\#(\nu^{\Jt}),
\eeaa
which is a tangent vectorfield on $\Si_\#$, where $\Phi^{\Jt}: \Si_\# \to\Si_{\Jt}$ is defined by
\beaa
\Phi^{\Jt}\big|_{S(\Psi(s),s)}=\Phi^{\Jt}(s).
\eeaa
\end{definition}
In the sequel, we denote
\bea
S:=S(\Psi(s),s),\qquad \ovS:=S(\ovu,\ovs).
\eea
We will need the following four lemmas.
\begin{lemma}\label{lemmaapp}
For a triplet of functions $\Jt$ on $\Si_\#$ satisfying \eqref{JJpdiff}, we define $(U^\Jt,S^\Jt,F^\Jt,\nu^\Jt_\#)$ as in Definition \ref{SigmaJ}. Then, we have
\bea \label{nujnbsl}
\nu^{\Jt}_\#=\left(1+A(U^\Jt,S^\Jt,F^{\Jt,\#})\right)\nu_\#+B^a(U^\Jt,S^\Jt,F^{\Jt,\#})\pr_{y^a},
\eea
where
\beaa 
A(U,S,F)=O(\dk_\#^{\leq 1}(U,S),F),\qquad B^a(U,S,F)=\epg r^{-2}O(\dk_\#^{\leq 1}(U,S),F),
\eeaa 
with
\beaa 
\dk_\#&:=&\{\dkb,\,\nu_\#\}.
\eeaa 
\end{lemma}
\begin{proof}
See Appendix \ref{nunu}.
\end{proof}
\begin{lemma}\label{lemmaappa}
For a triplet of functions $\Jt$ on $\Si_\#$ satisfying \eqref{JJpdiff}, we define $(U^\Jt,S^\Jt,F^\Jt,\nu^\Jt_\#)$ as in Definition \ref{SigmaJ}. In addition, we assume that
\bea\label{nuJO1}
\left\|\nu_\#^\Jt(\Jt)\right\|_{L^\infty(\Si_\#)}\leq 1.
\eea
Then, we have
\bea
\|\nu_\#(U^{\Jt},S^{\Jt})\|_{\hk_{s}(S)}\les r\dg,\qquad 0\leq s\leq s_{max}.
\eea
\end{lemma}
\begin{proof}
See Appendix \ref{AAAA}.
\end{proof}
\begin{lemma}\label{lemmaappb}
For two triplets of functions $\Jt$ and $\Jh$ on $\Si_\#$ satisfying \eqref{JJpdiff}, we define $(U^\Jt,S^\Jt,F^\Jt,\nu^\Jt_\#)$ and respectively $(U^\Jh,S^\Jh,F^\Jh,\nu^\Jh_\#)$ as in Definition \ref{SigmaJ}. In addition, we assume that
\bea\label{nuJtJhO1}
\left\|\nu_\#^\Jt(\Jt,\Jh)\right\|_{L^\infty(\Si_\#)}\leq 1.
\eea
Then, we have
\beq\label{cor6.11}
r^{-1}\|(U^{\Jt}-U^{\Jh},S^{\Jt}-S^{\Jh})\|_{\hk_{s+1}(S)}+ \|F^{\Jt,\#}-F^{\Jh,\#}\|_{\hk_{s+1}(S)}\les \epg r^{-1}\|\Jt-\Jh\|_{\hk_{s}(S)},
\eeq
for all $0\leq s \leq s_{max}$ and
\beq\label{nucor6.11}
\|\nu_\#(U^{\Jt}-U^{\Jh},S^{\Jt}-S^{\Jh})\|_{\hk_{s+1}(S)}\les\epg \|\Jt-\Jh\|_{\hk_{s}(S)}+\epg \|\nu_\#(\Jt-\Jh)\|_{\hk_{s}(S)},
\eeq
for all $0\leq s\leq s_{max}-1$.
\end{lemma}
\begin{proof}
See Appendix \ref{BBBB}.
\end{proof}
\begin{lemma}\label{transportlemma}
For a triplet of functions $\Jt$ satisfying \eqref{JJpdiff}, we define $\nu^\Jt_\#$ as in Definition \ref{SigmaJ}. Then, for any scalar function $h$ and $0\leq k\leq s_{max}$, we have
\bea
\sup_{S\subset\Si_\#}\|h\|_{\hk_k(S)} \leq (1+O(\epg))\|h\|_{\hk_k(\ovS)}+\sup_{S\subset\Si_\#}\left\|\nu_\#^\Jt(h)\right\|_{\hk_k(S)}.
\eea
\end{lemma}
\begin{proof}
See Appendix \ref{lemmaappd}.
\end{proof}
The following theorem will allow us to introduce a suitable family of hypersurfaces in Definition \ref{hypersurfaceSigma}.
\begin{theorem}\label{generalcoro}
There exists a unique $\Jt$ satisfying \eqref{JJpdiff} and verifying
\bea 
\begin{split}\label{desiredJt}
\nu^{\Jt}\left(\left((\Phi^{\Jt})^{-1}\right)^\#\Jt\right)&=0,\qquad \quad \mbox{ on }\Si_\Jt,\\
\left((\Phi^{\Jt})^{-1}\right)^\#\Jt^{(p)}&=J^{(\S_0,p)},\quad \mbox{ on }\S_0. 
\end{split}
\eea 
\end{theorem}
\begin{proof} The proof follows from a standard Banach fixed-point argument. More precisely, we construct a map $T$ which sends a triplet of functions $\Jt$ on $\Si_\#$ to another triplet of functions $T(\Jt)$ on $\Si_\#$. By taking a suitable choice of the domain of $T$, we prove that $T$ is a contraction map and hence admits a unique fixed point. Finally, we conclude that the fixed point of $T$ satisfies \eqref{desiredJt}.\\ \\
\noindent{\bf Step 1. }Construction of the map $T$.\\ \\
For a triplet of functions $\Jt$ defined on $\Si_\#$ satisfying \eqref{JJpdiff}, we define
the triplet of functions $T(\Jt)$ on $\Si_\#$ by the following transport equation:
\bea
\begin{split}\label{defTJ}
\nu^\Jt_\# ( T(\Jt))&= 0,\qquad\qquad \mbox{ on }\Si_\#,\\
T(\Jt)^{(p)}&=\Jp[\S_0] ,\quad\; \mbox{ on }\ovS,
\end{split}
\eea
where $\Jp[\S_0]$ is the extension of $J^{(\S_0,p)}$ to $\RR$ by $\pr_u(\Jp[\S_0])=\pr_s(\Jp[\S_0])=0$.\\ \\
We define the following norms for triplets of functions $H^{(p)}$ on $\Si_\#$ and $1\leq s\leq s_{max}$:
\bea \label{defXX}
\|H\|_{\XX_s}&:=&\sup_{S\subset\Si_\#}\sum_p\left(\|H^{(p)}\|_{\hk_{s}(S)}+\|\nu_\#(H^{(p)})\|_{\hk_{s-1}(S)}\right).
\eea 
We denote $H\in\XX_s$ if $\|H\|_{\XX_s}<+\infty$. Then $\XX_s$ is a Banach space for $1\leq s\leq s_{max}$. \\\\
Recalling $\pr_u(\Jp)=\pr_s(\Jp)=0$, we can rewrite \eqref{csjs} as follows:
\bea \label{rewrite}
\sum_p\|\Jp[\S_0]-\Jp\|_{\hk_{s_{max}+1}(\ovS)}\leq C_0\ovr\dg,
\eea 
where $C_0$ is a given constant. We define
\beq\label{Banachspace}
\XX:= \left\{\Jt\in\XX_{s_{max}}\Big/\|\Jt-J\|_{\XX_{s_{max}}}\leq 2C \ovr \dg,\quad \left\|\nu_\#^\Jt(\Jt)\right\|_{L^\infty(\Si_\#)}\leq 1\right\}.
\eeq
By definition, $\XX$ is a closed subset of $\XX_{s_{max}-1}$.\\ \\
\noindent{\bf Step 2.} Boundedness of $T$.\\ \\
Notice that for $\Jt\in\XX$, \eqref{JJpdiff} holds true. Thus, $T$ is well-defined on $\XX$. Moreover \eqref{nuJO1} holds, and we can thus apply Lemma \ref{lemmaappa} to obtain
\bea\label{nuUSyoujie}
\|\nu_\#(U^\Jt,S^\Jt)\|_{\hk_{s_{max}}(S)}\les r\dg.
\eea
Moreover, \eqref{eq:ThmGCMS1} and \eqref{eq:ThmGCMS4}, together with Lemma \ref{lemma:comparison-gaS-ga}\footnote{Lemma \ref{lemma:comparison-gaS-ga} allows to compare the norms $\|F^\Jt\|_{\hk_{s_{max}+1}(\S)}$ and $\|F^{\Jt,\#}\|_{\hk_{s_{max}+1}(S)}$.}, imply
\bea\label{USFyoujie}
r^{-1}\|(U^\Jt,S^\Jt)\|_{\hk_{s_{max}+1}(S)}+\|(F^{\Jt,\#})\|_{\hk_{s_{max}+1}(S)}\les\dg.
\eea
Recalling \eqref{assjp}, \eqref{nujnbsl}, \eqref{defTJ}, \eqref{nuUSyoujie} and $\|J\|_{\hk_{s_{max}+1}(S)}\les r$, we infer
\bea 
\begin{split}\label{nuJtTJtJ}
\|\nu^\Jt_\#(T(\Jt)-J)\|_{\hk_{s_{max}}(S)}&=\|\nu^\Jt_\#(J)\|_{\hk_{s_{max}}(S)}\\
&\les\|\nu_\#(J)\|_{\hk_{s_{max}}(S)}+\epg\dg r^{-2}\|\pr_{y^a}(J)\|_{\hk_{s_{max}}(S)}\\
&\les \dg+\epg\dg\les\dg.    
\end{split}
\eea 
Applying Lemma \ref{transportlemma} with $h=T(\Jt)-J$ and recalling \eqref{rewrite}, for $\epg$ small enough and $\ovr$ large enough, we obtain
\bea 
\begin{split}\label{Boundednessstep1}
\|T(\Jt)-J\|_{\hk_{s_{max}}(S)}&\leq(1+O(\epg))\|T(\Jt)-J\|_{\hk_{s_{max}}(\ovS)}+\sup_S\|\nu^\Jt_\#(T(\Jt)-J)\|_{\hk_{s_{max}}(S)} \\
&\leq \frac{3}{2}C_0\dg\ovr.
\end{split}
\eea 
Moreover, recalling \eqref{nujnbsl}, \eqref{nuUSyoujie}, \eqref{USFyoujie}\footnote{\eqref{nuUSyoujie} and \eqref{USFyoujie} imply the boundedness of $A$ and $B$ in \eqref{nujnbsl}.}, \eqref{nuJtTJtJ} and \eqref{Boundednessstep1}, we have
\beaa 
\|\nu_\#(T(\Jt)-J)\|_{\hk_{s_{max}-1}(S)}&\les&\|\nu^\Jt_\#(T(\Jt)-J)\|_{\hk_{s_{max}-1}(S)}+\epg r^{-2}\|\pr_{y^a}(T(\Jt)-J)\|_{\hk_{s_{max}-1}(S)}\\
&\les&\dg+\epg r^{-2}\dg\|T(\Jt)-J\|_{\hk_{s_{max}}(S)}\les \dg.
\eeaa 
Thus, for $\ovr$ large enough, we obtain
\bea\label{Boundednessstep2}
\|\nu_\#(T(\Jt)-J)\|_{\hk_{s_{max}-1}(S)}&\leq& \frac{1}{2}C_0\dg\ovr.
\eea 
From \eqref{Boundednessstep1} and \eqref{Boundednessstep2} we obtain
\bea\label{Boundednessstep3}
\|T(\Jt)-J\|_{\hk_{s_{max}}(S)}+\|\nu_\#(T(\Jt)-J)\|_{\hk_{s_{max}-1}(S)}\leq 2C_0\dg\ovr,
\eea
which implies 
\bea \label{tj1}
\|T(\Jt)-J\|_{\XX_{s_{max}}}\leq 2C_0\dg\ovr.
\eea 
Moreover, \eqref{Boundednessstep3} implies that
\bea \label{Boundednessstep4}
\|\dk_\#(T(\Jt))\|_{\hk_{s_{max}-1}(S)}\les r.
\eea 
Next, recalling \eqref{nujnbsl}, \eqref{defTJ}, \eqref{defXX}, \eqref{Boundednessstep3}, \eqref{Boundednessstep4} and applying Lemma \ref{lemmaappb}, we infer
\beaa 
\left|\nu_\#^{T(\Jt)}(T(\Jt))\right|&=&\left|\nu_\#^{T(\Jt)}(T(\Jt))-\nu_\#^\Jt(T(\Jt))\right|\\
&\les &\left|A(U^{T(\Jt)},S^{T(\Jt)},F^{T(\Jt),\#})-A(U^{\Jt},S^{\Jt},F^{\Jt,\#})\right||\nu_\#(T(\Jt))|\\
&+&\left|B(U^{T(\Jt)},S^{T(\Jt)},F^{T(\Jt),\#})-B(U^{\Jt},S^{\Jt},F^{\Jt,\#})\right||\pr_{y^a}(T(\Jt))|\\
&\les&\left|\dk_\#^{\leq 1}(U^{T(\Jt)}-U^{\Jt},S^{T(\Jt)}-S^{\Jt})\right|+\left|F^{T(\Jt),\#}-F^{\Jt,\#}\right|\\
&\les&\epg r^{-1}\|T(\Jt)-\Jt \|_{\hk_2(S)}+\epg r^{-1}\|\nu_\#(T(\Jt)-\Jt)\|_{\hk_2(S)}\\
&\les&\epg r^{-1}\|T(\Jt)-J\|_{\XX_{s_{max}}}\les \dg,
\eeaa 
which implies for $\dg$ small enough
\bea\label{tj2}
\left\|\nu_\#^{T(\Jt)}(T(\Jt))\right\|_{L^\infty(\Si_\#)}\leq 1.
\eea
Notice that \eqref{tj1} and \eqref{tj2} imply $T(\Jt)\in\XX$. Hence, we have $T: \XX\rightarrow \XX$. \\ \\
\noindent{\bf Step 3.} $T$ is a contraction map.\\ \\
For two triplets of functions $\Jt$ and $\Jh$ in $\XX$, we have
\beaa 
\|\nu^\Jt_\#(T(\Jt)-T(\Jh))\|_{\hk_{s_{max}-1}(S)}=\|\nu^\Jt_\#(T(\Jh))\|_{\hk_{s_{max}-1}(S)}=\|(\nu^\Jt_\#-\nu^\Jh_\#)(T(\Jh))\|_{\hk_{s_{max}-1}(S)}.
\eeaa 
Recall from \eqref{nujnbsl} that 
\beaa 
\nu_\#^\Jt-\nu_\#^\Jh&=&O\left(\dk_\#^{\leq 1}(U^\Jt-U^\Jh,S^\Jt-S^\Jh), F^{\Jt,\#}-F^{\Jh,\#}\right)\nu_\#\\
&+&\frac{\epg }{r^2}O\left(\dk_\#^{\leq 1}(U^\Jt-U^\Jh,S^\Jt-S^\Jh), F^{\Jt,\#}-F^{\Jh,\#}\right)\pr_{y^a}.
\eeaa 
Applying Lemma \ref{lemmaappb}, \eqref{defXX} and \eqref{Boundednessstep4}, we obtain
\bea 
\begin{split}\label{nuTJtTJh}
\|\nu^\Jt_\#(T(\Jt)-T(\Jh))\|_{\hk_{s_{max}-1}(S)}&\les\|\dk_\#^{\leq 1}(U^\Jt-U^\Jh,S^\Jt-S^\Jh), F^{\Jt,\#}-F^{\Jh,\#}\|_{\hk_{s_{max}-1}(S)}\\
&\les\epg \left(\|\Jt-\Jh\|_{\hk_{s_{max}-1}(S)}+\|\nu_\#(\Jt-\Jh)\|_{\hk_{s_{max}-2}(S)} \right)\\
&\les \epg \|\Jt-\Jh\|_{\XX_{s_{max}-1}}.    
\end{split}
\eea 
Applying Lemma \ref{transportlemma} with $h=T(\Jt)-T(\Jh)$ and Recalling that $T(\Jt)=T(\Jh)$ on $\ovS$, we obtain
\bea\label{TJtTJh}
\|T(\Jt)-T(\Jh)\|_{\hk_{s_{max}-1}(S)}&\les&\epg \|\Jt-\Jh\|_{\XX_{s_{max}-1}}.
\eea 
Moreover, we have from \eqref{nujnbsl}, \eqref{defXX} and \eqref{nuTJtTJh}
\beaa 
\|\nu_\#(T(\Jt)-T(\Jh))\|_{\hk_{s_{max}-2}(S)}&\les& \|\nu^\Jt_\#(T(\Jt)-T(\Jh))\|_{\hk_{s_{max}-2}(S)}+\epg r^{-2} \|T(\Jt)-T(\Jh)\|_{\hk_{s_{max}-1}(S)}\\
&\les&\epg \|\Jt-\Jh\|_{\XX_{s_{max}-1}}+\epg \|T(\Jt)-T(\Jh)\|_{\hk_{s_{max}-1}(S)}.
\eeaa 
Combining with \eqref{TJtTJh}, we obtain 
\beaa 
\|T(\Jt)-T(\Jh)\|_{\hk_{s_{max}-1}(S)}+\|\nu_\#(T(\Jt)-T(\Jh))\|_{\hk_{s_{max}-2}(S)}\les \epg \|\Jt-\Jh\|_{\XX_{s_{max}-1}}.
\eeaa 
Hence, by \eqref{defXX} and for $\epg$ small enough, we have
\bea\label{contractionT}
\|T(\Jt)-T(\Jh)\|_{\XX_{s_{max}-1}}\leq \f12 \|\Jt-\Jh\|_{\XX_{s_{max}-1}},
\eea
which implies that $T$ is a contraction map on $(\XX,\|\cdot\|_{\XX_{s_{max}-1}})$.\\ \\
\noindent{\bf Step 4.} Fixed point of $T$.\\ \\
Since $\XX$ is a closed subset of $\XX_{s_{max}-1}$, we deduce that $(\XX,\|\cdot\|_{\XX_{s_{max}-1}})$ is a complete metric space. Then, applying Banach fixed-point theorem to $T:\XX\rightarrow\XX$, we obtain a unique fixed point $\Jt\in\XX$ of $T$, i.e.
\bea \label{fixedpointT}
T(\Jt)=\Jt.
\eea
Injecting \eqref{fixedpointT} into \eqref{defTJ}, we obtain
\beaa 
\nu^\Jt_\#(\Jt)&=&0,\qquad\qquad \mbox{ on }\Si_\#,\\
\Jt^{(p)}&=&\Jp[\S_0],\quad \;\mbox{ on }\ovS.
\eeaa 
Finally, we have
\beaa 
\nu^\Jt \left( (\Phi^\Jt)^{-1})^\# (\Jt)\right)=(\Phi^\Jt)_\#(\nu^\Jt_\#)\left( (\Phi^\Jt)^{-1})^\# (\Jt)\right) =\nu^\Jt_\#(\Jt)=0,
\eeaa 
which implies \eqref{desiredJt}. This concludes the proof of Theorem \ref{generalcoro}.
\end{proof}
\begin{definition}\label{hypersurfaceSigma}
Let $P(s):=(\Psi(s),s,\La(s),\Lab(s))$ be a curve satisfying \eqref{curvecon}. By Theorem \ref{generalcoro}, we define a triplet of functions $\Jt^{(p)}$ and a hypersurface of the form
\bea\label{hyper}
\Si &=& \bigcup_{s\geq \ovs} \S[P(s)]=\bigcup_{s\geq\ovs}\S[\Psi(s),s,\La(s),\Lab(s),\Jt^{(p)}].
\eea
We define the following family of triplets on $\Si$:
\bea\label{jpsjpsjps}
\JpS:=(\Phi^{-1})^\#(\Jt^{(p)}).
\eea 
According to the construction in \eqref{hyper}, there is an $\S$--adapted null frame $(e_3^\S,e_4^\S,e_1^\S,e_2^\S)$ on every GCM sphere $\S[P(s)]$.

We also assume the transversality conditions:
\bea\label{transver} 
\xi^\S=0,\qquad \om^\S=0,\qquad \etab^\S=-\ze^\S,\quad \mbox{ on }\Si,
\eea
and
\bea\label{e4ue4r}
e_4^\S(\rS)=1,\qquad e_4^\S(\uS)=0, \quad \mbox{ on }\Si,
\eea 
where $\uS$ is defined as
\bea\label{defu}
\uS&:=&c_0-\rS\quad \mbox{on }\Si,\qquad \uS\big|_{\S_0}=\ovu.
\eea
\end{definition}
\begin{remark}\label{rem4.5}
Let us provide below justifications for introducing the transversality conditions \eqref{transver} and \eqref{e4ue4r}:
\begin{enumerate}
\item \eqref{transver} and \eqref{e4ue4r} are consistent with a local extension by an outgoing geodesic foliation initialized on $\Si$. The use of transversality conditions instead of a local extension is chosen here to have intrinsic definitions on $\Si$.
\item The role of the transversality conditions \eqref{transver} is to make sense of the Ricci coefficients $\etaS$, $\xibS$ and $\ombS$ on $\Si$ through the formulae:\footnote{Note that the L.H.S. of \eqref{trans1} are well defined on $\Si$ since $\nu^\S$ is tangent to $\Si$.}
\begin{align}
\begin{split}\label{trans1}
\g(\D_{\nu^\S}e_4^\S,e_a^\S)&=2\etaS_a+2\bS\xi_a^\S=2\etaS_a,\\
\g(\D_{\nu^\S}e_3^\S,e_a^\S)&=2\xib^\S_a+2\bS\etab^\S_a=2\xibS_a-2\bS\zeS_a,\\
\g(\D_{\nu^\S}e_3^\S,e_4^\S)&=4\ombS-4\bS\omS=4\ombS.
\end{split}
\end{align}
\item The role of the transversality conditions \eqref{e4ue4r} is to make sense of $e_3^\S(\rS)$ and $e_3^\S(\uS)$ on $\Si$ through the formulae:
\begin{align}
\begin{split}\label{trans2}
    e_3^\S(\rS)&=\nu^\S(\rS)-\bS e_4^\S(\rS)=\nu^\S(\rS)-\bS,\\
    e_3^\S(\uS)&=\nu^\S(\uS)-\bS e_4^\S(\uS)=\nu^\S(\uS).
\end{split}
\end{align}
\end{enumerate}
\end{remark}
We have the following proposition.
\begin{proposition}\label{propositionSigma}
Let $\Si$ a hypersurface given by Definition \ref{hypersurfaceSigma}. Then, the following properties hold:
\begin{enumerate}
\item The curve $\ga(s)=(\Psi(s),s,0,0)$ of South Poles of the background spheres $S(\Psi(s),s)$ verifies $\ga(s)\subset\Si$, $s\in\ovI$. 
\item On $\S$, the following GCM conditions hold
\bea
\begin{split}
\label{GCMtj}
\ka^\S&=\frac{2}{r^\S},\\
\kab^\S&=-\frac{2}{r^\S}\Up^\S+\Cb^\S_0+\sum_p \CbpS\JpS,\\
\mu^\S&=\frac{2m^\S}{(r^\S)^3}+M^\S_0+\sum _p\MpS \JpS,
\end{split}
\eea
where $\JpS$ is defined by \eqref{jpsjpsjps}. Moreover,
\bea\label{GCMlalab}
(\div^\S f)_{\ell=1}=\La,\qquad (\div^\S \fb)_{\ell=1}=\Lab,
\eea
where $(f,\fb,\la)$ are the transition parameters of the frame transformation from the background frame $(e_3,e_4,e_1,e_2)$ to the adapted frame $(e_3^\S,e_4^\S,e_1^\S,e_2^\S)$ and the $\ell=1$ modes are defined w.r.t. $\JpS$. The triplets of constants $\La^\S,\Lab^\S$ depend smoothly on the surfaces $\S$ and
\beaa
\La^{\S_0}=\La_0,\qquad \Lab^{\S_0}=\Lab_0.
\eeaa
\item There are two maps $\Xi_{S}:\RR_{S} \to \S$ and $\Xi_{N}:\RR_{N} \to \S$ given by
\bea
\begin{split}\lab{XiSN}
&\Xi_{S,N}:(u,s,y_{S,N}^1,y_{S,N}^2)\mapsto \\ &\left(u+U_{S,N}(y_{S,N}^1,y_{S,N}^2,u,s,\La,\Lab),s+S_{S,N}(y_{S,N}^1,y_{S,N}^2,u,s,\La,\Lab),y_{S,N}^1,y_{S,N}^2\right)
\end{split}
\eea
with $U_S,S_S$ vanishing at South Pole and verify the following transition condition
\beaa
\Xi_N=\Xi_S \circ \varphi_{NS},
\eeaa
where $\varphi_{NS}$ is the transition map between the 2 coordinates charts. Using these two maps, we can define a map $\Xi: S(u,s) \rightarrow \S$ by
\bea
\Xi=\Xi_{S} \;\mbox{ on }\RR_S,\qquad \Xi=\Xi_N \;\mbox{ on } \RR_{N}.
\eea
\item Let $\nu^\S$ be the unique vectorfield tangent to the hypersurface $\Sigma$, normal to $\S$, and normalized by $\g(\nu^\S,e_4^\S)=-2$. There exists a unique scalar function $b^\S$ on $\Sigma$ such that $\nu^\S$ is given by
\beaa
\nu^\S=e_3^\S+b^\S e_4^\S.
\eeaa
\item The transversality conditions hold:
\beaa
\xi^\S=0,\qquad \om^\S=0,\qquad \etab^\S=-\ze^\S,\quad \mbox{ on }\Si,
\eeaa
and 
\beaa
e_4^\S(\rS)=1,\qquad e_4^\S(\uS)=0,\quad \mbox{ on }\Si.
\eeaa
\item The Ricci coefficients $\ks^\S,\ksb^\S,\chih^\S,\chibh^\S,\zeS$ are well defined on each sphere $\S$ of $\Sigma$, and hence on $\Sigma$. The same holds true for all curvature coefficients $\a^\S,\b^\S,\rho^\S,\rhod^\S,\bb^\S,\aa^\S$. Taking into account our transversality conditions \eqref{transver} on $\Si$, we remark that all the Ricci coefficients are well defined on $\Si$ including $\etaS$, $\xibS$ and $\ombS$, see item 2 in Remark \ref{rem4.5}.
\item The $\ell=1$ modes $\JpS$ verify:
\beaa 
\nu^\S(\JpS)=0,\qquad\JpS\big|_{\S_0}=J^{(\S_0,p)},\qquad p=0,+,-.
\eeaa 
\item The hypersurface $\Si$ constructed above is a smooth hypersurface. 
\end{enumerate}
\end{proposition}
\begin{remark}
Since we always work on the hypersurface $\Si$, with no risk of confusion, we denote
\bea\label{defXi} 
\Xi_{S,N}(s,y_{S,N}^1,y_{S,N}^2)&:=&\Xi_{S,N}(\Psi(s),s,y_{S,N}^1,y_{S,N}^2)
\eea
in the following context.
\end{remark}
\begin{proof}[Proof of Proposition \ref{propositionSigma}]
Properties 1-7 are immediate consequences of Definition \ref{hypersurfaceSigma}. So we only prove Property 8. For this purpose, we first prove the smoothness of the functions $\Xi_{p}$ for $p=S,N$. Notice that
\begin{align}
\begin{split}\label{smooth}
\pr_s\Xi_p(s,y^1_p,y^2_p)&=\left(\Psi'(s)+\pr_P U_p(\c)P'(s),1+\pr_P S_p(\c)P'(s),0,0\right),\\
\pr_{y^a_p} \Xi_p(s,y^1_p,y^2_p)&=\left(\pr_{y^a_p}U_p(\c),\pr_{y^a_p}S_p(\c),\de_{a1},\de_{a2}\right),
\end{split}
\end{align}
where
\beaa
\pr_P U_p(\c)P'(s)&=&\Psi'(s)\pr_u U_p(\c)+\pr_s U_p(\c)+\La'(s)\pr_\La U_p(\c)+\Lab'(s)\pr_\Lab U_p(\c),\\
\pr_P S_p(\c)P'(s)&=&\Psi'(s)\pr_u S_p(\c)+\pr_s S_p(\c)+\La'(s)\pr_\La S_p(\c)+\Lab'(s)\pr_\Lab S_p(\c).
\eeaa
Thus, in view of the smoothness of $U_p,S_p$ w.r.t. the parameters $\La,\Lab$ and $u,s$, which are statements 7-9 in Theorem \ref{Theorem:ExistenceGCMS1} and Corollary 6.11 in \cite{KS:Kerr1}, we deduce that $\Xi_p$ for $p=S,N$ are smooth 
and they are immersions at every point. Recall that
\beaa
\Xi_p(s,y_p^1,y_p^2)&=&\left(\Psi(s)+U(y_p^1,y_p^2,P(s)),s+S(y_p^1,y_p^2,P(s)),y_p^1,y_p^2\right).
\eeaa
According to \eqref{smooth} and the implicit function theorem, at every point $(s,y_p^1,y_p^2)$, its inverse map is locally given by
\bea\label{inverseXiS}
\Xi_p^{-1}(h(z_p,y_p^1,y_p^2),z_p,y_p^1,y_p^2)=(s(z_p,y_p^1,y_p^2),y_p^1,y_p^2),
\eea
where $h(z_p,y_p^1,y_p^2),s(z_p,y_p^1,y_p^2)$ are smooth functions.

Now we recall the following result, see Theorem 2.1.2 in \cite{berger} for a proof.
\begin{theorem}\label{submanifold}
$V$ is a $d$-dimensional smooth submanifold of $\mathbb{R}^n$ if for every point $x\in V$, there exists an open neighborhood $U\subset\mathbb{R}^n$ of $x$, an open domain $\Omega\subset \mathbb{R}^d$ and a map $g:\Omega\to \mathbb{R}^n$ such that
\begin{enumerate}
\item $g$ is a smooth map.
\item $g$ is homeomorphism from $\Om$ to $U\cap V$.
\item $g$ is an immersion at every point of $\Om$.
\end{enumerate}
\end{theorem}
For every $x\in \Si$, without loss of generality, we assume that $x$ can be written as
\begin{equation*}
x=\Xi_S(s_0,y^1_0,y^2_0)=\left(\Psi(s_0)+U(y_0^1,y_0^2,P(s_0)),s_0+S(y_0^1,y_0^2,P(s_0)),y_0^1,y_0^2\right)
\end{equation*}
in the South coordinate chart. We define
\begin{equation*}
    U_x:=B(x,\varepsilon):=\left\{q\in\RRR^4\Big/\, \|q-x\|_{\RRR^4}\leq \varepsilon\right\}.
\end{equation*}
Then, for every\footnote{By taking $\vep$ small enough, $q$ can be written in the South coordinate chart.}
\begin{equation*}
q=(\Psi(s)+U(y^1,y^2,P(s)),s+S(y^1,y^2,P(s)),y^1,y^2)\in U_x\cap \Si,
\end{equation*}
we have
\beaa
\left|s+S(y^1,y^2,P(s))- s_0-S(y_0^1,y_0^2,P(s_0))\right|,\; |y^1-y^1_0|,\; |y^2-y^2_0|&\leq& \varepsilon.
\eeaa
Thus, in view of the smallness dependence of $S$ w.r.t. the parameters $(s,y^1,y^2)$, we deduce that $q\in\NN$ where
\beaa
\NN:= \left\{(s,y^1,y^2)\in\RRR^3/\, |s-s_0|, \, |y^1-y^1_0|,\, |y^2-y^2_0|\leq  2\vep\right\}.
\eeaa 
For $\vep$ small enough, $\Xi_S$ is a homeomorphism from $\NN$ to its image. We then denote
\beaa
\Om := \Xi_S^{-1} (U_x\cap \Si)\subset \NN,
\eeaa
where $\Xi_S^{-1}$ is defined in \eqref{inverseXiS}. Hence, the map
\beaa
\Xi_S \big|_{\Om} : \Om &\mapsto& U_x\cap \Si,
\eeaa 
is smooth, homeomorphism and is an immersion. The similar conclusion holds for $\Xi_N$. Applying Theorem \ref{submanifold}, we deduce that $\Si$ is a smooth hypersurface as stated. This concludes the proof of Proposition \ref{propositionSigma}.
\end{proof}
\subsection{Estimates for the transition functions \texorpdfstring{$(f,\fb,\la)$}{}}
Let $\Si$ be a smooth spacelike hypersurface defined as in Definition \ref{hypersurfaceSigma}. We define the scalars\footnote{Recall that \eqref{e4ue4r} allows to make sense of $e_3^\S(\uS)$ and $e_3^\S(\rS)$, see item 3 in Remark \ref{rem4.5}.}
\bea\label{zSOmbSfirst}
\zS:=e_3^\S(u^\S),\qquad \Omb^\S:=e_3^\S(\rS),
\eea
and the renormalized quantities
\bea
\zcS:=\zS-2,\qquad \OmbcS:=\OmbS+\UpS.
\eea
We define the following quantities along $\Si$
\bea\label{BBbDfirst}
\BS:= (\divS \etaS)_{\ell=1}, \qquad \BbS:=(\divS \xibS)_{\ell=1}, \qquad \DS:=\ov{\bcS}=\ov{\bS}+1+\frac{2\mS}{\rS},
\eea
where the $\ell=1$ modes are defined w.r.t. $\JpS$.

The goal of this section is to prove Proposition \ref{th24}. To this end, we first prove a series of lemmas which will be used throughout this paper.
\begin{lemma}\label{resume}
We have the following identities:
\begin{align}
\begin{split}
e_a^\S(\zS )&=(\zeS_a-\etaS_a)\zS,\\
e^\S_a(\OmbS) &= (\zeS_a-\etaS_a)\OmbS-\xib_a^\S,\\
\bS &=-\zS-\OmbS. \label{eqresume}
\end{split}
\end{align}
\end{lemma}
\begin{proof}
By Lemma \ref{comm}, we have
\beaa
[e_a,e_3]=\frac{1}{2} \ksb e_a+\chibh_{ab} e_b-\nab_{3}e_a +(\ze_a-\eta_a)e_3 -\xib_a e_4.
\eeaa
Since $e_a^\S(u^\S)=e_a^\S(r^\S)=0$, $e_4^\S(\rS)=1$ and $e_4^\S(\uS)=0$, we deduce that
\beaa
e_a^\S (e_3^\S (u^\S) )&=&[e_a^\S,e_3^\S]u^\S=\left[\frac{1}{2}\ksb^\S e^\S_a+\chibh_{ab}^\S e^\S_b-\nab_{e_3^\S}e_a^\S+(\ze^\S_a-\eta^\S_a)e^\S_3-\xib^\S_a e^\S_4\right]u^\S \\
&=&(\ze_a^\S-\eta_a^\S) e_3^\S (u^\S), \\
e_a^\S (e_3^\S (r^\S) )&=&[e_a^\S,e_3^\S]r^\S=\left[\frac{1}{2}\ksb^\S e^\S_a+\chibh_{ab}^\S e^\S_b-\nab_{e_3^\S}e_a^\S+(\ze^\S_a-\eta^\S_a)e^\S_3-\xib^\S_a e^\S_4\right]r^\S \\
&=&(\ze_a^\S-\eta_a^\S) e_3^\S (r^\S)-\xib_a^\S.
\eeaa
We deduce, in view of the definition of $\zS$ and $\OmbS$,
\begin{align}
\begin{split}\label{zombsea}
e^\S_a(z^\S) &= e^\S_a(e^\S_3(\uS))=(\zeS_a-\etaS_a)\zS,\\
e^\S_a(\OmbS) &= e^\S_a (e^\S_3 (\rS))=(\zeS_a-\etaS_a)\OmbS-\xib_a^\S,
\end{split}
\end{align}
which implies the first two identities in \eqref{eqresume}.

Next, we prove the last identity in \eqref{eqresume}. In view of the definition of $\nu^\S$ and $\zS$ we make use of \eqref{e4ue4r} to deduce
\beaa
\nu^\S (\uS) =e_3^\S(\uS)+b^\S e_4^\S(\uS)=z^\S.
\eeaa
On the other hand, since $u^\S=c_0-r^\S$ along $\Si$, we have
\bea\label{zbO}
z^\S=\nu^\S (\uS)=-\nu^\S (\rS)=-e_3^\S(r^\S)-\bS e_4^\S (r^\S) =-\OmbS -\bS,
\eea
where we used \eqref{e4ue4r} and \eqref{zSOmbSfirst}. This concludes the proof of Lemma \ref{resume}.
\end{proof}
\begin{lemma}\label{dint}
For every scalar function $h$, we have the formula
\bea
\nu^\S \left(\int_\S h \right)&=& z^\S \int_\S (z^\S)^{-1} \left(\nu^\S(h) + (\ksb^\S +b^\S \ks^\S) h \right).\lab{dinth}
\eea
In particular
\beaa
\nu^\S (\rS) &=& \frac{\rS}{2}z^\S\ov{(z^\S)^{-1}(\ksb^\S +b^\S \ks^\S)},
\eeaa
where the average is with respect to $\S$.
\end{lemma}
\begin{proof}
    See Lemma 5.2.2 in \cite{KS:main}.
\end{proof}
\begin{lemma}\label{nunununu}
We have the following properties for the transition coefficients $(f,\fb,\la)$:
\begin{align}
\begin{split}\label{nuffb}
\nab^\S_{\nu^\S}(f)=&2(\etaS-\eta)-\f12\ka(\fb+\bS f)+2f\omb+F\cdot \Ga_b+\lot ,\\
\nab^\S_{\nu^\S}(\fb)=&2(\xibS-\xib)-\f12\fb\kab-2\omb(\fb-\bS f)+\bS\left(2\nab^\S (\la)-\f12 \fb\ka\right)\\
&+F\cdot\Ga_b+\lot,\\
\nabS_{\nu^\S}(\la)=&2(\ombS-\omb)-2\omb\ovla+F\c\Ga_b+\lot,
\end{split}
\end{align}
where $\lot$ denotes terms which are linear in $\Ga_g,\Ga_b$ and linear and higher order in $F$.\footnote{Here $\lot$ is different from that of Proposition \ref{Prop:transformation-formulas-generalcasewithoutassumptions}.}
\end{lemma}
\begin{proof}
Schematically, from the transformation formulas for $\eta$, $\xib$ and $\omb$ in Proposition \ref{Prop:transformation-formulas-generalcasewithoutassumptions}, we have
\beaa
\nab_3^\S (f)&=&2(\etaS-\eta)-\f12\fb\ka+2\omb f+F\cdot \Ga_b +\lot,\\
\nab_3^\S (\fb)&=&2(\xibS-\xib)-\f12\fb\kab-2\omb\fb+F\cdot \Ga_b +\lot,\\
\nab_3^\S (\la) &=&2(\ombS-\omb)-2\omb\ovla+F\c\Ga_b+\lot,
\eeaa
where $F=(f,\fb,\ovla)$ and $\lot$ denotes terms which are linear in $\Ga_g,\Ga_b$ and linear and higher order in $F$. Recall also that the $e_4^\S$ derivatives of $F$ are fixed by our transversality condition \eqref{transversalityru}. More precisely we have from \eqref{transversalityru} and the transformation formulas of $\xi$, $\om$, $\etab$ and $\ze$ of Proposition \ref{Prop:transformation-formulas-generalcasewithoutassumptions}
\begin{align}
\begin{split}\label{e4ffbla}
\nab_4^\S(f)&= -\f12\ka f+F\cdot \Ga_b+\lot ,\\
\nab_4^\S(\fb)&= 2\nab^\S(\la)+2\omb f -\f12 \fb\ka+F\cdot \Ga_b+\lot,\\
\nab_4^\S(\la)&=F\c\Ga_b +\lot
\end{split}
\end{align}
Using $\nu^\S=e_3^\S +\bS e_4^\S$, we obtain \eqref{nuffb}. This concludes the proof of Lemma \ref{nunununu}.
\end{proof}
\begin{lemma}\label{tildelemma}
We have, for any scalar function $h$ on $\RR$, and any $1\leq l\leq s$,
\beaa
\|({\nu^\S})^l h\|_{\hk_{s-l}(\S)} \les r\sup_{\RR}\Big(|\widetilde{\dk}^{\leq s}h|+\dg|\dk^{\leq s}h|\Big)+\|(\nabS_{\nu^\S})^{\leq l-1}(F,\bS-b)\|_{\hk_{s-l}(\S)}\sup_{\RR}|\dk^{\leq s}h|,
\eeaa
where 
\beaa 
\dkt:=(e_3-(z+\Omb)e_4,\dkb),\qquad b:=-z-\Omb.
\eeaa
\end{lemma}
\begin{proof}
We denote
\beaa 
\nu_\RR:=e_3+b e_4.
\eeaa 
By definition, we have
\beaa
\nu^\S(h)-\nu_\RR(h)&=&({e_3^\S}-{e_3})h+\bS({e_4^\S}-{e_4})h+(\bS-b){e_4}(h).
\eeaa
Applying Lemma \ref{Lemma:Generalframetransf}, we have
\beaa
{e_4^\S}h-{e_4}h=O(F)(\dk^{\leq 1}h),\qquad e_3^\S h-e_3 h=O(F)(\dk^{\leq 1}h).
\eeaa
Thus, we obtain
\beaa
\nu^\S(h)-\nu_\RR(h)&=&O(F)\dk^{\leq 1}(h)+O(\bS-b)\dk^{\leq 1}(h).
\eeaa
More generally, for $1\leq l\leq s$ we have
\beaa
({\nu^\S})^l(h)-({\nu_\RR})^l(h)&=& O\left((\nabS_{\nu^\S})^{\leq l-1}(F,\bS-b)\right)\dk^{\leq l}(h),
\eeaa
which implies
\beq\label{cha1}
\|({\nu^\S})^l(h)-({\nu_\RR})^l(h)\|_{\hk_{s-l}(\S)}\les \|(\nabS_{\nu^\S})^{\leq l-1}(F,\bS-b)\|_{\hk_{s-l}(\S)}\sup_{\RR}|\dk^{\leq s}h|.
\eeq
Next, applying \eqref{eq:Prop:comparison2}, we have
\begin{align}
\begin{split}\label{cha3}
\|({\nu_\RR})^l(h)\|_{\hk_{s-l}(\S)}&\les r\sup_{\RR}\left(|\dkb^{\leq s-l}({\nu_\RR})^l(h)|+\dg |\dk^{\leq s-l}({\nu_\RR})^l(h)|\right)\\
&\les r\sup_{\RR}\left(|\dkt^{\leq s}h|+\dg |\dk^{\leq s}h|\right).
\end{split}
\end{align}
Combining \eqref{cha1} and \eqref{cha3}, we deduce
\beaa
\|({\nu^\S})^l h\|_{\hk_{s-l}(\S)} \les r\sup_{\RR}\Big(|\widetilde{\dk}^{\leq s}h|+\dg|\dk^{\leq s}h|\Big)+\|(\nabS_{\nu^\S})^{\leq l-1}(F,\bS-b)\|_{\hk_{s-l}(\S)}\sup_{\RR}|\dk^{\leq s}h|
\eeaa
as stated. This concludes the proof of Lemma \ref{tildelemma}.
\end{proof}
In the remainder of this section, we denote $K:=s_{max}+1$ and 
\begin{equation*}
    \nu:=\nu^\S=e_3^\S+\bS e_4^\S.
\end{equation*}
\begin{lemma}\label{lemmacorollary}
Under the assumption \eqref{eq:GCM-improved estimate2-again}, we have for $0\leq l\leq K$
\begin{align}
\begin{split}\label{resultkakabmu}
\left\|\nu^l(\kadot,\kabdot)\right\|_{\hk_{K-l}(\S)}&\les r^{-1}\dg+ r^{-1}\epg|\DS|+r^{-1}\epg\sum_{j=1}^{l}\|(\nabS_\nu)^j (F)\|_{\hk_{K-j}(\S)},\\ \left\|\nu^l(\mudot)\right\|_{\hk_{K-l}(\S)}&\les r^{-2}\dg+ r^{-2}\epg|\DS|+r^{-2}\epg\sum_{j=1}^{l}\|(\nabS_\nu)^j (F)\|_{\hk_{K-j}(\S)}.
\end{split}
\end{align}
\end{lemma}
\begin{proof}
Notice that we have from \eqref{USdepend} and \eqref{curvecon}
\begin{equation*}
    \zS-z ,\; \bS-b =O(1),
\end{equation*}
which implies from Lemma \ref{resume} that
\begin{equation}\label{bzobounded}
    |\bS|,\, |\zS|,\, |\OmbS|\les 1.
\end{equation}
Applying Lemma \ref{resume} and \eqref{bzobounded}, we infer that $(\zS-\ov{\zS})-(z-\ov{z})$ and $(\OmbS-\ov{\OmbS})-(\Omb-\ov{\Omb})$ can be estimated by $\etaS-\eta$ and $\xibS-\xib$. Thus, these quantities can be controlled by $\nabS_\nu (f,\fb)$ thanks to Lemma \ref{nunununu} and hence 
\bea\label{zwidecheckEE}
\left\|(\zS-\ov{\zS})-(z-\ov{z}),(\OmbS-\ov{\OmbS})-(\Omb-\ov{\Omb})\right\|_{\hk_{K}(\S)}\les \dg+r\|\nabS_\nu (f,\fb)\|_{\hk_{K-1}(\S)}.
\eea
Combining \eqref{zwidecheckEE} and Lemma \ref{resume}, we obtain
\bea\label{bSbdifference}
\left\|(\bS-\ov{\bS})-(b-\ov{b})\right\|_{\hk_{K}(\S)}\les r\dg+r\|\nabS_\nu(f,\fb)\|_{\hk_{K-1}(\S)}.
\eea 
Recalling \eqref{vsiOmbass} and Lemma \ref{lemma:comparison-gaS-ga}, we deduce
\bea\label{bSbov}
\ov{\bS}-\ov{b}=\DS-1-\frac{2\mS}{\rS}+1+\frac{2m_{(0)}}{r}=\DS+r^{-1}O(\dg).
\eea
Thus, we have
\bea\label{bSb}
\|\bS-b\|_{\hk_{K}(\S)}\les r\dg+r|\DS|+r\|\nabS_\nu(f,\fb)\|_{\hk_{K-1}(\S)}.
\eea
Similarly, we have
\begin{equation}\label{bSbl}
    \|(\nabS_\nu)^{\leq l-1}(\bS-b)\|_{\hk_{K-l-1}(\S)}\les r\dg+r|\DS|+r\sum_{j=1}^l\|(\nabS_\nu)^j(f,\fb)\|_{\hk_{K-j}(\S)}.
\end{equation}
Applying Proposition \ref{nullstructure}, we deduce
\beaa
e_4\kadot&=&-\frac{2}{r}\kadot+\Ga_g\c\Ga_g=r^{-3}O(\epg),\\
e_4\kabdot&=&-\frac{1}{r}\kabdot-2\div\ze+2\rhoc+\Ga_b\c\Ga_g=r^{-3}O(\epg),\\
e_4\mudot&=&e_4\left(-\div\ze-\rhoc+\Ga_b\c\Ga_g\right)=r^{-4}O(\epg).
\eeaa
Recalling \eqref{eq:GCM-improved estimate2-again}, we obtain
\bea\label{lastestimateinmy}
\dk^{\leq s}(\kadot,\kabdot)=r^{-2}O(\epg),\qquad \dk^{\leq s}\mudot=r^{-3}O(\epg).
\eea 
Applying Lemma \ref{tildelemma} with $h=(\kadot,\kabdot)$ and $s=K$, combining with \eqref{eq:GCM-improved estimate2-again}, \eqref{bSbl} and \eqref{lastestimateinmy}, we infer
\begin{align*}
\|\nu^l(\kadot,\kabdot)\|_{\hk_{K-l}(\S)}&\les r\sup_{\RR}\Big(|\widetilde{\dk}^{\leq s}(\kadot,\kabdot)|+\dg|\dk^{\leq s}(\kadot,\kabdot)|\Big)+\|(\nabS_{\nu^\S})^{\leq l-1}(F,\bS-b)\|_{\hk_{K-l}(\S)}\sup_{\RR}|\dk^{\leq s}(\kadot,\kabdot)|,\\
&\les r^{-1}\dg+r^{-2}\epg\left(r\dg+ r|\DS|+ r\sum_{j=1}^l\|(\nabS_\nu)^jF\|_{\hk_{K-j}(\S)}\right)\\
&\les r^{-1}\dg+ r^{-1}\epg|\DS|+ r^{-1}\epg \sum_{j=1}^l\|(\nabS_\nu)^jF\|_{\hk_{K-j}(\S)},
\end{align*}
which implies the first estimate in \eqref{resultkakabmu}. The second estimate in \eqref{resultkakabmu} is similar. This concludes the proof of Lemma \ref{lemmacorollary}.
\end{proof}
\begin{lemma}\label{GCMuse}
In view of the GCM conditions \eqref{def:GCMC}, the system \eqref{Generalizedsystem} can be written in the following form:
\bea
\begin{split}\label{eqGCMl}
\curl ^\S f &= h_1 -\ov{h_1}^\S,\\
\curl^\S \fb&= \underline{h}_1 - \ov{\underline{h}_1}^\S,\\
\div^\S f + \frac{2}{r^\S} \ovla -\frac{2}{(r^\S)^2}\ovb &=h_2,\\
\div^\S\fb + \frac{2}{r^\S} \ovla +\frac{2}{(r^\S)^2}\ovb
&=   \Cbdot_0+\sum_p \Cbpdot \JpS +\underline{h}_2,\\
\left(\Delta^\S+\frac{2}{(r^\S)^2}\right)\ovla  &=  \Mdot_0+\sum _p\Mpdot \JpS+\frac{1}{2r^\S}\left(\Cbdot_0+\sum_p \Cbpdot \JpS\right) +h_3,\\
\Delta^\S\ovb-\frac{1}{2}\div^\S\Big(\fb - f\Big) &= h_4 -\ov{h_4}^\S , \qquad \ov{\ovb}^\S=b_0,
\end{split}
\eea
where
\beaa 
\ovb&:=&r-\rS,
\eeaa 
and
\bea
\begin{split}\label{hhhhh}
&\left\|\nu^l(h_1,\hb_1,h_2,\hb_2,h_4)\right\|_{\hk_{s}(\S)}+r\left\|\nu^l(h_3)\right\|_{\hk_{s}(\S)}\\
\les& r^{-1}\dg+(\epg r^{-1}+r^{-2})\left(\left\|(\nabS_{\nu})^{l}(F)\right\|_{\hk_{s}(\S)}+r^{-1}\big\|\nu^{l}(\ovb)\big\|_{\hk_{s}(\S)}\right)\\
+&r^{-1}\left(\left\|(\nabS_{\nu})^{\leq l-1}(F)\right\|_{\hk_{s}(\S)}+r^{-1}\big\|\nu^{\leq l-1}\ovb\big\|_{\hk_{s}(\S)}\right)+r^{-2}\dg\|\nu^{\leq l-1}(\bS)\|_{\hk_{s}(\S)}.
\end{split}
\eea 
\end{lemma}
\begin{proof}
See Remark 4.10 in \cite{KS:Kerr1} for the proof of \eqref{eqGCMl} and the explicit expressions of $(h_1,\hb_1,h_2,\hb_2,h_3,h_4)$. Remark that \eqref{hhhhh} is a direct consequence of GCM conditions \eqref{def:GCMC}, Remark 4.11 in \cite{KS:Kerr1} and Lemma \ref{lemmacorollary}.
\end{proof}
\begin{proposition}\label{th24}
We have the following estimate
\bea\label{nuFdg}
\Vert F\Vert_{\hk_{s_{max}+1}(\S)}+\Vert\nab^\S_{\nu} F\Vert_{\hk_{s_{max}}(\S)} &\les& \dg+|\BS|+|\BbS|+|\DS|.
\eea
\end{proposition}
\begin{proof}
In view of the construction of $\Si$ and \eqref{eq:ThmGCMS1}, for every $\S\subset\Si$, we have
\bea\lab{Fest}
\Vert F\Vert_{\hk_{s_{max}+1}(\S)}\les\dg.
\eea
To derive the remaining tangential derivatives of $F$ along $\Si_0$, we commute the GCM system \eqref{eqGCMl} with respect to $\nabS_\nu$. Applying $\nu(\JpS)=0$ and $|\nu(\rS)|\les 1$\footnote{Notice that we have from \eqref{USdepend} and \eqref{curvecon} that $\nu(\rS)=O(1)$.}, we obtain
\bea
\begin{split}\label{nuGCM}
\curl ^\S \nabS_\nu (f) &=\nabS_\nu( h_1 -\ov{h_1}^\S)+[\curl^\S,\nabS_\nu]f,\\
\curl^\S \nabS_\nu(\fb)&=\nabS_\nu(\hb_1 -\ov{\underline{h}_1}^\S)+[\curl^\S,\nabS_\nu]\fb,\\
\div^\S \nabS_\nu(f)+\frac{2}{r^\S}\nabS_\nu(\ovla)-\frac{2}{(r^\S)^2}\nabS_\nu(\ovb) &=\nabS_\nu(h_2)+[\divS,\nabS_\nu]f+r^{-2}O(F)+r^{-3}O(\ovb),
\\
\div^\S \nabS_\nu(\fb)+\frac{2}{r^\S}\nabS_\nu(\ovla)+\frac{2}{(r^\S)^2}\nabS_\nu(\ovb)
&=\nu(\Cbdot_0)+\sum_p \nu(\Cbpdot) \JpS+\nabS_\nu(\hb_2)\\
&+[\divS,\nabS_\nu]\fb+r^{-2}O(F)+r^{-3}O(\ovb),\\
\left(\Delta^\S+\frac{2}{(r^\S)^2}\right)(\nabS_\nu\ovla) &=\nu(\Mdot_0)+\sum _p\nu(\Mpdot) \JpS\\
&+\frac{1}{2r^\S}\left(\nu(\Cbdot_0)+\sum_p \nu(\Cbpdot) \JpS\right)+\nabS_\nu(h_3)\\
&+[\De^\S,\nabS_\nu]\ovla+r^{-3}O(F)+r^{-2}O(\Cbdot_0,\Cbpdot),\\
\Delta^\S(\nabS_\nu\ovb)-\frac{1}{2}\div^\S\Big(\nabS_\nu\fb-\nabS_\nu f\Big)&=\nabS_\nu(h_4-\ov{h_4}^\S) +[\De^\S,\nabS_\nu]\ovb+[\divS,\nabS_\nu](F),\\
b_0^{(1)}&:=\ov{\nu\ovb}^\S.
\end{split}
\eea
Recalling that
\beaa 
F=r^{-1}O(\dg),\qquad \ovb=O(\dg),
\eeaa 
and using Lemma \ref{commfor}, we have
\beaa 
\left([\divS,\nabS_\nu]F,\,[\curl^\S,\nabS_\nu]F\right)&=&r^{-3}O(\dg)+r^{-1}O(\epg)\nabS_\nu(F),
\eeaa
and
\begin{align*}
 [\De^\S,\nabS_\nu] F&= r^{-4}O(\dg)+r^{-1}\dkb^{\leq 1}(\Ga_b\c\nabS_\nu F)=r^{-4}O(\dg)+r^{-2}O(\epg)\dkb^{\leq 1}(\nabS_\nu(F)),\\
[\De^\S,\nabS_\nu]\ovb &= r^{-3}O(\dg)+r^{-1}\dkb^{\leq 1}(\Ga_b\c\nabS_\nu\ovb)=r^{-3}O(\dg)+r^{-2}O(\epg)\dkb^{\leq 1}(\nabS_\nu(\ovb)).
\end{align*}
Moreover, applying Proposition \ref{Thm.GCMSequations-fixedS:contraction} to \eqref{eqGCMl}, we deduce
\bea\label{coefficientsestimates}
(\Cbdot_0,\Cbpdot)=r^{-2}O(\dg).
\eea
Thus, applying Proposition \ref{Thm.GCMSequations-fixedS:contraction} to \eqref{nuGCM}, in view of \eqref{hhhhh}, we infer
\bea
\begin{split}\label{ffbovlac} 
\left\Vert \nab^\S_{\nu}(f,\fb),\nu(\ovla)-\ov{\nu(\ovla)}^\S\right\Vert_{\hk_{K-1}(\S)}&\les\dg+\left|\left(\divS\nab^\S_{\nu}f\right)_{\ell=1}\right|+\left|\left(\divS\nab^\S_{\nu}\fb\right)_{\ell=1}\right|\\
&+(\epg+r^{-1})\left(\left\Vert\nab^\S_{\nu}(F)\right\Vert_{\hk_{K-1}(\S)}+r^{-1}\big\Vert\nu(\ovb)\big\Vert_{\hk_{K-1}(\S)}\right),
\end{split}
\eea 
and
\bea
\begin{split}\label{ovovla}
r\left|(\ov{\nu\ovla}^\S)\right| &\les\dg+\left|\left(\divS\nab^\S_{\nu}f\right)_{\ell=1} \right|+\left|\left(\divS\nab^\S_{\nu}\fb\right)_{\ell=1}\right|+|b_0^{(1)}|\\
&+(\epg+r^{-1})\left(\left\Vert\nab^\S_{\nu}(F)\right\Vert_{\hk_{K-1}(\S)}+r^{-1}\big\Vert\nu(\ovb)\big\Vert_{\hk_{K-1}(\S)}\right).
\end{split}
\eea
To estimate the $\ell=1$ modes of $(\nab^\S_{\nu})(f,\fb)$, we make use of the equations \eqref{nuffb} to obtain
\beaa 
\int_\S\divS(\nabS_\nu)(f)\JpS&=&2\int_\S(\divS\etaS-\divS\eta)\JpS+\lot\\
&=& 2\BS-\int_\S(\div\eta)\JpS+\int_\S (\sdiv\eta-\sdiv^\S\eta)\JpS+\lot\\
&=&2\BS+O(\dg),
\eeaa
where we used \eqref{assuml1}, \eqref{BBbDfirst}, Lemmas \ref{lemma:comparison-gaS-ga} and \ref{Lemma:coparison-forintegrals} at the last step. Similarly, we have
\beaa 
\int_\S\divS(\nabS_\nu)(\fb)\JpS&=&2\BbS+O(\dg),
\eeaa
Thus, we deduce from \eqref{ffbovlac} that
\begin{align}\label{nu1ffb}
\begin{split}
&\left\Vert\nabS_{\nu}(f,\fb),\nu\ovla-\ov{\nu\ovla}^\S\right\Vert_{\hk_{K-1}(\S)}\\
\les&\;\dg+|\BS|+|\BbS|+(\epg+r^{-1})\left(\left\Vert\nab^\S_{\nu}(F)\right\Vert_{\hk_{K-1}(\S)}+r^{-1}\big\Vert\nu(\ovb)\big\Vert_{\hk_{K-1}(\S)}\right),
\end{split}
\end{align}
and
\begin{align}
\begin{split}\label{nu1ovla}
 r\left|(\ov{\nu\ovla}^\S)\right|&\les \dg+|\BS|+|\BbS|+|b_0^{(1)}| \\
&+(\epg+r^{-1})\left(\left\Vert\nab^\S_{\nu}(F)\right\Vert_{\hk_{K-1}(\S)}+r^{-1}\big\Vert\nu(\ovb)\big\Vert_{\hk_{K-1}(\S)}\right).
\end{split}
\end{align}
In the sequel, we denote for $1\leq l\leq s_{max}+1$.
\beq\label{defEE}
\EE_l:=\dg+|\BS|+|\BbS|+|\DS|+(\epg+r^{-1})\left(\left\Vert(\nab^\S_{\nu})^l(F)\right\Vert_{\hk_{K-l}(\S)}+r^{-1}\big\Vert\nu^l(\ovb)\big\Vert_{\hk_{K-l}(\S)}\right),
\eeq
Next, we consider $b_0^{(1)}=\ov{\nu(r-\rS)}^\S$. Applying Lemmas \ref{nonpo}, \ref{dint}, \eqref{eq:Generalframetransf}, \eqref{eq:GCM-improved estimate2-again}, \eqref{eq:ThmGCMS1} and \eqref{zbO}, we infer
\beaa 
\nu(r-\rS)&=&e_3(r)+\bS e_4(r)-\nu(\rS)+O(\dg)\\
&=&\frac{r}{2}\left(\Omb\ov{\ka}+z\ov{z^{-1}\kab}-z\ov{z^{-1}\Omb\ka}\right)+\bS\frac{r}{2}\ov{\ka}-\frac{\rS}{2}\zS\ov{(\zS)^{-1}(\kabS+\bS\kaS)}^\S+O(\dg)\\
&=&\Omb+\frac{rz}{2}\ov{z^{-1}\kab}-z\ov{z^{-1}\Omb}+\bS-\frac{\rS}{2}\zS\ov{(\zS)^{-1}\kabS}^\S+\zS\ov{(\zS)^{-1}(\Omb^\S+\zS)}^\S+O(\dg)\\
&=&\Omb+\frac{rz}{2}\ov{z^{-1}\kab}-z\ov{z^{-1}\Omb}-\OmbS-\frac{\rS}{2}\zS\ov{(\zS)^{-1}\kabS}^\S+\zS\ov{(\zS)^{-1}\Omb^\S}^\S+O(\dg).
\eeaa
Recall from Lemma \ref{lemma:comparison-gaS-ga} and Proposition \ref{Prop:transformation-formulas-generalcasewithoutassumptions} that 
\bea\label{kabSkabest}
r-\rS=O(\dg),\qquad \kabS-\kab=O(\dg).
\eea
Notice that \eqref{bSb} and \eqref{nu1ffb} imply
\bea\label{bbSEE1}
|b-\bS|\les \dg+|\DS|+\Vert\nabS_\nu(f,\fb)\Vert_{\hk_{K-1}(\S)}=O(\EE_1).
\eea
Recall that Lemma \ref{Lemma:coparison-forintegrals} allows us to compare the average on S and $\S$. Thus, we may assume from now on that all averages are on $\S$. We obtain
\beaa 
\nu(r-\rS)&=&\Omb+\frac{rz}{2}\ov{z^{-1}\kab}-z\ov{z^{-1}\Omb}-\OmbS-\frac{r}{2}\zS\ov{(\zS)^{-1}\kab}+\zS\ov{(\zS)^{-1}\Omb^\S}+O(\dg)\\
&=&-b+\frac{rz}{2}\ov{z^{-1}\kab}+z\ov{z^{-1}b}+\bS-\frac{r\zS}{2}\ov{(\zS)^{-1}\kab}-\zS\ov{(\zS)^{-1}\bS}+O(\EE_1)\\
&=&z\ov{z^{-1}\left(\frac{r}{2}\kab+b\right)}-\zS\ov{(\zS)^{-1}\left(\frac{r}{2}\kab+b\right)}+O(\EE_1).
\eeaa
We denote
\beaa 
\frac{r}{2}\kab+b&:=&-2+h_0,\qquad h_0= r\Ga_b=O(\epg).
\eeaa 
Thus, we deduce
\bea\label{nbsl}
\nu(r-\rS)&=&-2\left(z\ov{z^{-1}}-\zS\ov{(\zS)^{-1}}\right)+O(\epg)\ov{z-\zS}+O(\EE_1).
\eea 
Taking the average of \eqref{nbsl} on $\S$, recalling that $\ov{ab}=\ov{a}\ov{b}+\ov{(a-\ov{a})(b-\ov{b})}$ and applying \eqref{zwidecheckEE}, we obtain
\beaa
\ov{\nu(r-\rS)}&=&-2\left(\ov{z}\ov{z^{-1}}-\ov{\zS}\,\ov{(\zS)^{-1}}\right)+O(\epg)|z-\zS|+O(\EE_1)\\
&=&2\left(\ov{(z-\ov{z})(z^{-1}-\ov{z^{-1}})}-\ov{(\zS-\ov{\zS})\left((\zS)^{-1}-\ov{(\zS^{-1}})\right)}\right)+O(\epg)|z-\zS|+O(\EE_1)\\
&=&O(\epg)\ov{z-\zS}+O(\EE_1).
\eeaa
On the other hand, we apply \eqref{eq:GCM-improved estimate2-again}, \eqref{e3re3s}, \eqref{zbO}, \eqref{bbSEE1} and Lemma \ref{nonpo} to obtain
\beaa 
\nu(r-\rS)&=&e_3(r)+\bS e_4(r)-\bS-\OmbS+O(\dg)=e_3(r)+b\frac{r}{2}\ov{\ka}+\zS+O(\EE_1)\\
&=&e_3(r)-e_3(s)+\Omb+b+\zS+O(\dg)=e_3(r)-e_3(s)+\zS-z+O(\EE_1)\\
&=&\zS-z+O(\EE_1).
\eeaa 
Thus, we deduce
\bea\label{zzSrrS}
\ov{\zS-z}=\ov{\nu(r-\rS)}+O(\EE_1)=O(\epg)(\ov{\zS-z})+O(\EE_1),
\eea
which implies
\beaa 
\ov{z-\zS}=O(\EE_1),\qquad \ov{\nu(r-\rS)}=O(\EE_1),
\eeaa 
and hence, recalling \eqref{defEE} and Lemma \ref{Lemma:coparison-forintegrals}, we obtain
\begin{align}
\begin{split}\label{nu1b0}
    b_0^{(1)}&=\ov{\nu(r-\rS)}+O(\dg)\\
    &\les\dg+|\BS|+|\BbS|+|\DS|\\
    &+(\epg+r^{-1})\left(\left\Vert\nab^\S_{\nu}(F)\right\Vert_{\hk_{K-1}(\S)}+r^{-1}\big\Vert\nu(\ovb)\big\Vert_{\hk_{K-1}(\S)}\right).
\end{split}
\end{align}
Finally, we recall the equation of $\ovb$ in \eqref{nuGCM}. By \eqref{hhhhh} and Lemma \ref{elliptic1}, we have
\beq\label{ovbc}
\left\|\nu(\ovb)-\ov{\nu(\ovb)}^\S\right\|_{\hk_{K-1}(\S)}\les r\dg+r\left\Vert\nab^\S_{\nu}(F)\right\Vert_{\hk_{K-1}(\S)}+(\epg +r^{-1})\big\Vert\nu(\ovb)\big\Vert_{\hk_{K-1}(\S)}.
\eeq
Thus, from \eqref{nu1ffb}, \eqref{nu1ovla} \eqref{nu1b0} and \eqref{ovbc}, we conclude that
\beaa 
&&\left\Vert\nabS_{\nu}(f,\fb,\ovla)\right\Vert_{\hk_{K-1}(\S)}+r^{-1}\left\|\nu(\ovb)\right\|_{\hk_{K-1}(\S)}\\
&\les&\dg+|\BS|+|\BbS|+|\DS|+(\epg+r^{-1})\left(\left\Vert\nab^\S_{\nu}(F)\right\Vert_{\hk_{K-1}(\S)}+r^{-1}\big\Vert\nu(\ovb)\big\Vert_{\hk_{K-1}(\S)}\right).
\eeaa
For $\epg$ and small enough and $\ovr$ large enough, we deduce \eqref{nuFdg}. This concludes the proof of Proposition \ref{th24}.
\end{proof}
\subsection{Extrinsic properties of \texorpdfstring{$\Sigma$}{}}\lab{sec4.2}
The goal of this section is to prove the following proposition.
\begin{proposition}\label{bS}
We assume that
\bea\label{BBbDass}
|\BS|+|\BbS|+|\DS| \leq \epg^{\frac{2}{3}}.
\eea
Then, the following estimates hold true for all $k\leq s_{max}$:
\begin{enumerate}
\item The Ricci coefficients $\etaS,\xibS,\ombS$ verify
\beaa
\Vert\etaS\Vert_{\hk_{k}(\S)} &\les& \epg+|\BS|,\\
\Vert\xibS\Vert_{\hk_{k}(\S)} &\les& \epg+|\BbS|,\\
\Vert\ombcheck^\S\Vert_{\hk_{k}(\S)} &\les& \epg+|\BS|+|\BbS|.
\eeaa
\item The scalar $\bS$ verifies
\beaa
(\rS)^{-1} \Vert \bcS \Vert_{\hk_k(\S)} &\les& \epg+|\BS|+|\DS|.
\eeaa
\item We also have
\beaa
(\rS)^{-1} \Vert \Ombc^\S \Vert_{\hk_{k}(\S)}&\les& \epg +|\BS|+|\BbS|+|\DS|,\\
(\rS)^{-1} \Vert \zcS \Vert_{\hk_{k}(\S)} &\les& \epg +|\BS|+|\BbS|+|\DS|.
\eeaa
\end{enumerate}
\end{proposition}
\begin{remark}\label{Gammagb}
We split the Ricci and curvature coefficients as follows.
\beaa
\Ga_g^\S &=& \left\{\kac^\S,\, \chihS,\, \zeS,\, \kabc^\S,\,r \rhocheck^\S,\,r\rhod^\S,\,r\b^\S,\,r\a^\S,\,r\mucheck^\S,\,r\widecheck{K}^\S \right\},\\
\Ga_{b,w}^\S &=& \left\{\chibhS,\, r\bb^\S,\,\aa^\S\right\},\\
\Ga_{b,i}^\S &=& \left\{\etaS,\,\xibS,\, \ombc^\S\right\}.
\eeaa
Since all the terms in $\GagS\cup \Ga_{b,w}^\S$ are well-defined on a sphere, as a direct consequence of Corollary \ref{generalGCMspheres}, see also \eqref{eq:ThmGCMS6}, we have the following estimates for $k\leq s_{max}$.
\bea
\begin{split}\label{Gaggj}
\Vert \GagS \Vert_{\hk_k(\S)}\les\epg r^{-1},\qquad\qquad\Vert \Ga_{b,w}^\S \Vert_{\hk_k(\S)}\les\epg.
\end{split}
\eea
By the null structure equations and Bianchi equations introduced in Proposition \ref{nullstructure}, we obtain immediately\footnote{As explained in Remark \ref{rem4.5}, the transversality conditions are consistent with extension by local outgoing geodesic foliation initialized on $\Si$. Hence, the equations in Proposition \ref{nullstructure} are valid on $\Si$.}
\bea\label{e4Gaggj}
\Vert\nabS_4\GagS\Vert_{\hk_{k-1}(\S)}\les \epg r^{-2},\qquad \Vert \nabS_4\Ga_{b,w}^\S \Vert_{\hk_{k-1}(\S)}\les \epg r^{-1}.
\eea
\end{remark}
\begin{remark}
In the sequel, unless stated otherwise, we admit the following conventions:
\begin{itemize}
    \item For a scalar quantity $h^\S$ defined on $\S$, we denote $\ov{h^\S}$ for its average over the sphere $\S$;
    \item  For a background scalar quantity $h$, we denote $\ov{h}$ for its average over the background sphere $S(u,s)$;
    \item For any scalar quantity $h$, we denote $\ov{h}^\S$ for its average over the sphere $\S$.
\end{itemize}
\end{remark}
Recalling that $\dg\leq\epg$, applying Proposition \ref{th24}, \eqref{BBbDass} and Lemma \ref{nunununu}, we deduce the following estimates for $k\leq s_{max}$
\bea\label{bootGab}
\Vert \Ga_{b,i}^\S\Vert_{\hk_k(\S)} &\les& \epg^\frac{2}{3}.
\eea
Notice that we have
\bea\label{e3Gab}
\Vert \nab_3^\S\GagS\Vert_{\hk_{k-1}(\S)}\les \epg^\frac{2}{3} r^{-1},\qquad \Vert \nab_3^\S\GabS\Vert_{\hk_{k-1}(\S)}\les \epg^\frac{2}{3},
\eea
as a direct consequence of \eqref{bootGab} and Proposition \ref{nullstructure}.\\ \\
To prove Proposition \ref{bS}, we need the following lemmas.
\begin{lemma}\label{GCMfz}
Under the GCM conditions \eqref{GCMtj} on $\Si$, we have
\beaa
\nab^\S_{\nu^\S} (\ddsStwo\ddsSone \muS)&=&r^{-5}O(\epg),\\
\nab^\S_{\nu^\S}(\nab^\S\kaS)&=&0,\\
\nab^\S_{\nu^\S}\left((\dddS_2\ddsStwo+2K^\S)\ddsStwo\ddsSone\kabS\right)&=&r^{-6}O(\epg).
\eeaa
\end{lemma}
\begin{lemma}\label{etaxibomb}
The following estimates hold true for $k\leq s_{max}-5$
\begin{align}
\begin{split}
\Vert\ddsStwo \etaS\Vert_{\hk_{k+4}(\S)}&\les \epg r^{-1},\\
\Vert\ddsStwo \xibS\Vert_{\hk_{k+4}(\S)}&\les \epg r^{-1},\\
\Vert\ddsSone \ombS\Vert_{\hk_{k+4}(\S)}&\les \epg r^{-1}+ r^{-1}\Vert \etaS \Vert_{\hk_{k+4}(\S)}+r^{-1} \Vert \xibS \Vert_{\hk_{k+4}(\S)}. \lab{etaxibombestimates}
\end{split}
\end{align}
\end{lemma}
\begin{lemma}\label{ombpjz}
We have the following identity for $\ov{\ombcheck^\S}$.
\beaa
\ov{\ombcheck^\S}=\ov{\ombS}-\frac{\mS}{(\rS)^2}&=&-\frac{1}{2\rS}\ov{\Ombc^\S}+\frac{\rS}{2} \left(\ov{\mucheck^\S}-\ov{\etaS\cdot\etaS}+\ov{\frac{\bS}{2}\hch^\S\cdot\hch^\S}\right).
\eeaa
\end{lemma}
\begin{lemma}\label{weakestimate}
    We have the following estimate for $k\leq s_{max}$:
    \begin{equation}
        \left\|(\bcS,\zcS,\OmbcS)\right\|_{\hk_k(\S)}\les r\epg^\frac{2}{3}.
    \end{equation}
\end{lemma}
\begin{proof}[Proof of Lemma \ref{GCMfz}]
Recalling that
\bea\label{numu}
\mucheck^\S&=& M_0^\S +\sum_p \MpS\JpS,
\eea
we infer
\beaa
\ddsStwo\ddsSone \muS &=& \sum_p \MpS (\ddsStwo\ddsSone\JpS).
\eeaa
Therefore,
\bea\label{mubiao}
\nab^\S_{\nu^\S}(\ddsStwo\ddsSone \muS)=\nu^\S(\MpS)\ddsStwo\ddsSone\JpS+\MpS\nab^\S_{\nu^\S}(\ddsStwo\ddsSone\JpS).
\eea
Firstly, we estimate $\nab^\S_{\nu^\S}(\ddsStwo\ddsSone\JpS)$. Recall that $\nu^\S(\JpS)=0$, By Lemma \ref{commfor}, we have
\beaa
\nab^\S_{\nu^\S}(\ddsStwo\ddsSone\JpS)&=&[\nab^\S_{\nu^\S},\ddsStwo\ddsSone]\JpS \\
&=&[\nab^\S_{\nu^\S},\ddsStwo]\ddsSone\JpS+\ddsStwo[\nab^\S_{\nu^\S},\ddsSone]\JpS \\
&=&\frac{2}{\rS}\ddsStwo\ddsSone\JpS + \Ga_b^\S\c\dk \ddsSone\JpS \\
&& + \ddsStwo\left(\frac{2}{\rS}\ddsSone \JpS + \Ga_b^\S \c \dk\JpS\right)\\
&=& \frac{4}{\rS}\ddsStwo\ddsSone\JpS+(\rS)^{-2} O(\epg^\f12).
\eeaa
Recall that (see Lemma \ref{jpprop} and also Corollary 5.2.10 in \cite{KS:main}).
\beaa
|\ddsStwo\ddsSone\JpS| &\les& \frac{\epg}{r^2},\qquad M^{(\S,p)}\in r^{-1}\GagS.
\eeaa
According to \eqref{mubiao}, we deduce
\beaa
\nab^\S_{\nu^\S}(\ddsStwo\ddsSone\muS)=(\rS)^{-5}O(\epg),
\eeaa
which implies the first estimate. The third estimate is similar and left to the reader. Finally, the second estimate follows immediately from the GCM condition $\kaS=\frac{2}{\rS}$. This concludes the proof of Lemma \ref{GCMfz}. 
\end{proof}
\begin{proof}[Proof of Lemma \ref{etaxibomb}]
We denote
\bea
\begin{split}
C_1&=\nab_3^\S(\ddsStwo\ddsSone \muS),\\
C_2&=\nab_3^\S(\nab \kaS),\\
C_3&=\nab^\S_3(\ddsStwo\dddS_2+2K^\S)\ddsStwo\ddsSone\kabS.
\end{split}
\eea
Recalling that $\nu^\S=e_3^\S+\bS e_4^\S$, we obtain in view of Lemma \ref{GCMfz}
\beaa
C_1&=&-\bS \nab^\S_4 (\ddsStwo\ddsSone \muS)+r^{-5}O(\epg) ,\\
C_2&=&-\bS \nab^\S_4 (\nab^\S\kaS) ,\\
C_3&=& -\bS \nab^\S_4 \left( \MpS (\dddS_2\ddsStwo+2K^\S)\ddsStwo\ddsSone\kabc^\S \right)+r^{-6}O(\epg).
\eeaa
We claim that
\bea
\begin{split}\label{nab4}
\Vert\nab^\S_4 (\ddsStwo\ddsSone\muS)\Vert_{\hk_{k-3}(\S)}&\les& \epg r^{-5},\\
\Vert\nab^\S_4 (\nab^\S\kaS)\Vert_{\hk_{k-2}(\S)}&\les& \epg r^{-3},\\
\left\Vert\nab^\S_4 \left((\ddsStwo\dddS_2+2K^\S)\ddsStwo\ddsSone\kabS\right)\right\Vert_{\hk_{k-5}(\S)}&\les& \epg r^{-6}.
\end{split}
\eea
Indeed, making use of Lemma \ref{commfor} and \eqref{e4Gaggj}, we have
\beaa
\nabS_4\nabS\kaS&=&[\nabS_4,\nabS]\kac^\S+\nabS\nabS_4\kac^\S\\
&=&-r^{-1}\nabS\kac^\S+r^{-1}\GagS\cdot\dk^{\leq 1}\kac^\S+r^{-1}\dk \nab_4^\S(\GagS)\\
&=&r^{-4}O(\epg),
\eeaa
which implies the second estimate in \eqref{nab4}. The first and third estimates in \eqref{nab4} are similar and left to the reader.\\ \\
Writing $\bS=\bcS-1-\frac{2\mS}{\rS}$, recall the bootstrap assumptions \eqref{bootGab} and making use of product estimates we deduce
\beaa
\Vert C_1\Vert_{\hk_{k-3}(\S)}&\les& \epg r^{-4},\\
\Vert C_2\Vert_{\hk_{k-2}(\S)}&\les& \epg r^{-3},\\
\Vert C_3\Vert_{\hk_{k-5}(\S)}&\les& \epg r^{-5}.
\eeaa
Next, we estimate the error terms
\beaa
E_k&:=& r^{-1}\Vert \GagS \Vert_{\hk_{k}(\S)}+\Vert \GabS\c\GabS\Vert_{\hk_{k}(\S)} \\
&\les & r^{-2}\epg + r^{-1} \Vert \GabS\Vert_{\hk_{k}(\S)}\Vert \GabS\Vert_{\hk_{k}(\S)} \\
&\les & r^{-1}\epg,\qquad \forall k\leq s_{max}.
\eeaa
We cite the following identities, see Propositions 5.1.21 and 5.1.22 in \cite{KS:main}:
\begin{align*}
 2\ddsStwo\ddsSone \dddS_1\dddS_2\ddsStwo\etaS &= -\ddsStwo\ddsSone \dddS_1\nab^\S_3\nab^\S\kaS+\frac{2}{\rS}\nab^\S_3\ddsStwo\ddsSone\muS -\frac{4}{\rS}\ddsStwo\ddsSone\divS\bb^\S  \\
& + (\rS)^{-5} (\dkb^\S)^{\le 4} \Ga_g^\S+ (\rS)^{-4} (\dkb^\S)^{\le 4} (\Ga_b^\S\c \Ga_b^\S),\\
2\ddsStwo\ddsSone \dddS_1\dddS_2\ddsStwo\xibS  &= \nab^\S_3(\ddsStwo\dddS_2+2K^\S)\ddsStwo\ddsSone\kabS +\frac{2}{\rS}\nab_3\ddsStwo\ddsSone\muS \\
&-\frac{4}{\rS}\ddsStwo\ddsSone \divS\bb^\S + (\rS)^{-5} (\dkb^\S)^{\le 4} \Ga_g^\S+ (\rS)^{-4} (\dkb^\S)^{\le 4} (\Ga_b^\S\c \Ga_b^\S),\\
2\nab^\S \ombS &=\frac{1}{\rS}\xibS - \nab^\S_3 \zeS-\bb^\S +\frac{1}{\rS}\etaS +(\rS)^{-1}\Ga^\S_g+\Ga^\S_b\c \Ga^\S_b.    
\end{align*}
Recalling that 
$$\ddsSone\dddS_1=2\dddS_2\ddsStwo+2K,$$
we deduce
\beaa 
\ddsStwo\ddsSone \dddS_1\dddS_2\ddsStwo\etaS&=&2\ddsStwo(\dddS_2\ddsStwo+K)\dddS_2\ddsStwo\etaS\\
&=&2\ddsStwo\dddS_2\ddsStwo\dddS_2\ddsStwo\etaS+\frac{2}{r^2}\ddsStwo\dddS_2\ddsStwo\etaS\\
&&+2\widecheck{K}\ddsStwo\dddS_2\ddsStwo\etaS+2\ddsStwo(\widecheck{K})\dddS_2\ddsStwo\etaS.
\eeaa 
Notice that
\beaa
2\widecheck{K}\ddsStwo\dddS_2\ddsStwo\etaS&=&O(\epg r^{-3})\ddsStwo\dddS_2\ddsStwo\etaS,\\
2\ddsStwo(\widecheck{K})\dddS_2\ddsStwo\etaS&=&O(\epg r^{-4})\dddS_2\ddsStwo\etaS,
\eeaa 
where we used \eqref{Gaggj}. Thus, by Corollary \ref{d2d2}, we obtain
\beaa 
\Vert\ddsStwo\etaS\Vert_{\hk_{k+4}(\S)}&\les& r^4\Vert\ddsStwo\dddS_2\ddsStwo\dddS_2\ddsStwo\etaS\Vert_{\hk_{k}(S)}+r^2\Vert\ddsStwo\dddS_2\ddsStwo\etaS\Vert_{\hk_k(S)}\\
&\les&r^4\Vert\ddsStwo\ddsSone\dddS_1\dddS_2\ddsStwo\etaS\Vert_{\hk_k(S)}+O(\epg)\Vert\ddsStwo\etaS\Vert_{\hk_{k+2}(S)},
\eeaa 
which implies
\beaa 
\Vert\ddsStwo\etaS\Vert_{\hk_{k+4}(\S)}&\les&r^4\Vert\ddsStwo\ddsSone\dddS_1\dddS_2\ddsStwo\etaS\Vert_{\hk_k(S)}.
\eeaa 
Using \eqref{Gaggj}, \eqref{e4Gaggj} and \eqref{bootGab}, we deduce
\beaa
\Vert\ddsStwo\etaS\Vert_{\hk_{k+4}(\S)}&\les& r\Vert C_2\Vert_{\hk_{k+3}(\S)}+r^3\Vert C_1\Vert_{\hk_{k}(\S)}+\epg r^{-1}+E_{4+k}\les\epg r^{-1}.
\eeaa 
Similarly, we have
\beaa
\Vert\ddsStwo\xibS\Vert_{\hk_{k+4}(\S)}&\les& r^4\Vert C_3\Vert_{\hk_{k}(\S)}+r^3\Vert C_1\Vert_{\hk_{k+3}(\S)}+\epg r^{-1}+E_{4+k}\les\epg r^{-1}.
\eeaa 
Finally, recalling that $r\bb^\S\in\GabS$ is well-defined on a sphere, we obtain
\beaa 
\Vert\ddsSone\ombS\Vert_{\hk_{k+4}(\S)}&\les& r^{-1}\Vert \xibS\Vert_{\hk_{k+4}(\S)}+r^{-1} \Vert \etaS\Vert_{\hk_{k+4}(\S)}+r^{-1} \epg+E_{4+k} \\
&\les& r^{-1}\Vert\xibS\Vert_{\hk_{k+4}(\S)}+r^{-1} \Vert \etaS\Vert_{\hk_{k+4}(\S)}+r^{-1}\epg
\eeaa
as stated. This concludes the proof of Lemma \ref{etaxibomb}.
\end{proof}
\begin{proof}[Proof of Lemma \ref{ombpjz}]
We recall the following null structure equation, see Proposition 7.4.1 in \cite{Ch-Kl},
\beaa
\nab_3^\S \kaS + \f12 \kabS\kaS -2\ombS\kaS&=& 2\divS\etaS -\chibhS\cdot\chihS+2(\etaS)^2 +2\rhoS.
\eeaa
Making use of the GCM condition $\kaS=\frac{2}{\rS}$ we deduce
\beaa
\ombS&=& \frac{1}{2\kaS}\left[\nab_3^\S\kaS+\f12\kabS\kaS-2\divS\etaS+\chibhS\cdot\chihS-2(\etaS)^2-2\rhoS \right]\\
&=&\frac{1}{2\kaS}\left[\nab_{\nu^\S}^\S\left(\frac{2}{\rS}\right)-\bS\nab^\S_{4}\kaS+\f12\kabS\kaS-2\divS\etaS+\chibhS\cdot\chihS-2(\etaS)^2-2\rhoS \right]\\
&=&\frac{1}{2\kaS}\left[-2\frac{e_3^\S(\rS)+\bS}{(\rS)^2}+\bS\left(\f12(\kaS)^2+|\hch^\S|^2\right)+\f12\kabS\kaS-2\divS\etaS+\chibhS\cdot\chihS-2(\etaS)^2-2\rhoS\right]\\
&=&\frac{1}{2\kaS}\left[-\frac{2e_3^\S(\rS)}{(\rS)^2}+\bS|\hch^\S|^2+\f12\kabS\kaS-2\divS\etaS+\chibhS\cdot\chihS-2(\etaS)^2-2\rhoS\right]\\
&=& -\frac{e_3^\S(\rS)}{2\rS} +\frac{1}{4}\bS\rS|\hch^\S|^2+\frac{1}{4}\kabS +\frac{\rS}{4}\left[-2\divS\etaS+\chibhS\cdot\chihS-2(\etaS)^2-2\rhoS \right],
\eeaa
where we used
\beaa 
e_4^\S(\rS)=1,\qquad\OmbS=e_3^\S(\rS),\qquad e_4^\S(\kaS)=-\f12(\kaS)^2-|\hch^\S|^2.
\eeaa
Taking the average over $\S$, and recalling that $\muS=-\divS\zeS-\rhoS+\f12\chibhS\cdot\chihS$ and \eqref{ovksb}, we obtain
\beaa
\ov{\ombS}&=& -\frac{1}{2\rS}\ov{\Ombc^\S} +\frac{\rS}{2}\ov{\muS-(\etaS)^2+\frac{\bS}{2}|\hch^\S|^2},
\eeaa
or
\beaa
\ov{\ombcheck^\S}=\ov{\ombS}-\frac{\mS}{(\rS)^2}&=&-\frac{1}{2\rS}\ov{\Ombc^\S}+\frac{\rS}{2} \left(\ov{\mucheck^\S}-\ov{\etaS\cdot\etaS}+\ov{\frac{\bS}{2}\hch^\S\cdot\hch^\S}\right).
\eeaa
This concludes the proof of Lemma \ref{ombpjz}.
\end{proof}
\begin{proof}[Proof of Lemma \ref{weakestimate}]
    We have from Proposition \ref{th24}, \eqref{bSbdifference} and \eqref{assuml1} that
    \begin{equation*}
        \left\|(\bcS-\ov{\bcS})-(\widecheck{b}-\ov{\widecheck{b}})\right\|_{\hk_k(\S)}\les r\dg+r|\DS|\les r\epg^\frac{2}{3},
    \end{equation*}
which implies
\begin{equation}\label{weakb}
    \|\bcS\|_{\hk_k(\S)}\les r\epg^\frac{2}{3}+r|\DS|+\|\widecheck{b}-\ov{\widecheck{b}}\|_{\hk_k(\S)}\les r\epg^\frac{2}{3}.
\end{equation}
We have from Lemma \ref{ombpjz} and \eqref{bootGab}
\begin{align*}
    \left|\ov{\OmbcS}\right|\les r|\ov{\ombS}|+r^2\left|\ov{\widecheck{\mu}^\S}\right|+r^2|\GabS\c\GabS|\les \epg^\frac{2}{3}.
\end{align*}
Next, we have from Proposition \ref{th24}, \eqref{zwidecheckEE} and \eqref{assuml1} that
        \begin{equation*}
        \left\|\left(\OmbcS-\ov{\OmbcS}\right)-(\Ombc-\ov{\Ombc})\right\|_{\hk_k(\S)}\les r\dg+r|\DS|\les r\epg^\frac{2}{3},
    \end{equation*}
which implies
\begin{equation}\label{weakOmb}
    \|\OmbcS\|_{\hk_k(\S)}\les r\epg^\frac{2}{3}+\left\|\widecheck{\Omb}-\ov{\Ombc}\right\|_{\hk_k(\S)}\les r\epg^\frac{2}{3}.
\end{equation}
Finally, we have from \eqref{eqresume} that
\begin{align}
\begin{split}
    \zcS&=\zS-2=-\bS-\OmbS-2\\
    &=-\bS-1-\frac{2\mS}{\rS}-\OmbS-1+\frac{2\mS}{\rS}\\
    &=-\bcS-\OmbcS,\label{zcbcOc}
\end{split}
\end{align}
which implies from \eqref{weakb} and \eqref{weakOmb}
\begin{equation}\label{weakz}
    \|\zcS\|_{\hk_k(\S)}\les \|\bcS\|_{\hk_k(\S)}+\|\OmbcS\|_{\hk_k(\S)}\les r\epg^\frac{2}{3}.
\end{equation}
Combining \eqref{weakb}, \eqref{weakOmb} and \eqref{weakz}, this concludes the proof of Lemma \ref{weakestimate}.
\end{proof}
We are now ready to prove Proposition \ref{bS}.
\begin{proof}[Proof of Proposition \ref{bS}]
We start with the estimate of $\Ombc^\S$. We first note that the GCM condition $\ks^\S=\frac{2}{\rS}$ together with the definition of the Hawking mass
\beaa
\frac{2m^\S}{\rS}&=& 1+ \frac{1}{16\pi}\int_\S \kaS\kabS
\eeaa
implies that,
\bea
\ov{\ksb^\S}=-\frac{2\UpS}{\rS},\label{ovksb}
\eea
where $\UpS=1-\frac{2m^\S}{\rS}$. Thus, in view of Lemma \ref{dint}, we deduce
\beaa
\OmbcS+\bcS-2 &=&\OmbS+\bS =\nu^\S(\rS) = \frac{\rS}{2}z^\S \ov{(z^\S)^{-1}(\kabS+\bS\kaS)}\\
&=& \frac{\rS}{2}z^\S \ov{(z^\S)^{-1}\left(\kabc^\S-\frac{2\UpS}{\rS}+\left(\bcS-1-\frac{2\mS}{\rS}\right) \frac{2}{\rS}\right)} \\
&=& \frac{\rS}{2}z^\S \ov{(z^\S)^{-1}\left(\kabc^\S-\frac{4}{\rS}+\frac{2\bcS}{\rS}\right)}\\
&=& \frac{\rS}{2}z^\S \ov{(z^\S)^{-1}\kabc^\S}-2\zS \ov{(\zS)^{-1}}+\zS\ov{(\zS)^{-1}\bcS}.
\eeaa
Multiplying by $(\zS)^{-1}$ and taking the average over $\S$, we infer
\beaa
\ov{\OmbcS(\zS)^{-1}} &=& \frac{\rS}{2} \ov{(z^\S)^{-1}\kabc^\S}.
\eeaa
We write
\beaa
\frac{1}{\zS}=\frac{1}{2+\zcS}=\frac{1}{2}+O(\zcS),
\eeaa
and hence, recalling that $\ov{\kabc^\S}=0$, we infer
\bea\label{ovOmbcScompute}
\ov{\OmbcS}&=&\ov{O(\zcS)\OmbcS}+\rS\ov{O(\zcS)\kabc^\S}.
\eea
Combining Lemma \ref{resume} and \eqref{ovOmbcScompute}, we obtain
\beaa
e_a^\S(\Ombc^\S)&=&(\zeS_a-\etaS_a)\Omb^\S-\xib_a^\S,\\
\ov{\Ombc^\S}&=&\ov{O(\zcS)\OmbcS}+r\GagS.
\eeaa
Applying Lemma \ref{weakestimate} and \eqref{Gaggj}, we deduce for $k\leq s_{max}-1$,
\bea
\begin{split}\label{Ombcggj}
r^{-1}\Vert \Ombc^\S \Vert_{\hk_{k+1}(\S)}\les& \Vert \etaS \Vert_{\hk_{k}(\S)}+\Vert \xibS \Vert_{\hk_{k}(\S)}+\Vert \GagS \Vert_{\hk_{k}(\S)}+r^2 \Vert \zcS\c\OmbcS\Vert_{L^\infty(\S)}\\
\les&\Vert \etaS \Vert_{\hk_{k}(\S)}+\Vert \xibS \Vert_{\hk_{k}(\S)}+\epg.
\end{split}
\eea
Applying Lemma \ref{elliptic1} to the equation
\beaa
e_a^\S(\zS)&=& (\zeS-\etaS)\zS,
\eeaa
we derive
\bea
\begin{split}\label{zcSbcggj}
r^{-1}\Vert\zS-\ov{\zS}\Vert_{\hk_{k+1}(\S)}&\les \Vert(\zeS-\etaS)\zS\Vert_{\hk_{k}(\S)} \\
&\les \Vert (\zeS-\etaS)\zcS \Vert_{\hk_k(\S)}+\Vert\etaS\Vert_{\hk_k(\S)}+\epg \\
&\les \epg+\Vert\etaS\Vert_{\hk_k(\S)}.
\end{split}
\eea
Next, we estimate $\ov{\zcS}$. For this, we have from \eqref{zcbcOc}
\beaa
\ov{\zcS}&=&-\ov{\bcS}-\ov{\Ombc^\S}=-\DS-\ov{\Ombc^\S}.
\eeaa
Then, applying \eqref{Ombcggj}, we obtain
\beaa
|\ov{\zcS}|\les|\DS|+\Vert\Ombc^\S\Vert_{L^\infty}\les |\DS|+\Vert \etaS \Vert_{\hk_{k}(\S)}+\Vert \xibS \Vert_{\hk_{k}(\S)}+\epg.
\eeaa
By \eqref{zcSbcggj} and the above estimate, we conclude
\bea\label{zcScggj}
r^{-1} \Vert \zcS \Vert_{\hk_{k+1}(\S)}&\les& |\DS|+\Vert \etaS \Vert_{\hk_{k}(\S)}+\Vert \xibS \Vert_{\hk_{k}(\S)}+\epg.
\eea
Finally, according to \eqref{zcbcOc}, we have
\bea
\begin{split}\label{bcScggj}
r^{-1} \Vert \bcS\Vert_{\hk_{k+1}(\S)}&\les r^{-1}\Vert \zcS\Vert_{\hk_{k+1}(\S)}+r^{-1}\Vert \Ombc^\S\Vert_{\hk_{k+1}(\S)} \\
&\les |\DS|+\Vert \etaS \Vert_{\hk_{k}(\S)}+\Vert \xibS \Vert_{\hk_{k}(\S)}+\epg.
\end{split}
\eea
It remains to estimate $(\etaS,\xibS,\ombc^\S)$. Recall from Lemma \ref{etaxibomb} that, $\forall k \le s_{max}-5$, we have
\begin{align*}
\Vert \ddsStwo \etaS \Vert_{\hk_{k+4}(\S)}&\les \epg r^{-1},\\
\Vert \ddsStwo \xibS \Vert_{\hk_{k+4}(\S)}&\les \epg r^{-1},\\
\Vert \ddsSone \ombS \Vert_{\hk_{k+4}(\S)}&\les \epg r^{-1}+r^{-1}\Vert \etaS \Vert_{\hk_{k+4}(\S)}+r^{-1}\Vert \xibS \Vert_{\hk_{k+4}(\S)}.
\end{align*}
In view of the definition of $\BS$ and $\BbS$ and Lemma \ref{elliptic2}, we deduce
\bea\label{etasxibsest}
\begin{split}
\Vert\etaS \Vert_{\hk_{k+5}(\S)}&\les& \epg+|\BS|+|(\curl^\S\etaS)_{\ell=1}|,\\
\Vert\xibS \Vert_{\hk_{k+5}(\S)}&\les&\epg +|\BbS|+|(\curl^\S\xibS)_{\ell=1}|.
\end{split}
\eea
Recall from Proposition \ref{nullstructure} that
\beaa 
\curl^\S\etaS&=&{^*\rho^\S}+\GabS\cdot\GagS,\\
\curl^\S\xibS&=&\GabS\cdot\GabS.
\eeaa
Applying \eqref{Gaggj} and \eqref{bootGab}, we obtain
\beaa 
\curl^\S\etaS&=&r^{-3}O(\epg),\\
\curl^\S\xibS&=&r^{-2}O(\epg).
\eeaa
Thus, \eqref{etasxibsest} implies
\beaa 
\Vert\etaS \Vert_{\hk_{k+5}(\S)}&\les&\epg+|\BS|,\\
\Vert\xibS \Vert_{\hk_{k+5}(\S)}&\les&\epg+|\BbS|.
\eeaa
Next, we have from Lemma \ref{ombpjz} that
\beaa
\left|\ov{\ombcheck^\S}\right|&\les& (\rS)^{-1} \Vert\ov{\Ombc^\S}\Vert_{L^\infty(\S)}+\GagS+ \rS \Vert \GabS\Vert_{L^\infty(\S)}^2 \\
&\les& \epg r^{-1} + r^{-1} \Vert \etaS\Vert_{\hk_k(\S)}+ r^{-1} \Vert \xibS\Vert_{\hk_k(\S)}.
\eeaa
Together with the above estimate for $\ddsSone \omb^\S$ and Lemma \ref{elliptic1}, we obtain
\beaa
\Vert\ombc^\S \Vert_{\hk_{k+1}(\S)}&\les&\epg +\Vert\etaS\Vert_{\hk_{k}(\S)}+\Vert\xibS\Vert_{\hk_{k}(\S)}+r|\ov{\ombc^\S}| \\
&\les& \epg+|\BS|+|\BbS|,\qquad k\leq s_{max}-1,
\eeaa
which concludes the statement 1 of Proposition \ref{bS}. 

We can then go back to the preliminary estimates \eqref{Ombcggj}, \eqref{zcScggj} and \eqref{bcScggj} obtained above to derive 
\begin{align*}
r^{-1}\Vert \Ombc^\S \Vert_{\hk_{k+1}(\S)}
&\les\epg+\Vert \etaS \Vert_{\hk_{k}(\S)}+\Vert \xibS \Vert_{\hk_{k} (\S)}\les \epg+|\BS|+|\BbS|,\\
r^{-1}\Vert \zcS \Vert_{\hk_{k+1}(\S)}
&\les \epg+|\DS|+\Vert \etaS \Vert_{\hk_{k}(\S)}+\Vert \xibS \Vert_{\hk_{k} (\S)}\les \epg+|\BS|+|\BbS|+|\DS|,\\
r^{-1}\Vert \bcS \Vert_{\hk_{k+1}(\S)}
&\les \epg+|\DS|+\Vert\etaS\Vert_{\hk_{k}(\S)}+\Vert\xibS\Vert_{\hk_{k} (\S)}\les \epg+|\BS|+|\BbS|+|\DS|.
\end{align*}
This concludes the remaining statements 2 and 3 of Proposition \ref{bS}.
\end{proof}
\subsection{ODE system for \texorpdfstring{$(\Psi,\La,\Lab)$}{}}\label{sec4.3}
The goal of this section is to prove the following proposition which provides an ODE system for $(\Psi,\La,\Lab)$.
\begin{proposition}\label{LaLabprop}
Let $\Si$ a smooth spacelike hypersurface as in Definition \ref{hypersurfaceSigma}. We define $r(s),B(s),\Bb(s)$ and $D(s)$ on $\S:=\S(s)=\S[P(s)]$ as
\begin{equation}\label{D(s)}
r(s):=\rS,\qquad B(s):=\BS,\qquad \Bb(s):=\BbS,\qquad D(s):=\ov{\bS}+1+\frac{2m_{(0)}}{\rS},
\end{equation}
where $\BS,\BbS$ are given by \eqref{BBbDfirst}.\\ \\
Then we have the following equations for the functions $D(s),r(s)$ and the triplets $\La(s)$, $\Lab(s)$, $B(s)$, $\Bb(s)$:
\begin{align*}
\frac{1}{-1+\psi'(s)}\La'(s)&=B(s)+G(\La,\Lab,\psi)(s)+N(B,\Bb,D,\La,\Lab,\psi)(s),\\
\frac{1}{-1+\psi'(s)}\Lab'(s)&=\Bb(s)+\Gb(\La,\Lab,\psi)(s)+\Nb(B,\Bb,D,\La,\Lab,\psi)(s),\\
\psi'(s)&=-\f12 D(s)+H(B,\Bb,\La,\Lab,\psi)(s)+M(B,\Bb,D,\La,\Lab,\psi)(s),
\end{align*}
where 
\begin{equation}
    \psi(s):=\Psi(s)+s-c_0,
\end{equation}
and $G,\Gb,H,N,\Nb,M$ satisfy the following properties:
\begin{itemize}
\item $G,\Gb$ are $O(1)$--Lipschitz functions of $(\La,\Lab,\psi)$;
\item $H$ is a $O(1)$--Lipschitz function of $(B,\Bb,\La,\Lab,\psi)$;
\item $M,N,\Nb$ are $O(\epg^\f12)$--Lipschitz functions of $(B,\Bb,D,\La,\Lab,\psi)$.
\end{itemize}
\end{proposition}
Before proving this proposition, we first study some properties of the derivation $\frac{d}{ds}$. Recall that, according to Section \ref{sec4.1}, we have
\bea\label{defSi0}
\Si=\bigcup_{p=\{S,N\}} \left\{ \Xi_p(s,y^1_p,y^2_p) \Big/ \, s\geq \ovs,\, (y^1_p)^2+(y^2_p)^2<2 \right\},
\eea
where the maps $\Xi_p(s,y^1_p,y^2_p)$, $p=S,N$ are defined in \eqref{defXi}.
\begin{definition}
We define the vectorfield $X_p$ along $\Si$ as the push forward of $\frac{d}{ds}$ by $\Xi_p$, i.e, for every scalar function $h$, we define
\bea\lab{defXSP}
X_p (h) &:=& \frac{d}{ds} h\left(\Xi_p(s,y_p^1,y_p^2)\right).
\eea
\end{definition}
\begin{proposition}\label{st14}
We have the following expression for $X_p$
\beaa
X_p|_{(s,y^1_p,y^2_p)}&=&\left(\Psi'(s)+\Abr_p(s,y^1_p,y^2_p)\right)\pr_u+\left(1+\Bbr_p(s,y^1_p,y^2_p)\right)\pr_s,
\eeaa
where
\beaa
\Abr_p(s,y^1_p,y^2_p)&:=& \pr_P U_p(y^1_p,y^2_p,P(s))P'(s), \\
\Bbr_p(s,y^1_p,y^2_p)&:=& \pr_P S_p(y^1_p,y^2_p,P(s))P'(s),
\eeaa
with
\beaa
\pr_P U_p(\cdot)P'(s)&=& \Psi'(s)\pr_u U_p(\cdot)+\pr_s U_p(\cdot)+\La'(s)\cdot\pr_\La U_p(\cdot)+\Lab'(s)\cdot\pr_{\Lab}U_p(\cdot)+\pr_J U_p\c J'(s),\\
\pr_P S_p(\cdot)P'(s)&=& \Psi'(s)\pr_u S_p(\cdot)+\pr_s S_p(\cdot)+\La'(s)\cdot\pr_\La S_p(\cdot)+\Lab'(s)\cdot\pr_{\Lab}S_p(\cdot)+\pr_J S_p\c J'(s).
\eeaa
\end{proposition}
\begin{proof}
Clearly,
\beaa
\pr_s\Xi_p(s,y^1_p,y^2_p) &=& \left(\Psi'(s)+\pr_P U_p(y^1_p,y^2_p,P(s))P'(s),1+\pr_P S_p(y^1_p,y^2_p,P(s))P'(s),0,0\right).
\eeaa
Given a function on $\Si$, we infer
\beaa
\frac{d}{ds}h\left(\Xi_p(s,y^1_p,y^2_p)\right)&=&\left(\Psi'(s)+\pr_P U_p(y^1_p,y^2_p,P(s))P'(s)\right)\pr_u h\\
&&+\left(1+\pr_P S_p(y^1_p,y^2_p,P(s))P'(s) \right)\pr_s h\\
&=& X_p(h),
\eeaa
as stated.
\end{proof}
\begin{lemma}\label{st15}
Let $\ga(s)$ be the curve of the South Poles of the background\footnote{By the construction of GCM spheres in Theorem \ref{Theorem:ExistenceGCMS1}, the South poles of GCM spheres coincide with that of background spheres. Hence, we have $\ga(s)\subset \Si$.} that intersect the spheres foliating $\Si$, i.e.
\bea\label{definga}
\ga(s)=(\Psi(s),s,0,0)
\eea 
in the south coordinates chart.\\ \\
Then, for a function $h$ defined on $\ga(s)$:
$$
h(s):=h(\Psi(s),s,0,0),
$$
we have (recall $\nu^\S=e_3^\S+\bS e_4^\S$)
\bea \label{Xgas}
X\big|_{\ga(s)} (h)&=&C(s) \nu^\S\big|_{\ga(s)}(h),
\eea 
where
\bea
C(s)&=&\frac{\la}{z}\Psi'(s)\Big|_{\ga(s)}+F\c\Ga_b+O(F^2),\label{C(s)}
\eea
and
\beaa
F&:=&(f,\fb,\ovla),\qquad \ovla=\la-1.
\eeaa
Moreover, we have the following expression for $\bS$:
\bea
\bS\big|_{\ga(s)}&= &\frac{z-\Omb\Psi'(s)}{\la^2\Psi'(s)}\bigg|_{\ga(s)}+F\c\Ga_b+O(F^2).\label{Eq:bSsp}
\eea
Here $f,\fb,\la$ are the transition coefficients and $z,\Omb$ correspond to the background foliation.
\end{lemma}
\begin{proof}
Note that $U(0,0,P(s))=S(0,0,P(s))=0$, so we have
\beaa 
\Abr(s,0,0)=\Bbr(s,0,0)=0
\eeaa 
in Proposition \ref{st14}. Thus, we have
\beaa
X\big|_{\ga(s)}=\Psi'(s) \pr_u + \pr_s.
\eeaa
Note that $X$ is tangent to $\Si$. Thus, we have
\beaa 
\g(X,N^\S)\big|_{\ga(s)}=0,\qquad N^\S=e_3^\S-\bS e_4^\S,
\eeaa 
since the vectorfield $N^\S$ is normal to $\Si$. Now, applying \eqref{eq:decompositionofnullframeoncoordinatesframeforbackgroundfoliation}, \eqref{eq:Generalframetransf} and $Z^c\in\Ga_b$\footnote{See mentioned in Assumption {\bf{A3}}.}, and recalling that $z=\frac{2}{\vsi}$, we have
\beaa 
\g(X,N^\S)&=&\g\left(\Psi'(s)\pr_u+\pr_s,e_3^\S-\bS e_4^\S\right)\\
&=&\g\left(\Psi'(s)\left(\frac{1}{z}e_3-\frac{\Omb}{z}e_4-\frac{2}{z}Z^ce_c\right)+e_4,\la^{-1}e_3+\la^{-1}\fb^c e_c-\bS\left(\la e_4+\la f^c e_c\right)\right)+O(F^2)\\
&=&\g\left(\frac{\Psi'}{z}e_3+\left(1-\frac{\Omb\Psi'}{z}\right)e_4-\frac{2\Psi'Z^c}{z}e_c,\la^{-1}e_3-\la\bS e_4+\left(\la^{-1}\fb^c-\la\bS f^c\right)e_c\right)+O(F^2)\\
&=&-\frac{2}{\la}\left(1-\frac{\Omb\Psi'}{z}\right)+2\la\bS\frac{\Psi'}{z}+F\c \Ga_b+O(F^2).
\eeaa 
Since $\g(X,N^\S)\big|_{\ga(s)}=0 $, we deduce
\beaa 
\left[-\frac{2}{\la}\left(1-\frac{\Omb\Psi'}{z}\right)+2\la\bS\frac{\Psi'}{z}\right]\Bigg|_{\ga(s)}=F\c \Ga_b+O(F^2),
\eeaa 
which implies
\beaa 
\bS\big|_{\ga(s)}=\frac{z-\Omb\Psi'}{\la^2\Psi'}\bigg|_{\ga(s)}+F\c \Ga_b+O(F^2),
\eeaa 
as stated in \eqref{Eq:bSsp}.\\ \\
In view of $\nu^\S=e_3^\S+\bS e_4^\S$, \eqref{eq:Generalframetransf} and \eqref{suc}, we have
\bea
\begin{split}\label{nue4e3}
\nu^\S=&e_3^\S+\bS e_4^\S\\
=&\la^{-1}\left(e_3+\fb^b e_b\right)+\bS\la \left(e_4+f^be_b\right)+O(F^2)\\
=&\la^{-1}e_3+\la\bS e_4+(\la^{-1}\fb^c+\la\bS f^c) e_c+O(F^2)\\
=&\la^{-1}(z\pr_u+\Omb\pr_s+2\Bb^a\pr_{y^a})+\la\bS\pr_s+(\la^{-1}\fb^c+\la\bS f^c)X^a_{(c)}\pr_{y^a}+O(F^2)\\
=&\frac{z}{\la}\pr_u+\left(\frac{\Omb}{\la}+\la\bS\right)\pr_s+\left(2\la^{-1}\Bb^a+(\la^{-1}\fb^c+\la\bS f^c)X^a_{(c)}\right)\pr_{y^a}+O(F^2).
\end{split}
\eea
Applying \eqref{nue4e3} to a function $h(s)=h(\Psi(s),s,0,0)$, we obtain
\bea\label{nuh}
\nu^\S\big|_{\ga(s)}(h)&=&\frac{z}{\la}\pr_u(h)+\left(\frac{\Omb}{\la}+\la\bS\right)\pr_s(h).
\eea
Recalling \eqref{Eq:bSsp} and comparing \eqref{nuh} to $X\big|_{\ga(s)}=\Psi'(s)\pr_u+\pr_s$, we obtain
\beaa 
X\big|_{\ga(s)}(h)=C(s)\nu^\S\big|_{\ga(s)}(h),
\eeaa 
where
\beaa
C(s)&=&\frac{\la}{z}\Psi'(s)\Big|_{\ga(s)}+F\c\Ga_b+O(F^2),
\eeaa
as stated in \eqref{Xgas}. This concludes the proof of Lemma \ref{st15}.
\end{proof}
Now, we derive the following equations for $\La(s)$ and $\Lab(s)$.
\begin{lemma}\label{ODElalab}
We have the following identities
\bea
\begin{split}\lab{Eq:ODElalab}
\nu^\S\big|_{\ga(s)}\La(s)&=&\int_\S\nu^\S(\divS f) J^{(\S,p)}-\frac{4}{\rS}\La(s)+E(s),\\
\nu^\S\big|_{\ga(s)}\Lab(s)&=&\int_\S\nu^\S(\divS \fb) J^{(\S,p)}-\frac{4}{\rS}\Lab(s)+\Eb(s),
\end{split}
\eea
with error terms\footnote{For a scalar function $h$ defined on $\S$, $h-h\big|_{\ga(s)}$ is the scalar function define on $\S$ by subtracting its value at South pole $\ga(s)$.}
\beaa
E(s)&=&\int_\S \left(\frac{2}{\rS}\bcS+\kabc^\S\right)(\divS f) J^{(\S,p)}  \\
&&+ \zS\big|_{\ga(s)}\int_\S \left((\zS)^{-1}-(\zS)^{-1}\big|_{\ga(s)}\right)\left(\nu^\S(\divS f)-\frac{4}{\rS}\divS f +\left(\frac{2}{\rS}\bcS+\kabc^\S\right)\divS f\right)J^{(\S,p)},\\
\Eb(s)&=&\int_\S\left(\frac{2}{\rS}\bcS+\kabc^\S\right)(\divS \fb) J^{(\S,p)}  \\
&&+ \zS\big|_{\ga(s)}\int_\S \left((\zS)^{-1}-(\zS)^{-1}\big|_{\ga(s)}\right)\left(\nu^\S(\divS\fb)-\frac{4}{\rS}\divS \fb+\left(\frac{2}{\rS}\bcS+\kabc^\S\right)\divS\fb\right)J^{(\S,p)}. \\
\eeaa
\end{lemma}
\begin{proof}
According to Lemma \ref{dint} we have
\beaa
\nu^\S\left(\int_\S h \right)&=&\zS \int_\S (\zS)^{-1}\left(\nu^\S(h) + (\kabS +\bS \kaS) h \right).
\eeaa
Thus, choosing $h=(\divS f) J^{(\S,p)}$, we infer
\beaa
\nu^\S\big|_{\ga(s)}(\La)&=&\nu^\S\left(\int_\S (\divS f)\JpS\right)\bigg|_{\ga(s)}\\
&=&\zS\big|_{\ga(s)}\int_\S (\zS)^{-1}\left(\nu^\S((\divS f) J^{(\S,p)})+(\kabS+\bS\kaS)(\divS f) J^{(\S,p)}\right).
\eeaa
Recalling that $\kaS=\frac{2}{\rS}$, we have
\beaa
\kabS+\bS\kaS&=&\frac{2}{\rS}\bS-\frac{2\UpS}{\rS}+\kabc^\S \\
&=& \frac{2}{\rS}\left(\bS+1+\frac{2\mS}{\rS}\right)-\frac{2}{\rS}\left(-\UpS+2\right)-\frac{2\UpS}{\rS}+\kabc^\S\\
&=& \frac{2}{\rS}\bcS-\frac{4}{\rS}+\kabc^\S.
\eeaa
Since $\nu^\S(J^{(\S,p)})=0$, we infer
\begin{align*}
\nu^\S\big|_{\ga(s)}(\La)=&\int_\S \left(\nu^\S(\divS f) J^{(\S,p)}+(\kabS+\bS\kaS)(\divS f) J^{(\S,p)} \right) \\
&+\zS\big|_{\ga(s)} \int_\S \left((\zS)^{-1}-(\zS)^{-1}\big|_{\ga(s)}\right)\left(\nu^\S(\divS f) J^{(\S,p)}+(\kabS+\bS\kaS)(\divS f) J^{(\S,p)} \right)\\
=&\int_\S \left(\nu^\S(\divS f) J^{(\S,p)}+\left(\frac{2}{\rS}\bcS-\frac{4}{\rS}+\kabc^\S\right)(\divS f) J^{(\S,p)} \right) \\
&+\zS\big|_{\ga(s)}\int_\S\left((\zS)^{-1}-(\zS)^{-1}\big|_{\ga(s)}\right)\left(\nu^\S(\divS f)-\frac{4}{\rS}\divS f+\left(\frac{2}{\rS}\bcS+\kabc^\S\right)\divS f\right)J^{(\S,p)} \\
=& \int_\S \nu^\S(\divS f) J^{(\S,p)}-\frac{4}{\rS}\La(s)+E(s),
\end{align*}
as stated. The proof for $\Lab(s)$ is similar and left to the reader. This concludes the proof of Lemma \ref{ODElalab}.
\end{proof}
Recall that in Proposition \ref{bS} we have made the auxiliary assumption \eqref{BBbDass}, i.e.
\beaa
|\BS|+|\BbS|+|\DS| \leq \epg^{\frac{2}{3}}.
\eeaa
Moreover, recalling the definition \eqref{BBbDfirst} of $\DS$ and \eqref{D(s)} of $D(s)$ and applying Assumption {\bf A1} and \eqref{vsiOmbass}, we have
\beaa 
\left|\frac{2m}{r}-\frac{2m_{(0)}}{r}\right|&\les&\left|\ov{\zc+\Ombc}\right|+\dg\les\epg+\dg\les\epg.
\eeaa 
Recall from \eqref{2.50} that 
\beaa 
|m-\mS|&\les&\dg,
\eeaa 
thus, we obtain
\beaa
|\DS-D(s)| = \frac{2}{\rS}|\mS-m_{(0)}|\les\frac{1}{r}|m-m_{(0)}|\les\epg,
\eeaa
so we have the estimate
\beaa
|B(s)|+|\Bb(s)|+|D(s)|\les \epg^{\frac{2}{3}}.
\eeaa
The following notation is very useful in the proof of Proposition \ref{LaLabprop}.
\begin{definition}\label{good}
For a function $h$ of $(B,\Bb,D,\La,\Lab,\psi)$, we denote
\beaa
h=\good,
\eeaa
if
\beaa 
h=O(\epg),\qquad \frac{\pr h}{\pr v}=O(\epgf),\quad\forall v\in \{B,\Bb,D,\La,\Lab,\psi\}.
\eeaa
\end{definition}
\begin{lemma}\label{propgood}
We have
\bea\label{Fgood}
\dk^{\leq 1}(F\c\Ga_b)=r^{-2}\good,\qquad F\c F=r^{-2}\good,\qquad \err_1=r^{-2}\good,
\eea
where $\err_1$ is defined in Definition \ref{Definition:errorterms-prime} and Good is defined in Definition \ref{good}. We also have
\bea\label{nudivfgood}
\nu^\S (\divS f)\c\Ga_b&=r^{-3}\good.
\eea
Finally, for the following background quantities, we have
\bea\label{backgroundgood}
\left(\int_{S(\Psi(s),s)}(\div\eta) J^{(p)},\int_{S(\Psi(s),s)}(\div\xib) J^{(p)},\int_{S(\Psi(s),s)}\De(\kac) J^{(p)}\right)&=\good,
\eea
and
\bea\label{zombgood}
\ov{z+\Omb}-1-\frac{2m_{(0)}}{r}=\good. 
\eea
\end{lemma}
\begin{proof}
Notice that \eqref{Fgood} is a direct consequence of Definitions \ref{Definition:errorterms-prime}, \ref{good}, Assumption {\bf A1} and \eqref{eq:ThmGCMS1}.\\ \\
Since $\div\eta$, $\div\xib$, $\De(\kac)$ and $\Jp$ only depend on the background foliation, the quantities in \eqref{backgroundgood} only depend on $\psi$. Recalling \eqref{bootGab}, we have
\beaa 
\pr_\psi\left(\int_{S(\Psi(s),s)}(\div\eta)\Jp\right)&=&\pr_u\left(\int_{S(u,s)}(\div\eta)\Jp\right)\\
&=&\pr_u\left(\int_{\mathbb{S}}\sqrt{\det \g(s)}(\div\eta)\Jp dy^1dy^2\right)\\
&=&\int_{\mathbb{S}}\pr_u(\sqrt{\det \g(s)}\Jp)(\div\eta) dy^1dy^2+\int_{S(u,s)}\pr_u(\div\eta)\Jp\\
&=&r\dkb(\Ga_b)+r\dk^{\leq 2}(\Ga_b)=O(\epgf).
\eeaa 
Recalling from \eqref{assuml1} that $(\div\eta)_{\ell=1}=O(\dg)$, we conclude
\beaa
\int_{S(u,s)}(\div\eta)\Jp&=&\good.
\eeaa 
Similarly, using \eqref{assuml1} and \eqref{eq:GCM-improved estimate2-again}, we have
\beaa 
\int_{S(u,s)}(\div\xib)\Jp=\good,\qquad \int_{S(u,s)}(\De\kac)\Jp=\good.
\eeaa 
This concludes \eqref{backgroundgood}. \\ \\
Next, Recall that \eqref{vsiOmbass} implies
\beaa 
\ov{z+\Omb}-1-\frac{2m_{(0)}}{r}=O(\dg).
\eeaa 
Since $z,\Omb,r$ only depends on the background foliation, $\ov{z+\Omb}-1-\frac{2m_{(0)}}{r}$ only depends on $\psi$. We have
\beaa 
\pr_u\left(\ov{z+\Omb}-1-\frac{2m_{(0)}}{r}\right)=r\dk^{\leq 1}\Ga_b=O(\epgf),
\eeaa 
which implies \eqref{zombgood}.\\ \\
We have from Proposition \ref{th24} that
\bea\label{estnuF1}
\Vert \nab_{\nu^\S}^\S (F)\Vert_{\hk_{s_{max}}(\S)}&\les&\dg+|\BS|+|\BbS|+|\DS|,
\eea 
Applying Lemma \ref{commfor}, \eqref{Fgood} and \eqref{estnuF1} , we have
\beaa
\nu^\S(\divS f)&=&[\nabS_{\nu^\S},\divS]f+\divS(\nabS_{\nu^\S} f)\\
&=&\frac{2}{\rS}(\divS f)+\GabS\c\nabS_{\nu^\S}f+r^{-1}\GabS\c\dk^{\leq 1}f+r^{-1}\dkb^{\leq 1}(\nabS_{\nu^\S}f)\\
&=&\frac{2}{\rS}(\divS f)+r^{-1}\dkb^{\leq 1}(\nabS_{\nu^\S}f)+r^{-2}\good.
\eeaa
Then, we infer
\beaa 
\nu^\S(\divS f)\c\Ga_b&=&r^{-1}\dkb^{\leq 1}(f)\c\Ga_b+r^{-1}\dkb^{\leq 1}(\nabS_{\nu^\S}f)\c\Ga_b+r^{-3}\good\\
&=&r^{-3}\good,
\eeaa 
where we used \eqref{Fgood} and \eqref{estnuF1}. This concludes the proof of \eqref{nudivfgood}, and hence concludes the proof of Lemma \ref{propgood}.
\end{proof}
We are now ready to prove Proposition \ref{LaLabprop}. 
\begin{proof}[Proof of Proposition \ref{LaLabprop}]
{\bf Step 1} Firstly, we prove the following identities:
\bea
\begin{split}\label{intnuffb}
\int_{\S(s)} \nu^\S(\divS f)J^{(\S,p)}&=2B(s)+3r^{-1}\La(s)-r^{-1}\Lab(s)+\La(s)O(r^{-2})+\good,\\
\int_{\S(s)} \nu^\S(\divS \fb)J^{(\S,p)}&=2\Bb(s)+5r^{-1}\Lab(s)+r^{-1}\La(s)+O(r^{-2})(\La(s)+\Lab(s))+\good.
\end{split}
\eea
In order to prove the first identity in \eqref{intnuffb}, we write
\beaa
\nu^\S(\divS f)&=&[\nab_{\nu^\S}^\S,\divS]f+\divS (\nab_{\nu^\S}f),
\eeaa
According to \eqref{nuffb}, Proposition \ref{bS} and Lemma \ref{propgood}, we have
\beaa
\divS(\nab_{\nu^\S}f)&=&\divS\left( 2(\etaS-\eta)-\f12\ka(\fb+\bS f)+2f\omb\right)+ r^{-1}\dk^{\leq 1}( F\c \Ga_b)+\lot \\
&=& 2\divS\etaS-2\divS\eta - \frac{1}{r} \divS \fb -\frac{\bS}{r} \divS f +\frac{2m}{r^2}\divS f+r^{-1}\dk^{\leq 1}( F\c \Ga_b)+\lot \\
&=& 2\divS\etaS-2\divS\eta-\frac{1}{r}\divS\fb+\frac{1}{r}\divS f +\frac{4m}{r^2}\divS f+ r^{-3}\good.
\eeaa
Also, by Lemma \ref{commfor} and \eqref{estnuF1}, we have
\beaa
[\nab_{\nu^\S}^\S,\divS]f &=& \frac{2}{\rS}\divS f+\Ga_b\c\nabS_{\nu^\S} f+r^{-1}\Ga_b\c\dk^{\leq 1}f+\lot\\
&=&\frac{2}{\rS}\divS f+r^{-3}\good.
\eeaa
Recalling from Lemma \ref{lemma:comparison-gaS-ga} that $|\rS-r|\les \dg$, we obtain
\bea\label{nusdivsf}
\nu^\S(\divS f)=2\divS\etaS-2\divS\eta - \frac{1}{r} \divS \fb+\frac{3}{r}\divS f +\frac{4m}{r^2}\divS f+ r^{-3}\good.
\eea
Thus, recalling from Lemma \ref{lemma:comparison-gaS-ga} that $|\mS-m|\les\dg$, we obtain
\beaa
\int_{\S(s)} \nu^\S (\divS f) J^{(\S,p)}=2B(s)-2\int_{\S(s)}(\divS\eta) J^{(\S,p)}+\left(\frac{1}{3r}+\frac{4m}{r^2}\right)\La(s)-\frac{1}{r}\Lab(s)+ r^{-1}\good.
\eeaa
We claim that
\bea\label{divSetadg}
\left|\int_{\S(s)}(\divS\eta) J^{(\S,p)}-\int_{S(\Psi(s),s)}(\div\eta)\Jp\right|\les\epgf\dg.
\eea
Indeed,
\begin{align*}
& \left|\int_{\S(s)}(\divS\eta) J^{(\S,p)} -\int_{S(\Psi(s),s)}(\div\eta) J^{(p)} \right| \\
\les & \left|\int_{\S(s)}(\divS\eta-\div\eta) J^{(\S,p)}\right|+\left|\int_{\S(s)}\div\eta(\JpS-\Jp)\right|\\
&+\left|\int_{\S(s)}(\div\eta) J^{(p)}-\int_{S(\Psi(s),s)}(\div\eta) J^{(p)} \right|.
\end{align*}
The first term can be estimated by Lemma \ref{lemma:comparison-gaS-ga}. The second term can be estimated by \eqref{JJpdiff}. The last term can be estimated by Lemma \ref{Lemma:coparison-forintegrals}.
Hence,
\beaa
\left|\int_{\S(s)}(\divS\eta) J^{(\S,p)} -\int_{S(\Psi(s),s)}(\div\eta) J^{(p)} \right| \les\epgf\dg.
\eeaa
Recalling its structure and Definition \ref{good}, we obtain
\beaa
\int_{\S(s)}(\divS\eta) J^{(\S,p)} -\int_{S(\Psi(s),s)}(\div\eta) J^{(p)} =\good.
\eeaa
Applying \eqref{backgroundgood}, we deduce
\bea\label{goodeta}
\int_{\S(s)}(\divS\eta) J^{(\S,p)}=\good.
\eea 
Therefore,
\beaa
\int_{\S(s)} \nu^\S (\divS f) J^{(\S,p)}= 2B(s)+3r^{-1}\La(s)-r^{-1}\Lab(s)+\La(s)O(r^{-2})+\good.
\eeaa
Similarly, starting with
\beaa
\nab^\S_{\nu^\S}(\fb)&=&2(\xibS-\xib)-\f12\fb\kab-2\omb(\fb-\bS f)+\bS\left(2\nab^\S(\la)-\f12 \fb\ka\right)+F\cdot\Ga_b+\lot,
\eeaa
we deduce
\begin{align*}
\int_{\S(s)} \nu^\S (\divS \fb) J^{(\S,p)}=&2\Bb(s)-2\int_{\S(s)} (\divS\xib) J^{(\S,p)}+4r^{-1}\Lab(s)+2\int_{\S(s)}(\De^\S\la) J^{(\S,p)}\\
&+O(r^{-2})(\La(s)+\Lab(s))+r^{-1}\good.
\end{align*}
Similarly to \eqref{divSetadg}, we can prove
\bea\lab{divSxibdg}
\left|\int_{\S(s)}(\divS\xib)J^{(\S,p)}-\int_{S(\Psi(s),s)}(\div\xib)\Jp\right|\les\epgf\dg,
\eea
and thus,
\beaa
\int_{\S(s)}(\divS\xib)J^{(\S,p)}-\int_{S(\Psi(s),s)}(\div\xib)\Jp=\good.
\eeaa
Applying \eqref{backgroundgood}, we obtain
\beaa 
\int_{\S(s)}(\divS\xib)J^{(\S,p)}=\good.
\eeaa 
Next, we recall the transformation formula for $\ka$ in Proposition \ref{Prop:transformation-formulas-generalcasewithoutassumptions}:
\beaa 
\la^{-1}\kaS&=&\ka+\divS f+F\c\Ga_b+\lot
\eeaa 
Applying Definition \ref{good} and Lemma \ref{propgood}, recalling $\kaS=\frac{2}{\rS}$ and from Lemma \ref{lemma:comparison-gaS-ga} that $|r-\rS|\les\dg$, we infer
\beaa 
\ovla&=&\la-1=\frac{\kaS}{\ka+\divS f+F\c\Ga_b+\lot}-1\\
&=&\frac{\frac{2}{\rS}}{\frac{2}{r}+\kac+\divS f+\good}-1\\
&=&\frac{r}{\rS}\left(1-\frac{r}{2}\kac-\frac{r}{2}\divS f+\good\right)-1\\
&=&\frac{r-\rS}{\rS}-\frac{r^2}{2\rS}\kac-\frac{r^2}{2\rS}\divS f+\good\\
&=&\frac{r-\rS}{\rS}-\frac{\rS}{2}\kac-\frac{\rS}{2}\divS f+\good.
\eeaa
Thus,
\bea
\begin{split}\label{dela}
\int_{\S}(\De^\S\ovla)\JpS&=\frac{1}{\rS}\int_{\S}\De^\S(r-\rS)\JpS-\frac{\rS}{2}\int_\S\De^\S(\kac)\JpS\\
&-\frac{\rS}{2}\int_\S\De^\S\div^\S(f)\JpS+\good.
\end{split}
\eea
Similary as \eqref{divSetadg} and applying \eqref{backgroundgood}, we obtain 
\beaa
\frac{\rS}{2}\int_\S \De^\S(\kac)\JpS=\good.
\eeaa
Applying Proposition \ref{jpprop} and Lemma \ref{lemma:comparison-gaS-ga}, we infer
\beaa 
\int_\S\De^\S\div^\S(f)\JpS&=&\int_\S\divS (f)\De^\S\JpS\\
&=&-\frac{2}{(\rS)^2}\int_\S\divS(f)\JpS+\good\\
&=&-\frac{2}{r^2}\La(s)+\good.
\eeaa 
Recalling \eqref{Generalizedsystem}, we infer
\beaa 
\int_{\S}\De^\S(r-\rS)\JpS&=&\f12\int_\S\divS\left(\fb-\Up f+\err_1[\De^\S\ovb]\right)\JpS\\
&=&\f12\Lab(s)-\f12\La(s)+O(r^{-2})\La(s)+r\good.
\eeaa 
Now, from \eqref{dela} we deduce
\beaa 
\int_\S(\De^\S\ovla)\JpS=\frac{1}{2r}\Lab(s)+\frac{1}{2r}\La(s)+O(r^{-2})\La(s)+\good.
\eeaa 
Finally, we obtain
\begin{align*}
\int_{\S(s)} \nu^\S (\divS \fb) J^{(\S,p)}=2\Bb(s)+5r^{-1}\Lab(s)+r^{-1}\La(s)+O(r^{-2})(\La(s)+\Lab(s))+\good.
\end{align*}
This concludes the proof of \eqref{intnuffb}.\\ \\
{\bf Step 2} Now, we prove that the error terms $E(s),\Eb(s)$ defined in Lemma \ref{ODElalab} satisfy
\bea\label{EEbest}
E(s)=\good,\qquad \Eb(s)=\good.
\eea
Recall that
\beaa 
E(s)&=&\int_\S \left(\frac{2}{\rS}\bcS+\kabc^\S\right)(\divS f) J^{(\S,p)}  \\
&&+ \zS\big|_{\ga(s)}\int_\S \left((\zS)^{-1}-(\zS)^{-1}\big|_{\ga(s)}\right)\left(\nu^\S(\divS f)-\frac{4}{\rS}\divS f +\left(\frac{2}{\rS}\bcS+\kabc^\S\right)\divS f\right)J^{(\S,p)}.
\eeaa 
Since
\bea\label{use}
\left(\frac{2}{\rS}\bcS+\kabc^\S\right)=r^{-1} \Ga_b \c\dk^{\leq 1}(f),
\eea 
we obtain from Lemma \ref{propgood}
\bea\label{E1good}
\int_\S \left(\frac{2}{\rS}\bcS+\kabc^\S\right)(\divS f) J^{(\S,p)}=\good.
\eea 
Also, since $\zS=2+r\Ga_b$ and
\beaa
\int_\S \left((\zS)^{-1}-(\zS)^{-1}\big|_{\ga(s)}\right)\left(-\frac{4}{\rS}\divS f +\left(\frac{2}{\rS}\bcS+\kabc^\S\right)\divS f\right)J^{(\S,p)}=r^{-1}\Ga_b\c\dk^{\leq 1}(f),
\eeaa
we obtain from Lemma \ref{propgood} that
\bea\label{E2good}
\zS\big|_{\ga(s)}\int_\S \left((\zS)^{-1}-(\zS)^{-1}\big|_{\ga(s)}\right)\left(-\frac{4}{\rS}+ \frac{2}{\rS}\bcS+\kabc^\S\right)(\divS f)J^{(\S,p)}=\good.
\eea
Next, recalling $\zS=2+r\Ga_b$ and applying Lemma \ref{propgood}, we have
\bea\label{Isgood} 
\int_\S \left((\zS)^{-1}-(\zS)^{-1}\big|_{\ga(s)}\right)\nu^\S(\divS f)\JpS=\good.
\eea 
Combining \eqref{E1good}, \eqref{E2good} and \eqref{Isgood}, we deduce
\beaa 
E(s)=\good
\eeaa 
as stated. The proof for $\Eb(s)$ is similar and left to the reader. This concludes the proof of \eqref{EEbest}.\\ \\
{\bf Step 3} Recall from \eqref{Eq:ODElalab} that we have
\beaa
\nu^\S\big|_{\ga(s)}\La(s) &=& \int_\S \nu^\S(\divS f) J^{(\S,p)}-\frac{4}{\rS}\La(s)+E(s),\\
\nu^\S\big|_{\ga(s)}\Lab(s) &=& \int_\S \nu^\S(\divS \fb) J^{(\S,p)}-\frac{4}{\rS}\Lab(s)+\Eb(s).
\eeaa
Applying \eqref{Xgas}, \eqref{intnuffb} and \eqref{EEbest}, we obtain
\beaa
\frac{1}{\Psi'(s)}\La'(s)&=&\frac{C(s)}{\Psi'(s)}\left(2B(s)-r^{-1}\La(s)-r^{-1}\Lab(s)+O(r^{-2})\La(s)+\good\right),\\
\frac{1}{\Psi'(s)}\Lab'(s)&=&\frac{C(s)}{\Psi'(s)}\left(2\Bb(s)+r^{-1}\La(s)+r^{-1}\Lab(s)+O(r^{-2})(\La(s)+\Lab(s))+\good\right).
\eeaa
Recall that $z=2+O(\epg)$ and also $\la=1+O(r^{-1}\dg)$, thus, in view of \eqref{C(s)}
\beaa
\frac{C(s)}{\Psi'(s)}=\left(\frac{\la}{z}\right)\Bigg|_{\ga(s)}+O(\dg)=\frac{1+O(r^{-1}\dg)}{2+O(\epg)}+O(\dg)=\f12+O(\epg).
\eeaa
Hence,
\beaa
\frac{1}{\Psi'(s)}\La'(s)&=&\left(\frac{1}{2}+O(\epg)\right)\left(2B(s)-r^{-1}\La(s)-r^{-1}\Lab(s)+\La(s)O(r^{-2})+\good\right)\\
&=&B(s)-\f12 r^{-1} \La(s) -\f12 r^{-1} \Lab(s)+\La(s)O(r^{-2})+\good,
\eeaa
where we used
\beaa 
O(\epg)B(s)=\good.
\eeaa
Recalling that $\Psi(s)=-s+\psi(s)+c_0$, we obtain
\beaa
\frac{1}{-1+\psi'(s)}\La'(s) &=&B(s)+G(\La,\Lab,\psi)(s)+N(B,\Bb,D,\La,\Lab,\psi)(s)
\eeaa
where 
\beaa 
G(\La,\Lab,\psi)(s)=-\f12 r^{-1}\La(s)-\f12 r^{-1}\Lab(s)+O(r^{-2})\La(s),\qquad N=\good,
\eeaa 
verify the properties mentioned in Proposition \ref{LaLabprop}.\\ \\
In the same manner, we derive
\beaa
\frac{1}{-1+\psi'(s)}\Lab'(s) &=& \Bb(s) +\Gb(\La,\Lab,\psi)(s)+\Nb(B,\Bb,D,\La,\Lab,\psi)(s),
\eeaa
where
\beaa 
\Gb(\La,\Lab,\psi)(s)=\f12{{r^{-1}}}\La(s)+\f12{{r^{-1}}}\Lab(s)+O(r^{-2})(\La(s)+\Lab(s)),\qquad\Nb=\good,
\eeaa
verify the properties mentioned in Proposition \ref{LaLabprop}.\\ \\
{\bf Step 4} It remains to derive the equation of $\psi'(s)$. In view of \eqref{Eq:bSsp} and Lemma \ref{propgood}, we have
\bea\label{Eq;Psider}
\Psi'(s) &=& \left[\frac{z}{\la^2\bS+\Omb}\right]\Bigg|_{\ga(s)}+\good.
\eea
By Definition \ref{good}, $\zS=2+r\Ga_b$, $\bS=-1-\frac{2\mS}{\rS}+r\Ga_b$ and $\Omb=-\Up+r\Ga_b$, we have
\beaa
\frac{z}{\la^2\bS+\Omb} &=& \frac{z}{\bS+(\la^2-1)\bS+\Omb}=\frac{z}{\bS+\Omb}\left(\frac{1}{1+\frac{(\la^2-1)\bS}{\bS+\Omb}}\right)\\
&=&\frac{z}{\bS+\Omb}-\frac{2z\ovla\bS}{(\bS+\Omb)^2}+\good\\
&=&\frac{z}{\bS+\Omb}-\frac{4\ovla\left(-1-\frac{2\mS}{\rS}\right)}{\left(-2-\frac{2\mS}{\rS}+\frac{2m}{r}\right)^2}+\good\\
&=&\frac{z}{\bS+\Omb}+\ovla\left(1+\frac{2m}{r}\right)+\good.
\eeaa
Hence,
\beaa
\psi'(s)&=&\Psi'(s)+1=\left[\frac{\bS+z+\Omb}{\bS+\Omb}\right]\Bigg|_{\ga(s)}+\ovla\left(1+\frac{2m}{r}\right)+\good.
\eeaa
Recall that
\beaa
\ov{\bS} &=& D(s)-1-\frac{2m_{(0)}}{\rS}.
\eeaa 
Applying \eqref{zombgood}, we infer\footnote{Recall that $\ov{\bS}$ is the average of $\bS$ over $\S$ while $\ov{z+\Omb}$ is the average of $z+\Omb$ on $S(\Psi(s),s)$.}
\beaa
(\bS+z+\Omb)\big|_{\ga(s)}&=&(\bS-\ov{\bS})\big|_{\ga(s)}+\ov{\bS}+(z+\Omb-\ov{z+\Omb})\big|_{\ga(s)}+\ov{z+\Omb}\\
&=&D(s)-\left(\frac{2m_{(0)}}{\rS}-\frac{2m_{(0)}}{r}\right)\bigg|_{\ga(s)}+(\bcS-\ov{\bcS})\big|_{\ga(s)}\\
&&+(z+\Omb-\ov{z+\Omb})|_{\ga(s)}+\good\\
&=&D(s)+H_1(s)+H_2(s)+H_3(s)+\good,
\eeaa
where
\beaa 
H_1(\La,\Lab,\psi)(s)&:=&\left(\frac{2m_{(0)}}{r}-\frac{2m_{(0)}}{\rS}\right)\bigg|_{\ga(s)},\\
H_2(B,\Bb,\La,\Lab,\psi)(s)&:=&(\bcS-\ov{\bcS})\big|_{\ga(s)},\\
H_3(s)&:=&(z+\Omb-\ov{z+\Omb})|_{\ga(s)}.
\eeaa
From \eqref{eq:property9oftheoldGCMtheoremnoncanonicalell=1modes} and Assumption {\bf A1}, we deduce that $H_1$ is $O(1)$--Lipschitz on $(\La,\Lab,\psi)$.

We then consider $H_2=(\bS-\ov{\bS})\big|_{\ga(s)}$. Applying Lemma \ref{resume}, we obtain
\beaa
e_a^\S(\bS)&=&(\ze_a^\S-\eta_a^\S)\bS+\xibS_a.
\eeaa 
Taking another choice of $(\widetilde{\La},\widetilde{\Lab},\widetilde{\psi})$, we construct by Definition \ref{hypersurfaceSigma} another hypersurface $\widetilde{\Si}$. Denoting $\widetilde{X}$ the associated quantity of $X$ on $\widetilde{\Si}$, we have for $a=1,2$:\footnote{Recall that we extended $\rS$ and $\uS$ to $\RR$ in Definition \ref{hypersurfaceSigma}. Thus, $\zS$ and $\OmbS$ are well defined in $\RR$ by \eqref{zSOmbSfirst} and hence $\bS$ is also well defined in $\RR$ by $\bS=-\zS-\OmbS$.}
\begin{equation*}
    e_a^\S(\bS)-e_a^{\widetilde{\S}}(\widetilde{\bS})=(\ze_a^\S-\eta_a^\S)(\bS-\widetilde{\bS})+\left((\ze_a^\S-\widetilde{\zeS_a})-(\etaS_a-\widetilde{\etaS_a})\right)\widetilde{\bS}+\xibS_a-\widetilde{\xibS_a}.
\end{equation*}
Applying Lemma \ref{Lemma:Generalframetransf}, \eqref{eq:property7oftheoldGCMtheoremnoncanonicalell=1modes} and elliptic estimate, we deduce
\beaa
(\bS-\ov{\bS})-(\widetilde{\bS}-\ov{\widetilde{\bS}}) &=& O(r)\left(\etaS_a-\widetilde{\etaS_a},\xibS_a-\widetilde{\xibS_a}\right)+O(1)(\La-\widetilde{\La},\Lab-\widetilde{\Lab},\psi-\widetilde{\psi}).
\eeaa
which implies that $\bS-\ov{\bS}$ is $O(r)$--dependent on $\etaS$ and $\xibS$. Next, we proceed as in the proofs of Lemma \ref{etaxibomb} and Proposition \ref{bS} to deduce
\begin{align*}
    \ddsStwo\etaS_{ab}-\widetilde{\ddsStwo}\widetilde{\etaS_{ab}}&=O(r^{-1})(\La-\widetilde{\La},\Lab-\widetilde{\Lab},\psi-\widetilde{\psi}),\\
    \curl^\S\etaS-\widetilde{\curl^\S}\widetilde{\etaS}&=O(r^{-1})(\La-\widetilde{\La},\Lab-\widetilde{\Lab},\psi-\widetilde{\psi}).
\end{align*}
Applying Lemmas \ref{elliptic2}, \ref{Lemma:Generalframetransf} and \ref{lemma:comparison-gaS-ga}, we deduce that 
\beaa 
(\bS-\ov{\bS})-(\widetilde{\bS}-\ov{\widetilde{\bS}})=O(1)(B-\widetilde{B},\Bb-\widetilde{\Bb},\La-\widetilde{\La},\Lab-\widetilde{\Lab},\psi-\widetilde{\psi}).
\eeaa 
Thus, $H_2$ is $O(1)$--Lipschitz on $(B,\Bb,\La,\Lab,\psi)$ but independent on $D$. Moreover, 
\beaa
H_1&=&\frac{2m_{(0)}(\rS-r)}{r\rS}=r^{-2}O(\dg),\\
H_2&=&\bS-\ov{\bS}=\bcS-\ov{\bcS}=O(\epg),\\
H_3&=&r\Ga_b=O(\epg).
\eeaa
On the other hand, we have
\beaa
(\bS+\Omb)\big|_{\ga(s)}&=&D(s)-1-\frac{2m_{(0)}}{\rS}+\left(\bS\big|_{\ga(s)}-\ov{\bS}\right)+\Ombc|_{\ga(s)}-\Up|_{\ga(s)}\\
&=& D(s)-2+H_2(s)+H_4(s),
\eeaa
where
\beaa
H_4(\La,\Lab,\psi)(s):=\left(\frac{2m}{r}-\frac{2m_{(0)}}{\rS}\right)\Big|_{\ga(s)}+\Ombc|_{\ga(s)}.
\eeaa
By Assumption {\bf A2}, \eqref{eq:ThmGCMS3}, \eqref{eq:ThmGCMS5} and \eqref{eq:ThmGCMS6}, we have
\beaa 
H_4(s)=O(\dg)+r\Ga_b=O(\epg).
\eeaa
Remark that $H_4(s)$ is $O(1)$--Lipschitz on $(\La,\Lab,\psi)$.

Then, we infer
\beaa
\psi'(s)&=&\frac{\bS+z+\Omb}{\bS+\Omb}\bigg|_{\ga(s)}+\ovla\left(1+\frac{2m}{r}\right)+\good\\
&=&\frac{D(s)+H_1(s)+H_2(s)+H_3(s)+\good}{D(s)-2+H_2(s)+H_4(s)}+\ovla\left(1+\frac{2m}{r}\right)+\good\\
&=&\left(D(s)+H_1(s)+H_2(s)+H_3(s)+\good\right)\left(-\f12+O(D(s),H_2(s),H_4(s))\right)\\
&&+\ovla\left(1+\frac{2m}{r}\right)+\good\\
&=&-\f12D(s)-\f12\left(H_1(s)+H_2(s)+H_3(s)\right)+O({\bf H}(s)^2) +\ovla\left(1+\frac{2m}{r}\right)\\
&+&O(D(s){\bf H}(s))+O(D(s)^2)+\good,
\eeaa 
where
\beaa 
{\bf H}(s)&:=&(H_1(s),H_2(s),H_3(s),H_4(s))
\eeaa 
is $O(1)$--Lipschitz on $(B,\Bb,\La,\Lab,\psi)$. Thus, we can write
\beaa 
\psi'(s)=-\f12 D(s)+H(B,\Bb,\La,\Lab,\psi)(s)+M(B,\Bb,D,\La,\Lab,\psi)(s),
\eeaa 
where 
\beaa 
H(B,\Bb,\La,\Lab,\psi)(s)&=&-\f12\left(H_1(s)+H_2(s)+H_3(s)\right)+O({\bf H}(s)^2) +\ovla\left(1+\frac{2m}{r}\right),\\
M(B,\Bb,D,\La,\Lab,\psi)(s)&=&O(D(s){\bf H}(s))+O(D(s)^2)+\good,
\eeaa 
verify the properties stated in Proposition \ref{LaLabprop}. This concludes the proof of Proposition \ref{LaLabprop}.
\end{proof}
\subsection{Proof of Theorem \ref{mthm}}\lab{sec4.4}
\noindent{\bf Step 1.} From Proposition \ref{LaLabprop}, we have the following closed system of equations in $(\La,\Lab,\psi)$:
\bea
\begin{split}\label{system}
\frac{1}{-1+\psi'(s)}\La'(s)=&B(s)+G(\La,\Lab,\psi)(s)+N(B,\Bb,D,\La,\Lab,\psi)(s),\\
\frac{1}{-1+\psi'(s)}\Lab'(s)=&\Bb(s) +\Gb(\La,\Lab,\psi)(s)+\Nb(B,\Bb,D,\La,\Lab,\psi)(s),\\
\psi'(s)=& -\f12 D(s)+H(B,\Bb,\La,\Lab,\psi)(s)+M(B,\Bb,D,\La,\Lab,\psi)(s),
\end{split}
\eea
with initial conditions
\bea\lab{initialcon}
\psi(\ovs)=0,\qquad \La(\ovs)=\La_0,\qquad \Lab(\ovs)=\Lab_0,
\eea
see Proposition \ref{LaLabprop} for the properties of $G,\Gb,H,N,\Nb,M$.\\ \\
The system \eqref{system} is verified for all hypersurface $\Si$ as in Definition \ref{hypersurfaceSigma}, and hence for any choice of $(\La,\Lab,\psi)$. We now make a suitable choice to obtain the GCM hypersurface $\Si_0$. \\ \\
Consider in particular the system obtained from \eqref{system} by setting $B,\Bb,D$ to zero, i.e.
\begin{align}
\begin{split}\label{system0}
\frac{1}{-1+\psibr'(s)}\Labr'(s)=&G(\Labr,\Labbr,\psibr)(s)+\Nbr(\Labr,\Labbr,\psibr)(s),\\
\frac{1}{-1+\psibr'(s)}\Labbr'(s)=&\Gb(\Labr,\Labbr,\psibr)(s)+\Nbbr(\Labr,\Labbr,\psibr)(s),\\
\psibr'(s)=&H(0,0,\Labr,\Labbr,\psibr)(s)+\Mbr(\Labr,\Labbr,\psibr)(s),
\end{split}
\end{align}
where
\beaa
\Nbr(\Labr,\Labbr,\psibr)(s)&:=&N(0,0,0,\Labr,\Labbr,\psibr)(s),\\
\Nbbr(\Labr,\Labbr,\psibr)(s)&:=&\Nb(0,0,0,\Labr,\Labbr,\psibr)(s),\\
\Mbr(\Labr,\Labbr,\psibr)(s)&:=&M(0,0,0,\Labr,\Labbr,\psibr)(s).
\eeaa
We initialize the system at $s=\ovs$ as in \eqref{initialcon}, i.e.
\beaa
\psibr(\ovs)=0,\qquad \Labr(\ovs)=\La_0,\qquad \Labbr(\ovs)=\Lab_0.
\eeaa
Proposition \ref{LaLabprop} implies that $G,\Gb,H,{\Nbr},{\Nbbr},{\Mbr}$ are $O(1)$--Lipschitz functions of $(\Labr,\Labbr,\psibr)$. Applying Cauchy-Lipschitz theorem, we deduce that the system admit a unique solution $(\Labr(s),\Labbr(s),\psibr(s))$ defined in a small neighborhood $\overset{\circ}{I}$ of $\ovs$, which satisfies
\beaa
|(\Labr(s),\Labbr(s),\psibr(s))| &\les& \dg.
\eeaa
The desired hypersurface $\Si_0$ is given from this solution $(\Labr(s),\Labbr(s),\psibr(s))$ by:
\beaa 
\Si_0:=\bigcup_{s\in\overset{\circ}{I}}\S[\Psibr(s),s,\Labr(s),\Labbr(s)].
\eeaa 
\noindent{\bf Step 2.} Next, we show that the functions $B,\Bb,D$ vanish on the hypersurface $\Si_0$ defined above. Since the system \eqref{system} is verified for all functions and triplets $(\La,\Lab,\psi)$ and in particular by $(\Labr,\Labbr,\psibr)$, and since $ (\Labr,\Labbr,\psibr)$ satisfies \eqref{system0}, we deduce, taking the difference of \eqref{system} and \eqref{system0},
\bea
\begin{split}\label{lastDBBb}
-\f12 \DDbr(s)+H(\BBbr,\Bbbr,\Pbr)(s)+M(\BBbr,\BBbr,\DDbr,\Pbr)(s)&=H(0,0,\Pbr)(s)+M(0,0,0,\Pbr)(s),\\
\BBbr(s)+N(\BBbr,\Bbbr,\DDbr,\Pbr)(s)&=N(0,0,0,\Pbr)(s),\\
\Bbbr(s)+\Nb(\BBbr,\Bbbr,\DDbr,\Pbr)(s)&=\Nb(0,0,0,\Pbr)(s),    
\end{split}
\eea
where
\beaa 
\Pbr&:=&(\Labr,\Labbr,\psibr).
\eeaa 
The first equation and the properties of $M$ and $H$ stated in Proposition \ref{LaLabprop} imply that
\beaa
|\DDbr(s) | &\les& |H(\BBbr,\Bbbr,\Pbr)(s)-H(0,0,\Pbr)(s)|+|M(\BBbr,\Bbbr,\DDbr,\Pbr)(s)-M(0,0,0,\Pbr)(s)|\\
&\les&\sup_{\overset{\circ}{I}}|\BBbr(s)|+\sup_{\overset{\circ}{I}}|\Bbbr(s)|+\epgf\sup_{\overset{\circ}{I}}|\DDbr(s)|.
\eeaa
Taking the supreme over ${\overset{\circ}{I}}$, for $\epg$ small enough, we obtain
\bea\label{DleqBBb}
\sup_{\overset{\circ}{I}}|\DDbr(s)|&\les& \sup_{\overset{\circ}{I}}|\BBbr(s)|+\sup_{\overset{\circ}{I}}|\Bbbr(s)|.
\eea
The properties of $N,\Nb$ imply, together with \eqref{DleqBBb}, that
\beaa
|N(\BBbr,\Bbbr,\DDbr,\Pbr)(s)- N(0,0,0,\Pbr)(s)|&\les&\epgf \left(\sup_{\overset{\circ}{I}}|\BBbr(s)|+\sup_{\overset{\circ}{I}} |\Bbbr(s)|\right),\\
|\Nb(\BBbr,\Bbbr,\DDbr,\Pbr)(s)- \Nb(0,0,0,\Pbr)(s)|&\les& \epgf \left(\sup_{\overset{\circ}{I}}|\BBbr(s)|+\sup_{\overset{\circ}{I}}|\Bbbr(s)|\right).
\eeaa
Hence, the second and third equations of \eqref{lastDBBb} imply that
\beaa
\sup_{\overset{\circ}{I}}|\BBbr(s)|+\sup_{\overset{\circ}{I}}|\Bbbr(s)|\les \epg^\f12 \left(\sup_{\overset{\circ}{I}}|\BBbr(s)|+\sup_{\overset{\circ}{I}}|\Bbbr(s)|\right),
\eeaa
and thus, for $\epg$ small enough,
\bea\label{vanishBBb}
\BBbr=0,\qquad \Bbbr=0,\quad \mbox{ on }\Si_0.
\eea 
Finally, \eqref{DleqBBb} and \eqref{vanishBBb} imply that
\bea\label{vanishD}
\DDbr=0,\qquad \mbox{ on }\Si_0.
\eea
In view of the definition of $B$, $\Bb$ and $D$, this implies that the desired identities \eqref{4.7} and \eqref{4.10} hold on $\Si_0$.\\ \\
\noindent{\bf Step 3.} $\Si_0$ satisfies the statements 1,2,4,5 of Theorem \ref{mthm} by Proposition \ref{propositionSigma}. Also, it satisfies the statements 3,6 of Theorem \ref{mthm} by Step 2. Thus, it only remains to derive the estimates in the statement 7 which we recall below for convenience:
\bea\label{recall7}
\Vert (f,\fb,\ovla)\Vert_{\hk_{s_{max}+1}(\S)}+\Vert\dk(f,\fb,\ovla)\Vert_{\hk_{s_{max}}(\S)}\les \dg,
\eea
To this end, we first apply Proposition \ref{th24} to obtain
\bea\label{nuFdgcor}
\Vert F\Vert_{\hk_{s_{max}+1}(\S)}+\Vert\nab^\S_{\nu} F\Vert_{\hk_{s_{max}}(\S)} &\les& \dg,
\eea
where we used \eqref{vanishBBb} and \eqref{vanishD}.
Next, we recall the following null transformation formulae \eqref{e4ffbla}
\beaa 
\la^{-2}\xi^\S &=& \xi + \f12 \nabS_{\la^{-1}e_4^\S}f+r^{-1}O(F)+\lot,\\
\zeS&=& \ze- \nabS(\log \la)+r^{-1}O(F)+\lot,\\
\etabS&=& \etab+\f12\nabS_{\la^{-1}e^\S_4}\fb+r^{-1}O(F)+\lot,\\
\la^{-1}\omS&=& \om -\f12 \la^{-1} e_4^\S(\log\la)+r^{-1}O(F)+\lot,
\eeaa 
where $\lot$ denotes terms decay better. Combining with \eqref{transversalityru}, we deduce
\bea\label{e4Fdg}
\| \nabS_{e_4^\S}(F)\|_{\hk_{s_{max}}(\S)}\les r^{-1}\| F\|_{\hk_{s_{max}+1}(\S)}\les r^{-1}\dg.
\eea
Then, we have from \eqref{nuFdgcor} and \eqref{e4Fdg}
\begin{equation} \label{e3Fdg}
\| \nabS_{e_3^\S}(F)\|_{\hk_{s_{max}}(\S)}\les \| \nabS_{\nu^\S}(F)\|_{\hk_{s_{max}}(\S)}+ \|\bS \nabS_{e_4^\S}(F)\|_{\hk_{s_{max}}(\S)}\les\dg.
\end{equation}
Combining \eqref{nuFdgcor}, \eqref{e4Fdg} and \eqref{e3Fdg}, we obtain \eqref{recall7} as stated. This concludes the proof of Theorem \ref{mthm}.
\subsection{Proof of Corollary \ref{coro25}}\label{sec4.5}
{\bf Step 1.} Consider first the simpler case where
\beaa
\|(f,\fb,\ovla)\|_{\hk_{s_{max}+1}(\S_0)}&\les&\dg,
\eeaa
so that the estimates \eqref{Fest} holds true for $\S_0$. We then proceed exactly as Proposition \ref{th24} to derive the estimates for $\nab^{\S_0}_\nu(f,\fb,\ovla)$ for our distinguished sphere $\S_0$. Note that $\S_0$ can be viewed as a deformation of the unique background sphere sharing the same south pole. In the sequel, we denote
\beaa 
\S:=\S_0,\qquad \nu:=\nu^\S=\nu^{\S_0}.
\eeaa 
The proof of the estimates for $(\nabS_\nu)^l(f,\fb,\ovla)$ in the cases $l\geq 2$ is similar so we only provide a sketch:
\begin{enumerate}
\item Elliptic estimates.\\
We commute the GCM system \eqref{eqGCMl} by $(\nabS_\nu)^{l}(F)$ and obtain the analog of \eqref{nuGCM} for $(\nabS_\nu)^l F$. All the commutators can be estimated by Lemma \ref{commfor} as before, remark that
\beaa 
\nabS_\nu\left([\nabS_\nu,\divS]f\right)&=&\nabS_\nu\left(\frac{2}{\rS}\divS f+\GabS\c\nabS_\nu(f)\right)+\lot\\
&=&\nu\left(\frac{2}{\rS}\right)(\divS f)+\frac{2}{\rS}\nabS_\nu(\divS f)\\
&+&\nabS_\nu(\GabS)\c\nabS_\nu(f)+\GabS\c(\nabS_\nu)^2(f)+\lot
\eeaa 
Thus, in the cases $l\geq 2$, we obtain the new terms $(\nabS_\nu)^{\leq l-1}(\etaS,\xibS,\ombS)$ from the commutators. Notice also that
\beaa
\nu^2\left(\frac{2}{\rS}\ovla\right)&=&\frac{2}{\rS}\nu^2(\ovla)+2\nu\left(\frac{2}{\rS}\right)\nu(\ovla)+\nu^2\left(\frac{2}{\rS}\right)\ovla\\
&=&\frac{2}{\rS}\nu^2(\ovla)+r^{-2}\nu(F)-2\nu\left(\frac{\bS+\OmbS}{(\rS)^2}\right)\ovla\\
&=&\frac{2}{\rS}\nu^2(\ovla)+r^{-2}\nu(F)+r^{-2}\nu(\bS+\OmbS)F+r^{-3}F.
\eeaa
Recalling \eqref{zbO}, in the cases $l\geq 2$, we have the new terms $(\nabS_\nu)^{\leq l-1}(\bS,\zS,\OmbS)$ in the R.H.S. of the analog of \eqref{nuGCM} for $(\nabS_\nu)^{l}(F)$. Applying Proposition \ref{Thm.GCMSequations-fixedS:contraction}, we infer
\bea
\begin{split}\label{ffbovlacl}
&\left\Vert (\nab^\S_{\nu})^l (f,\fb),\nu^l(\ovla)-\ov{\nu^l(\ovla)}^\S \right\Vert_{\hk_{K-l}(\S)}\les\dg+ \left|\left( \divS (\nab^\S_{\nu})^l f \right)_{\ell=1} \right|+\left|\left(\divS (\nab^\S_{\nu})^l\fb\right)_{\ell=1}\right|\\
&+\dg r^{-1} \sum_{j=1}^l\Vert \nu^{j-1}(\bS,\zS,\OmbS)\Vert_{\hk_{K-j+1}(\S)}+\dg\sum_{j=1}^{l}\Vert\nu^{j-1}(\etaS,\xibS,\ombS)\Vert_{\hk_{K-j+1}(\S)} \\
&+\sum_{j=1}^l\Vert(\nab^\S_{\nu})^{j-1}(F)\Vert_{\hk_{K-j+1}(\S)}+r^{-1}\sum_{j=1}^l\Vert(\nab^\S_{\nu})^{j-1}(\ovb)\Vert_{\hk_{K-j+1}(\S)}\\
&+(\epg+r^{-1})\left(\Vert(\nabS_\nu)^l(F)\Vert_{\hk_{K-l}(\S)}+r^{-1}\Vert\nu^l(\ovb)\Vert_{\hk_{K-l}(\S)}\right),
\end{split}
\eea
and
\bea
\begin{split}\label{ovovlal}
&r\left|\ov{\nu^l(\ovla)}^\S\right|\les\dg+\left|\left( \divS (\nab^\S_{\nu})^l f\right)_{\ell=1} \right|+\left|\left(\divS (\nab^\S_{\nu})^l\fb\right)_{\ell=1} \right|+|b_0^{(l)}|\\
&+\dg r^{-1} \sum_{j=1}^l\Vert \nu^{j-1}(\bS,\zS,\OmbS)\Vert_{\hk_{K-j+1}(\S)}+\dg\sum_{j=1}^{l}\Vert\nu^{j-1}(\etaS,\xibS,\ombS)\Vert_{\hk_{K-j+1}(\S)} \\
&+\sum_{j=1}^l\Vert(\nab^\S_{\nu})^{j-1}(F)\Vert_{\hk_{K-j+1}(\S)}+r^{-1}\sum_{j=1}^l\Vert(\nab^\S_{\nu})^{j-1}(\ovb)\Vert_{\hk_{K-j+1}(\S)}\\
&+(\epg+r^{-1})\left(\Vert(\nabS_\nu)^l(F)\Vert_{\hk_{K-l}(\S)}+r^{-1}\Vert\nu^l(\ovb)\Vert_{\hk_{K-l}(\S)}\right),
\end{split}
\eea
where
\beaa 
b_0^{(l)}&:=&\ov{\nu^l(\ovb)}^\S.
\eeaa 
Applying Lemma \ref{elliptic1} for the equation of $\nu^l(\ovb)$, we obtain
\bea\label{ovb0nul}
\begin{split}
&\left\Vert\nu^l(\ovb)-\ov{\nu^l(\ovb)}^\S\right\Vert_{\hk_{K-l}(\S)}\les r\dg+r\Vert (\nabS_\nu)^l(F)\Vert_{\hk_{K-l}(\S)}+(\epg+r^{-1})\Vert\nu^l(\ovb)\Vert_{\hk_{K-l}(\S)}\\
&+\dg r^{-1}\sum_{j=1}^l\Vert \nu^{j-1}(\bS,\zS,\OmbS)\Vert_{\hk_{K-j+1}(\S)}+\dg\sum_{j=1}^{l}\Vert\nu^{j-1}(\etaS,\xibS,\ombS)\Vert_{\hk_{K-j+1}(\S)} \\
&+\sum_{j=1}^l\Vert(\nab^\S_{\nu})^{j-1}(F)\Vert_{\hk_{K-j+1}(\S)}+r^{-1}\sum_{j=1}^l\Vert(\nab^\S_{\nu})^{j-1}(\ovb)\Vert_{\hk_{K-j+1}(\S)}.\\
\end{split}
\eea
\item Estimates for the $\ell=1$ modes of $(\nabS_\nu)^l(f,\fb)$.\\
Applying Lemmas \ref{commfor}, \ref{dint}, \ref{nunununu} and $B=0$, we infer, since $\nu$ is tangent to $\Si_0$,
\bea
\begin{split}\label{nuB}
0&=\nu(B(s))=\nu\left(\int_\S(\divS\etaS)\JpS\right)\\
&=\zS\int_\S(\zS)^{-1}\left(\nu\left((\divS\etaS)\JpS\right)+(\kabS+\bS\kaS)(\divS\etaS)\JpS\right)\\
&=\zS\int_\S(\zS)^{-1}\left(\divS(\nabS_\nu\etaS)+[\nabS_\nu,\divS]\etaS+(\kabS+\bS\kaS)(\divS\etaS)\right)\JpS.    
\end{split}
\eea
Recall from \eqref{divetadivxibpr} that we have
\beaa
\nu\left(\int_S (\div\eta)\Jp\right)&=&O(\dg).
\eeaa
Similarly, applying Lemmas \ref{commfor}, \ref{lemma:comparison-gaS-ga}, \ref{Lemma:coparison-forintegrals}, \ref{dint}, \ref{nunununu} and
\beaa 
(\bS-b,\,\kabS-\kab,\,\zS-z,\,\OmbS-\Omb,\,\kaS-\ka)=O(\dg),
\eeaa 
which is proved in Proposition \ref{th24}, combining with \eqref{nuB} and the fact that $\JpS$ verifies \eqref{JJpdiff}, we deduce
\beaa 
\int_\S \left(\divS(\nabS_\nu(\etaS-\eta))+[\nabS_\nu,\divS](\etaS-\eta)+(\kabS+\bS\kaS)\divS(\etaS-\eta)\right)\JpS=O(\dg).
\eeaa 
Recalling from Lemma \ref{nunununu} that
\beaa 
\nabS_\nu(f)&=&\etaS-\eta+r^{-1}O(F)+\lot,\\
(\nabS_\nu)^2 (f)&=&\nabS_\nu(\etaS-\eta)+r^{-1}O(F)\nu(\bS)+r^{-1}\nabS_\nu(F)+\lot,
\eeaa 
we deduce
\beaa
(\kabS+\bS\kaS)(\divS(\etaS-\eta))=(\kabS+\bS\kaS)(\divS\nabS_\nu(f))+r^{-3}O(\dg)=r^{-3}O(\dg),
\eeaa
and from Lemma \ref{commfor} that
\beaa 
[\nabS_\nu,\divS](\etaS-\eta)&=&\frac{2}{\rS}\divS(\etaS-\eta)+\GabS\c\nabS_\nu(\etaS-\eta)+\lot\\
&=&r^{-3}O(\dg)+O(\epg)r^{-1}(\nabS_\nu)^2(F)+r^{-3}O(\dg)\nu(\bS).
\eeaa
Thus, we obtain
\beaa 
\int_\S \left(\divS(\nabS_\nu(\etaS-\eta)))\right)\JpS=O(\dg)+r O(\epg)(\nabS_\nu)^2(F)+O(\dg)\nu(\bS).
\eeaa 
Hence, we have
\beaa 
\left|\int_\S \divS (\nab^\S_{\nu})^2(f) \JpS \right|&\les& \dg+r^{-1}\dg\Vert\nu(\bS,\zS,\OmbS)\Vert_{\hk_{K-1}(\S)}+\Vert (\nab^\S_{\nu})F\Vert_{\hk_{K-1}(\S)}\\
&&+\epg\Vert (\nab^\S_{\nu})^{2}(F)\Vert_{\hk_{K-2}(\S)}.
\eeaa 
Similarly, we have
\beaa 
\left|\int_\S \divS (\nab^\S_{\nu})^2(\fb) \JpS \right|&\les& \dg+r^{-1}\dg\Vert\nu(\bS,\zS,\OmbS)\Vert_{\hk_{K-1}(\S)}+\Vert (\nab^\S_{\nu})(F)\Vert_{\hk_{K-1}(\S)}\\
&&+\epg\Vert (\nab^\S_{\nu})^{2}(F)\Vert_{\hk_{K-2}(\S)}.
\eeaa 
These are the desired estimates for $l=2$. The cases of $l\geq 3$ are similar. Hence, we have for $l\geq 2$
\bea\label{ffbnul}
\begin{split}
&\left|\left(\divS (\nab^\S_{\nu})^l(f)\right)_{\ell=1}\right|+\left|\left(\divS (\nab^\S_{\nu})^l(\fb)\right)_{\ell=1}\right|\\
&\les\dg+r^{-1}\dg\sum_{j=1}^l\Vert\nu^{j-1}(\bS,\zS,\OmbS)\Vert_{\hk_{K-j+1}(\S)}+\sum_{j=1}^l\Vert(\nab^\S_{\nu})^{j-1}(F)\Vert_{\hk_{K-j+1}(\S)}\\
&+\epg\Vert(\nab^\S_{\nu})^{l}(F)\Vert_{\hk_{K-l}(\S)}.    
\end{split}
\eea
\item Boundedness of $\nu^{\leq l-1}(\bS,\zS,\OmbS)$ and $(\nabS_\nu)^{\leq l-1}(\etaS,\xibS,\ombS)$.\\ 
Recall that the $\ell=1$ modes of $\nabS_\nu(\etaS,\xibS)$ are estimated in the previous step. Then, we commute the equations in the proof of Lemmas \ref{etaxibomb} and \ref{ombpjz} with $\nu$ to deduce
\begin{align}
\begin{split}\label{nuetaxibomb}
&\left\|\nabS_\nu(\etaS),\nabS_\nu(\xibS),\nabS_\nu(\ombS)\right\|_{\hk_{K-1}(\S)}\\
&\les \epg +r^{-1}\dg\sum_{j=1}^2\Vert\nu^{j-1}(\bS,\zS,\OmbS)\Vert_{\hk_{K-j+1}(\S)}+\sum_{j=1}^2\Vert (\nab^\S_{\nu})^{j}(F)\Vert_{\hk_{K-j}(\S)}.    
\end{split}
\end{align}
We commute equations \eqref{zombsea} and \eqref{zbO} with $\nu$ to obtain
\begin{align}
\begin{split}\label{boundednesseveryone}
&\left\|\nu(\bS)-\ov{\nu(\bS)},\nu(\zS)-\ov{\nu(\zS)},\nu(\OmbS)-\ov{\nu(\OmbS)}\right\|_{\hk_{K-1}(\S)}\\
&\les r \left\|\nabS_\nu(\etaS),\nabS_\nu(\xibS)\right\|_{\hk_{K-1}(\S)}\\
&\les\epg r+\dg\sum_{j=1}^2\Vert\nu^{j-1}(\bS,\zS,\OmbS)\Vert_{\hk_{K-j+1}(\S)}+r\sum_{j=1}^2\Vert (\nab^\S_{\nu})^{j}(F)\Vert_{\hk_{K-j}(\S)}.
\end{split}
\end{align}
To estimate $\ov{\nu(\bS)}$, we apply Lemma \ref{dint} and $D=0$ to deduce
\beaa 
\nu(\ov{\bS})&=&-\nu\left(\frac{2\mS}{\rS}\right)=-\frac{1}{16\pi}\nu\left(\int_\S \kaS\kabS\right)=O(1).
\eeaa 
On the other hand
\beaa 
\nu(\ov{\bS})&=&\frac{1}{|S|}\nu\left(\int_\S\bS\right)-\frac{\nu(|S|)}{|S|}\ov{\bS}\\
&=&\frac{1}{|S|}\zS\int_\S(\zS)^{-1}\left(\nu(\bS)+(\kabS+\bS\kaS)\bS\right)+O(1)\\
&=&\frac{1}{|S|}\zS\int_\S \ov{(\zS)^{-1}}\left(\nu(\bS)\right)+O(1).
\eeaa 
Combining these three estimates, we obtain
\bea\label{nubsbound}
\ov{\nu(\bS)}=O(1).
\eea 
To prove the boundedness of $\ov{\nu(\zS)}$, we start from 
\beaa
-\zS=\OmbS+\bS=\nu(\rS)=\frac{\rS}{2}\zS\ov{(\zS)^{-1}(\kaS+\bS\kabS)},
\eeaa
which implies
\beaa
\int_\S (\zS)^{-1}(\kaS+\bS\kabS)=-8\pi\rS.
\eeaa
Hence, we have
\bea\label{zsdiff}
\nu\left(\int_\S (\zS)^{-1}(\kaS+\bS\kabS)\right)=-8\pi\nu(\rS)=O(1).
\eea
Then, applying Lemma \ref{dint} and \eqref{nubsbound}, we can deduce $$|\ov{\nu((\zS)^{-1})}|\les 1+\sum_{j=1}^2\Vert (\nab^\S_{\nu})^{j}(F)\Vert_{\hk_{K-j}(\S)},$$
details are left to the reader. Then \eqref{boundednesseveryone} and \eqref{zbO} imply that, for $\dg$ small enough
\bea
r^{-1}\left\|\nu\left(\bS,\zS,\OmbS\right)\right\|_{\hk_{K-1}(\S)}&\les&1+\sum_{j=1}^2\Vert (\nab^\S_{\nu})^{j}(F)\Vert_{\hk_{K-j}(\S)}.
\eea
Similarly, we can prove that
\bea
\begin{split}\label{boundedbs}
&\left\|(\nabS_\nu)^{l-1}(\etaS,\xibS,\ombS)\right\|_{\hk_{K-l}(\S)}+r^{-1}\left\|\nu^{l-1}(\bS,\zS,\OmbS)\right\|_{\hk_{K-l}(\S)}\\
&\les 1+\sum_{j=1}^l\Vert (\nab^\S_{\nu})^{j}(F)\Vert_{\hk_{K-j}(\S)}.    
\end{split}
\eea
Notice that \eqref{ffbovlacl}, \eqref{ffbnul} and \eqref{boundedbs} imply
\bea 
\begin{split}\label{nulllffbovla}
&\left\Vert (\nab^\S_{\nu})^l (f,\fb),\nu^l(\ovla)-\ov{\nu^l(\ovla)}^\S \right\Vert_{\hk_{K-l}(\S)} \\
&\les\dg+\sum_{j=1}^l\Vert(\nab^\S_{\nu})^{j-1}(F)\Vert_{\hk_{K-j+1}(\S)}+r^{-1}\sum_{j=1}^l\Vert(\nab^\S_{\nu})^{j-1}(\ovb)\Vert_{\hk_{K-j+1}(\S)}\\
&+(\epg+r^{-1})\left(\Vert(\nabS_\nu)^l(F)\Vert_{\hk_{K-l}(\S)}+r^{-1}\Vert\nu^l(\ovb)\Vert_{\hk_{K-l}(\S)}\right).
\end{split}
\eea
\item Estimate for $b_0^{(l)}$.\\
We need to prove that
\bea\label{b0l}
b_0^{(l)}:=\ov{\nu^l(\ovb)}^\S=\ov{\nu^l(r-\rS)}=O(\dg).
\eea 
Similar to \eqref{nbsl}, we can deduce
\begin{align*}
\nu^l(r-\rS)&=-2\nu^{l-1}\left(z\ov{z^{-1}}-\zS\ov{(\zS)^{-1}}\right)+O(\epg)\ov{\nu^{l-1}(z-\zS)}\\
&+O(\EE_l)+O\left((\nabS_\nu)^{\leq l-1}(F)\right)+O({\nu^{\leq l-1}(r-\rS)}),
\end{align*}
where $\EE_l$ is defined in \eqref{defEE}. Similar to \eqref{zzSrrS}, we can obtain
\beaa
\ov{\nu^{l-1}(z-\zS)}=O(\epg)\ov{\nu^{l-1}(z-\zS)}+O(\EE_l).
\eeaa
Thus, we have
$$\ov{\nu^{l-1}(z-\zS)}=O(\EE_l).$$
By iteration, we conclude \eqref{b0l}. The details of the proof is tedious but straightforward and is similar to the case of $l=1$. We leave to the reader. Notice that \eqref{ovovlal}, \eqref{ffbnul}, \eqref{boundedbs} and \eqref{b0l} imply
\bea
\begin{split}\label{nulllovovla}
\left\Vert\ov{\nu^l(\ovla)}^\S \right\Vert_{\hk_{K-l}(\S)}&\les\dg+\sum_{j=1}^l\Vert(\nab^\S_{\nu})^{j-1}(F)\Vert_{\hk_{K-j+1}(\S)}+r^{-1}\sum_{j=1}^l\Vert(\nab^\S_{\nu})^{j-1}(\ovb)\Vert_{\hk_{K-j+1}(\S)}\\
&+(\epg+r^{-1})\left(\Vert(\nabS_\nu)^l(F)\Vert_{\hk_{K-l}(\S)}+r^{-1}\Vert\nu^l(\ovb)\Vert_{\hk_{K-l}(\S)}\right).
\end{split}
\eea 
\item Conclusion.\\
Combining all the above estimates, we obtain for $\epg$ small enough and $\ovr$ large enough,
\beaa
\Vert (\nab^\S_{\nu})^l F\Vert_{\hk_{K-l}(\S)}+r^{-1}\Vert\nu^l(\ovb)\Vert_{\hk_{K-l}(\S)}\les\dg+\sum_{j=0}^{l-1} \Vert(\nabS_{\nu})^j F\Vert_{\hk_{K-j}(\S)}+r^{-1}\sum_{j=0}^{l-1}\Vert\nu^j(\ovb)\Vert_{\hk_{K-j}(\S)},
\eeaa
which, together with \eqref{Fest}, yields by iteration the desired estimates for all tangential derivatives
\bea\label{finalF}
\Vert (\nab^\S_{\nu})^l F \Vert_{\hk_{K-l}(\S)} &\les& \dg, \qquad \forall \; l\leq K.
\eea
Notice that \eqref{finalF} implies for $1\leq l\leq K$
\beq
\left\|(\nabS_\nu)^{l-1}(\etaS,\xibS,\ombS)\right\|_{\hk_{K-l}(\S)}+r^{-1}\left\|\nu^{l-1}(\bS,\zS,\OmbS)\right\|_{\hk_{K-l}(\S)}=O(1).
\eeq
Next, remark that the estimates for $(\nabS_4)^l(F)$ follows directly from the transversality conditions \eqref{transversalityru} and the null transformation formulas \eqref{e4ffbla}. Then, the estimates for $(\nabS_3)^l(F)$ follows directly from $\nu=e_3^\S+\bS e_4^\S$ and the estimates for $\nu^{\leq l}(\bS)$. 
\end{enumerate}
This concludes the proof of the statement 1 of Corollary \ref{coro25}.\\ \\
{\bf Step 2.} It remains to prove Corollary \ref{coro25} in the more difficult case where
\bea\label{anormal}
\| f\|_{\hk_{s_{max}+1}(\S_0)}+(r^{\S_0})^{-1} \|(\fb,\ovla)\|_{\hk_{s_{max}+1}(\S_0)} &\les& \dg.
\eea
In view of the control \eqref{anormal} for $(f,\fb,\ovla)$, we may apply Lemma 7.3 in \cite{KS:Kerr1}, which yields
\beaa
\left|\frac{r^{\S_0}}{\ovr}-1\right|+\sup_{\S_0}\left|\frac{r^{\S_0}}{r}-1\right| &\les& \dg
\eeaa
so that $r$ and $r^{\S_0}$ are comparable, and hence
\bea\label{anomalous}
\| f\|_{\hk_{s_{max}+1}(\S_0)}+r^{-1} \|(\fb,\ovla)\|_{\hk_{s_{max}+1}(\S_0)} &\les& \dg.
\eea
Next, we introduce as in Proposition \ref{th24} the notations $K=s_{max}+1$ and $\nu:=\nu^{\S_0}$. Note that $\S_0$ can be viewed as a deformation of the unique background sphere sharing the same south pole. We proceed exactly as Step 1 to derive the desired estimates for our distinguished sphere $\S:=\S_0$.\\\\
In the following, we revisit all the terms estimated in Step 1, and explain that we can obtain the same estimates for $(\nabS_\nu)^l(F)$ when $l\geq 1$ under the weaker assumption \eqref{anomalous}.
\begin{enumerate}
    \item Commutators.\\
Applying Lemma \ref{commfor}, we have
\beaa
[\nabS_\nu,\nabS]F&=&\frac{2}{r}\nabS(F)+\GabS\nabS_\nu (F)+r^{-1}\GabS\c\dk^{\leq 1}(F),
\eeaa
which implies\footnote{Recall that $\nabS_4(F)$ can be easily estimated by \eqref{e4ffbla}.}
\beaa
\|[\nabS_\nu,\nabS]F\|_{\hk_l(\S)}&\les& \frac{1}{r^2}\| F\|_{\hk_{l+1}(\S)}+\frac{\epg}{r}\|\nabS_\nu(F)\|_{\hk_l(S)}.
\eeaa
Similarly, by Lemma \ref{commfor}, we have
\beaa 
\Vert [\nabS_\nu,\De^\S]F\Vert_{\hk_l(\S)}&\les &\frac{1}{r^3}\Vert F\Vert_{\hk_{l+2}(\S)}+\frac{\epg}{r^2}\|\nabS_\nu(F)\|_{\hk_l(S)}.
\eeaa 
Observe that the terms of $F$ on the R.H.S. gain a power of $r^{-1}$ when compared to the terms of $\nabS_\nu(F)$, which is consistent with the anomalous behavior of $(\fb,\ovla)$\footnote{In \eqref{atj}, the decay of $(\fb,\ovla)$ lose a power of $r^{-1}$ when compared to the decay of $f$, it is called anomalous behavior of $(\fb,\ovla)$}.
\item Coefficients $(\Cbdot_0,\Cbpdot)$.\\
Recalling \eqref{coefficientsestimates}, we have
\beaa
(\Cbdot_0,\Cbpdot)&=&r^{-1}O(\dg),
\eeaa
according to the anomalous behavior of $(\fb,\ovla)$. We have
\beaa 
&&\nu\left(\frac{1}{2\rS}\left(\Cbdot_0+\sum_p\Cbpdot\JpS\right)\right)\\
&=&O\left(\frac{1}{r^2}\right)\left(\Cbdot_0+\sum_p\Cbpdot\JpS\right)+\frac{1}{2\rS}\left(\nu(\Cbdot_0)+\sum_p\nu(\Cbpdot)\JpS\right).
\eeaa
Remark that the first term on the R.H.S. has order $r^{-3}O(\dg)$, which is consistent with the estimates \eqref{hhhhh}. Thus, we can obtain the same estimates as Step 1.\footnote{The second term on the R.H.S. appears on the L.H.S when applying Proposition \ref{Thm.GCMSequations-fixedS:contraction}.}
\item Error terms in $(h_1,\hb_1,h_2,\hb_2,h_3,h_4)$.\\
We recall Definition \ref{Definition:errorterms-prime}:
\beaa
r\err_1&=&F\c (r\Ga_b)+F\c (r\nabS)^{\leq 1}F,\\
r^2\err_2&=&(r\nabS)^{\leq 1}(r\err_1)+F\c r\dk \Ga_b.
\eeaa
According to the anomalous behavior of $(\fb,\ovla)$, we lose one power of $r^{-1}$ in $F\c(r\Ga_b)$. But if we inspect all the error terms in \eqref{Generalizedsystem} carefully (see Lemma 4.3 and Corollary 4.6 in \cite{KS:Kerr1} for their expressions), we can write that
\beaa
\err_1&=& f\c \Ga_b+(\fb,\ovla)\c \Ga_g+r^{-1}F\c F+ F\c \nabS F, \\
\err_2&=& (r\nabS)^{\leq 1}(r\err_1)+f\c r\dk \Ga_g+(\fb,\ovla)\c r\dk \Ga_b.
\eeaa
According to Assumption {\bf A1} and \eqref{anomalous}, we have
\beaa
\nu(\err_1)&=& \nu(F)\c\Ga_b+r^{-2}O(\fb,\ovla)+r^{-1}O(f)+r^{-1}\dkb^{\leq 1}(F\c\nu(F))+r^{-2}F\c F.
\eeaa
Thus, we gain a power of $r^{-1}$ for $(\fb,\ovla)$ when comparing to the other terms.
The estimate for $\nu(\err_2)$ is similar.
\item Nonlinear terms in $(h_1,\hb_1,h_2,\hb_2,h_3,h_4)$.\\
The nonlinear terms include (see Remark 4.10 in \cite{KS:Kerr1}):
\beq\label{nonlinear1}
\frac{2}{\rS}\ovla,\quad\frac{2}{(\rS)^2}\ovla,\quad\left(\ka-\frac{2}{\rS}\right)\ovla,\quad \left(\kab+\frac{2}{\rS}\right)\ovla,\quad\left(V-\frac{2}{(\rS)^2}\right)\ovla,
\eeq 
and
\bea\label{nonlinear2}
\frac{2(r-\rS)^2}{r(\rS)^2},\qquad \ddot{\Cb}_0+\sum_p\ddot{\Cb}^{(p)}\JpS,\qquad \ddot{M}_0+\sum_p\ddot{M}^{(p)}\JpS,
\eea
where
\beaa 
V&:=&-\left(\f12\ka\kab+\ka\omb+\kab\om\right).
\eeaa
We also have
\beaa
\nu\left(\frac{2}{\rS}\ovla\right)&=&-\frac{2\nu(\rS)}{(\rS)^2}\ovla+\frac{2}{\rS}\nu(\ovla)=O\left(\frac{1}{r^2}\right)\ovla+O\left(\frac{1}{r}\right)\nu(\ovla).
\eeaa 
Remark that the first term on the R.H.S. gains a power of $r^{-1}$. The estimates for the other terms in \eqref{nonlinear1} are similar. The first term in \eqref{nonlinear2} can be estimated as before. The other two terms in \eqref{nonlinear2} can be estimated by Corollary 5.20 in \cite{KS:Kerr1}, Recalling that their estimates gain a power of $r^{-1}$ for $\fb$ and are independent on $\ovla$.
\item The remaining terms in $(h_1,\hb_1,h_2,\hb_2,h_3,h_4)$.\\
According to Remark 4.11 in \cite{KS:Kerr1}, the remaining terms in $(h_1,\hb_1,h_2,\hb_2,h_3,h_4)$ either depend on $(\kadot,\kabdot,\mudot)$ or contain an additional power of $r^{-1}$ compared to the other terms. The latter can be ignored since we obtain the desired estimate despite the anomalous behavior of $(\fb,\ovla)$.\\ \\
Notice that $\nu^l(\kadot,\kabdot,\mudot)$ can be estimated by Lemma \ref{lemmacorollary} as before, Recalling that its proof gains a power of $r^{-1}$ for $\dg$, which is consistent with the anomalous behavior of $(\fb,\ovla)$.
\item Estimate for $|b_0|$.\\
To estimate $b_0$, we proceed exactly as Proposition \ref{th24}, just Recalling that \eqref{kabSkabest} also holds under the weaker assumption for $(\fb,\ovla)$.
\end{enumerate}
In summary, we can obtain, as in Step 1, the following estimate
\bea\label{nuFdg0}
\sum_{j=1}^K\|(\nabS_{\nu})^j(F)\|_{\hk_{K-j}(\S_0)}&\les& \dg.
\eea
To complete the proof of statement 2, we use the transversality conditions \eqref{transversalityru} and the null transformation formulas \eqref{e4ffbla} to deduce the estimates of $\nabS_4(F)$, and then recover the estimates of $\nabS_3(F)$ from $\nu=e_3^\S+\bS e_4^\S$. This concludes the proof of Corollary \ref{coro25}.
\appendix
\section{Proof of Lemma \ref{lemmaapp}}\label{nunu}
We denote
\beaa
X_0&:=&\Psi'(s)\pr_u+\pr_s.
\eeaa 
Notice that $X_0$ is tangent to $\Si_\#$ and recall that $\nu_\#$ is the unique tangent vectorfield to $\Si_\#$, normal to $S$ such that $\g(\nu_\#,e_4)=-2$. Together with \eqref{eq:decompositionofnullframeoncoordinatesframeforbackgroundfoliation}, we deduce
\bea
\begin{split}\label{X0andnu}
X_0&=\Psi'(s)\pr_u+\pr_s=\Psi'(s)\left(\frac{1}{z}e_3-\frac{\Omb}{z}e_4-\frac{2}{z}\Bb^a\pr_{y^a}\right)+e_4\\
&=\frac{\Psi'(s)}{z}e_3+\left(1-\frac{\Omb\Psi'(s)}{z}\right)e_4-\frac{2\Psi'(s)}{z}\Bb^a\pr_{y^a}\\
&=\frac{\Psi'(s)}{z}\nu_\#-\frac{2\Psi'(s)}{z}\Bb^a\pr_{y^a},
\end{split}
\eea
where
\bea \label{nudiesebdiese}
\nu_\#=e_3+b_\# e_4,\qquad b_\#=\frac{z}{\Psi'(s)}\left(1-\frac{\Omb\Psi'(s)}{z}\right).
\eea 
Notice that $\Phi^{\Jt}_\#(X_0,\pr_{y^1},\pr_{y^2})$ forms a basis of the tangent space of $\Si_{\Jt}$. We have
\beaa
\Phi^{\Jt}_\#(X_0)&=&\left(\Psi'(s)+X_0\big(U^{\Jt}\big)\right)\pr_u+\left(1+X_0\big(S^{\Jt}\big)\right)\pr_s,\\
\Phi^{\Jt}_\#(\pr_{y^a})&=&\pr_{y^a}\big(U^{\Jt}\big)\pr_{u}+\pr_{y^a}\big(S^{\Jt}\big)\pr_{s}+\pr_{y^a}.
\eeaa
By Lemma \ref{Lemma:Generalframetransf} and \eqref{eq:decompositionofnullframeoncoordinatesframeforbackgroundfoliation}, we can express $\nu^\Jt$ as follows:
\beaa 
\nu^{\Jt}&=&e_3^{\Jt}+b^{\Jt}e_4^{\Jt}\\
&=&e_3+b^{\Jt}e_4+O(F^{\Jt})(e_1,e_2,e_3,e_4)\\
&=&z\pr_u+\Omb\pr_s+2\Bb^a\pr_{y^a}+b^{\Jt}\pr_s+O(F^{\Jt})(e_1,e_2,e_3,e_4)\\
&=&A_{\Jt}\Phi^{\Jt}_\#(X_0)+B^a_{\Jt}\Phi^{\Jt}_\#(\pr_{y^a}),
\eeaa 
for suitable $A_{\Jt}$ and $B^a_{\Jt}$. In order to compute $A_{\Jt}$ and $B^a_{\Jt}$, we write
\beaa
A_{\Jt}\Phi^{\Jt}_\#(X_0)+B^a_{\Jt}\Phi^{\Jt}_\#(\pr_{y^a})&=&\left(A_{\Jt}\Psi'(s)+A_{\Jt}X_0(U^{\Jt})+B^a_{\Jt}\pr_{y^a}(U^{\Jt})\right)\pr_u\\
&+&\left(A_{\Jt}+A_{\Jt}X_0(S^{\Jt})+B^a_{\Jt}\pr_{y^a}(S^{\Jt})\right)\pr_s+B^a_{\Jt}\pr_{y^a}.
\eeaa
Comparing the coefficients of $\pr_u$ and $\pr_{y^a}$ with the above expression of $\nu^\Jt$, we obtain
\beaa 
A_{\Jt}&=&\frac{z}{\Psi'(s)}+A_0(U^\Jt,F^{\Jt}),\\
B^a_{\Jt}&=&2\Bb^a+B_0(F^{\Jt}),
\eeaa 
where 
\bea\label{A0B0prop}
A_0(U,F)=O\left(\dkb(U),X_0(U),F\right),\qquad B_0(F)=O(F).
\eea
Thus, we have
\bea\label{nunnbsl}
\nu^\Jt_\#=A_{\Jt}^\# X_0+ (B^a_{\Jt})^\# \pr_{y^a},
\eea 
where
\begin{align}
\begin{split}\label{Annbsl}
A_{\Jt}^\# =\frac{z\circ\Phi^{\Jt}}{\Psi'(s)}+A_0\circ\Phi^{\Jt},\qquad (B^a_{\Jt})^\# =2\Bb^a\circ\Phi^{\Jt}+B_0\circ\Phi^{\Jt}.    
\end{split}
\end{align} 
In view of \eqref{X0andnu}, \eqref{nunnbsl} and \eqref{Annbsl}, we infer
\beaa
\nu^\Jt_\#&=&\left(\frac{z\circ\Phi^{\Jt}}{\Psi'(s)}+A_0\circ\Phi^\Jt\right)X_0+\left(2\Bb^a\circ\Phi^{\Jt}+B_0\circ\Phi^\Jt\right)\pr_{y^a}\\
&=&\left(\frac{z\circ\Phi^{\Jt}}{\Psi'(s)}+A_0\circ\Phi^\Jt\right)\left(\frac{\Psi'(s)}{z}\nu_\#-\frac{2\Psi'(s)}{z}\Bb^a\pr_{y^a}\right)+\left(2\Bb^a\circ\Phi^{\Jt}+B_0\circ\Phi^\Jt\right)\pr_{y^a}\\
&=&\left(\frac{z\circ\Phi^{\Jt}}{z}+\frac{\Psi'(s)}{z}A_0\circ\Phi^\Jt\right)\nu_\#+\left(2\Bb^a\left(\frac{z-z\circ\Phi^{\Jt}}{z}\right)+2(\Bb^a\circ\Phi^{\Jt}-\Bb^a)+D_0\right)\pr_{y^a},
\eeaa
where
\bea\label{D0prop}
D_0=-\frac{2\Psi'(s)}{z}\Bb^a(A_0\circ\Phi^\Jt)+(B_0\circ\Phi^\Jt).
\eea
Thus, we obtain
\bea 
\begin{split}
A&=\frac{z\circ\Phi^{\Jt}-z}{z}+\frac{\Psi'(s)}{z}A_0\circ\Phi^\Jt,\\   
B^a&=2\Bb^a\left(\frac{z-z\circ\Phi^{\Jt}}{z}\right)+2(\Bb^a\circ\Phi^{\Jt}-\Bb^a)+D_0.
\end{split}
\eea 
Applying Taylor formula, recalling $\zc\in r\Ga_b$, $\Bb\in r^{-1}\Ga_b$, \eqref{A0B0prop} and \eqref{D0prop}, we have
\beaa 
A(U^\Jt,S^\Jt,F^{\Jt,\#})=O(\dk_\#^{\leq 1}(U^\Jt,S^\Jt),F^{\Jt,\#}),\qquad B(U^\Jt,S^\Jt,F^{\Jt,\#})=\frac{\epg}{r^2}O(\dk_\#^{\leq 1}(U^\Jt,S^\Jt),F^{\Jt,\#}).
\eeaa 
This concludes the proof of Lemma \ref{lemmaapp}.
\section{Proof of Lemma \ref{lemmaappa}}\label{AAAA}
Firstly, notice that \eqref{eq:ThmGCMS1}, \eqref{eq:ThmGCMS4} and Lemma \ref{lemma:comparison-gaS-ga} imply
\bea \label{departpoint}
r^{-1}\|(U^{\Jt},S^{\Jt})\|_{\hk_{s_{max}+1}(S)}+\|(f^{\Jt,\#},\fb^{\Jt,\#},\ovla^{\Jt,\#})\|_{\hk_{s_{max}+1}(S)}\les \dg.
\eea
Recall that the GCM system \eqref{nuGCM} is defined on $\S:=\S[\Psi(s),s,\La(s),\Lab(s),\Jt(s)]$. We pull it back by $\Phi^\Jt$ to obtain a GCM system for $\nab^{\S,\#}_{\nu_\#^\Jt}(f^{\Jt,\#},\fb^{\Jt,\#},\ovla^{\Jt,\#})$ on $S:=S(\Psi(s),s)$. Similar as \eqref{ffbovlac} and \eqref{ovbc}, we have, for $\epg$ small enough and $\ovr$ large enough
\beq\label{ffbdiese}
\|\nab^{\S,\#}_{\nu^\Jt_\#}(f^{\Jt,\#},\fb^{\Jt,\#})\|_{\hk_{s}(S)}\les \dg +\left|\left(\div^{\S,\#}\nab^{\S,\#}_{\nu^\Jt_\#}f^{\Jt,\#}\right)_{\ell=1}\right|+\left|\left(\div^{\S,\#}\nab^{\S,\#}_{\nu^\Jt_\#}\fb^{\Jt,\#}\right)_{\ell=1}\right|,
\eeq
where the $\ell=1$ modes are computed w.r.t. $\Jt^{(p)}(s)$. Recall from Lemma \ref{st15} that
\beaa 
|\nu^\Jt(\La(s))|&\les &|\La'(s)|\les \dg.
\eeaa 
Denoting
\beaa 
\Jt^{(\S,p)}&:=&\left((\Phi^\Jt)^{-1}\right)^\#(\Jt^{(p)}),
\eeaa 
we have from \eqref{nuJO1}
\bea \label{nuJO2}
\left|\nu^\Jt (\Jt^{(\S,p)})\right|=\left|\Phi^\Jt_\#(\nu^\Jt_\#)\left(\left((\Phi^\Jt)^{-1}\right)^\#\Jt^{(p)}\right)\right|=|\nu^\Jt_\#(\Jt^{(p)})|\leq 1,\quad\mbox{ on }\Si_\#.
\eea
Applying Lemmas \ref{commfor}, \ref{dint} and \eqref{nuJO2}, we infer
\beaa 
\nu^\Jt(\La(s))&=&\nu^\Jt \left(\int_\S \left(\divS f^{\Jt}\right)\Jt^{(\S,p)} \right)\\
&=&\zS\int_\S (\zS)^{-1}\left[\nu^\Jt\left(\left(\divS f^{\Jt}\right)\Jt^{(\S,p)}\right)+(\kaS+\bS \kabS)\left(\divS f^{\Jt}\right)\Jt^{(\S,p)}\right]\\
&=&\zS\int_\S(\zS)^{-1}\nu^\Jt\left(\divS f^{\Jt}\right)\Jt^{(\S,p)}+\zS\int_\S(\zS)^{-1} \left(\divS f^{\Jt}\right)\nu^\Jt(\Jt^{(\S,p)})+r^{-1}O(\dg)\\
&=&\zS\int_\S (\zS)^{-1}\left(\divS \nabS_{\nu^\Jt} f^{\Jt}\right)\Jt^{(\S,p)}+\zS\int_\S(\zS)^{-1}\left( [\nabS_{\nu^\Jt},\divS]f^{\Jt}\right)\Jt^{(\S,p)}+O(\dg)\\
&=&\zS\int_\S (\zS)^{-1}\left(\divS \nabS_{\nu^\Jt} f^{\Jt}\right)\Jt^{(\S,p)}+\epg r O\left(\nabS_{\nu^\Jt}f^\Jt\right)+O(\dg)\\
&=&\int_\S\left(\divS \nabS_{\nu^\Jt}f^{\Jt}\right)\Jt^{(\S,p)}+\epg r O\left(\nabS_{\nu^\Jt}f^\Jt\right)+O(\dg),
\eeaa 
which implies
\beaa 
\left|\left(\div^{\S,\#}\nab^{\S,\#}_{\nu^\Jt_\#}f^{\Jt,\#}\right)_{\ell=1}\right|&=&\sum_p\left|\int_\S \left(\divS \nabS_{\nu^\Jt} f^{\Jt}\right)\Jt^{(\S,p)}\right|\\
&\les&|\nu^\Jt(\La(s))|+\dg+\epg\|\nab_{\nu^\Jt}^{\S}(f^{\Jt})\|_{\hk_{s}(\S)}\\
&\les&\dg+\epg\|\nab_{\nu^\Jt_\#}^{\S,\#}(f^{\Jt,\#})\|_{\hk_{s}(S)}.
\eeaa
Similarly, we have
\beaa 
\left|\left(\div^{\S,\#}\nab^{\S,\#}_{\nu^\Jt_\#}\fb^{\Jt,\#}\right)_{\ell=1}\right|&\les&\dg+\epg\|\nab_{\nu^\Jt_\#}^{\S,\#}(\fb^{\Jt,\#})\|_{\hk_{s}(S)}.
\eeaa 
Thus, \eqref{ffbdiese} implies that for $\epg$ small enough
\bea\label{l=1moxin}
\|\nab^{\S,\#}_{\nu^\Jt_\#}(f^{\Jt,\#},\fb^{\Jt,\#})\|_{\hk_{s}(S)}\les\dg.
\eea 
Next, we recall that
\bea \label{UUSS}
\pr_{y^a}U^\Jt = \left(\UU(f^\Jt,\fb^\Jt,\Ga)_b Y^b_{(a)} \right)^\#,\qquad \pr_{y^a}S^\Jt = \left(\SS(f^\Jt,\fb^\Jt,\Ga)_b Y^b_{(a)} \right)^\#,
\eea 
where 
\beaa 
\UU(f,\fb,\Ga)=f+O(\epg)(f,\fb),\qquad \SS(f,\fb,\Ga)=\f12\left(-\Up f+\fb\right)+O(\epg)(f,\fb),
\eeaa 
see Proposition 5.14 and Remark 5.15 in \cite{KS:Kerr1}. Commuting \eqref{UUSS} with $\nu_\#^\Jt$ and using Lemma \ref{commfor}, \eqref{departpoint} and \eqref{l=1moxin}, we obtain
\bea\label{USffb}
r^{-1}\|\dkb(\nu_\#^\Jt(U^\Jt,S^\Jt))\|_{\hk_s(S)}\les\dg+\|\nab_{\nu_\#^\Jt}^{\S,\#}(f^{\Jt,\#},\fb^{\Jt,\#})\|_{\hk_s(S)}\les\dg.
\eea
We deduce
\bea \label{xing}
r^{-1}\|\nu_\#^\Jt(U^\Jt,S^\Jt)\|_{\hk_s(S)}\les\dg+\left|\nu_\#^\Jt(U^\Jt,S^\Jt)\big|_{\ga(s)}\right|,
\eea 
where $\ga(s)$ is the curve of South Poles of $S(\Psi(s),s)$.\\\\
Next, we estimate $\nu_\#^\Jt(U^\Jt,S^\Jt)\big|_{\ga(s)}$. Recall that we have 
\beaa 
U^\Jt(\ga(s))=0,\qquad S^\Jt(\ga(s))=0.
\eeaa 
Then, we have
\beaa 
\nu_{SP}(U^\Jt,S^\Jt)\big|_{\ga(s)}=0,
\eeaa 
where $\nu_{SP}$ is the tangent vector field along $\ga(s)$. Notice that $\nu_{SP}$ is parallel to $X_0=\Psi'(s)\pr_u+\pr_s$. Combining with \eqref{X0andnu} and \eqref{departpoint}, we obtain
\bea\label{UJtSJtdg}
\nu_\#(U^\Jt,S^\Jt)\big|_{\ga(s)}=2\Bb^a\pr_{y^a}(U^\Jt,S^\Jt)\big|_{\ga(s)}=O(\dg).
\eea
Applying \eqref{nujnbsl}, \eqref{departpoint} and \eqref{UJtSJtdg}, we infer
\bea\label{nuUSgj}
\nu_\#^\Jt(U^\Jt,S^\Jt)\big|_{\ga(s)}=O(\dg).
\eea
Thus, \eqref{xing} and \eqref{nuUSgj} yield
\bea\label{nuJtUJtSJt}
r^{-1}\|\nu^\Jt_\#(U^\Jt,S^\Jt)\|_{\hk_{s}(S)}\les\dg.
\eea
To conclude, it remains to prove \eqref{nuJtUJtSJt} with $\nu^\Jt_\#$ replaced by $\nu_\#$. To this end, we consider the connected open component $\varpi_\la \subset S$ centered at $\ga(s)$, i.e.
\beaa 
\varpi_\la&:=&\{p\in S/\,d(p,\ga(s))\leq \la\},
\eeaa 
where $d$ is the geodesic distance on $S$ induced by $\g$ and $\la>0$ is a constant. Define $\Om\subset(0,+\infty)$ the set of $\la$ such that the following holds:
\bea \label{bootjiashe}
r^{-1}\|\nu_\#(U^\Jt,S^\Jt)\|_{\hk_{s}(\varpi_\la)}\leq\dg^\f12.
\eea 
Notice that \eqref{UJtSJtdg} implies that \eqref{bootjiashe} holds for $\varpi_\la$ when $\la$ small enough, thus $\Om$ is non-empty. Note also that $\Om$ is closed. Recalling \eqref{departpoint} and \eqref{bootjiashe}, we have for $\la\in\Om$
\beaa 
\|A(U^\Jt,S^\Jt,F^{\Jt,\#})\|_{\hk_s(\varpi_\la)}\les r\dg^\f12,\qquad \|B(U^\Jt,S^\Jt,F^{\Jt,\#})\|_{\hk_s(\varpi_\la)}\les \epg\dg^\f12 r^{-1}.
\eeaa 
Together with \eqref{nujnbsl}, \eqref{departpoint} and \eqref{nuJtUJtSJt}, we have for $\dg$ small enough
\beq\label{finalboot}
\|\nu_\#(U^\Jt,S^\Jt)\|_{\hk_{s}(\varpi_\la)}\les \|\nu_\#^\Jt(U^\Jt,S^\Jt)\|_{\hk_{s}(\varpi_\la)}+\|\pr_{y^a}(U^\Jt,S^\Jt)\|_{\hk_{s}(\varpi_\la)}\les r\dg\leq\f12 r\dg^\f12.
\eeq
Thus, we have $\la+\ep\in\Om$ for $\ep$ small enough. By definition of $\Om$, we obtain $(0,\la+\ep)\subset\Om$, which implies that $\Om$ is open in $(0,+\infty)$. Hence $\Om$ is a open-closed non-empty subset of $(0,+\infty)$, which implies $\Om=(0,+\infty)$. Since $S$ is compact, there exists a $\la\in\Om$ large enough such that $\varpi_\la=S$. Finally, injecting $\varpi_\la=S$ into \eqref{finalboot}, we obtain
\bea 
r^{-1}\|\nu_\#(U^\Jt,S^\Jt)\|_{\hk_{s}(S)}\les\dg.
\eea 
This concludes the proof of Lemma \ref{lemmaappa}.
\section{Proof of Lemma \ref{lemmaappb}}\label{BBBB}
Notice that \eqref{cor6.11} is a direct consequence of Corollary 6.11 in \cite{KS:Kerr1}. Also, the proof of \eqref{nucor6.11} is similar to \eqref{cor6.11} and Lemma \ref{lemmaappa}, so we only provide a sketch.\\ \\
We denote
\beaa 
\de H^\#&:=&H^{\Jt,\#}-H^{\Jh,\#}, \qquad H\in\{f,\fb,\ovla,\ovb\},\\
\de h&:=&h^{\Jt}-h^{\Jh},\qquad\qquad\;\;\, h\in\left\{\Cbdot_0,\,\Cbpdot,\,M_0,\,\Mp,\,h_1,\,\hb_1,\,h_2,\,\hb_2,\,h_3,\,h_4\right\}.
\eeaa 
We recall the following GCM system, see (C.1)-(C.5) in \cite{KS:Kerr1},
\bea
\begin{split}\label{GCMde}
    \curl^{\St,\#}\de f&=\de h_1,\\
    \curl^{\St,\#}\de \fb&=\de \hb_1,\\
    \div^{\St,\#}\de f+\frac{2}{\rS}\de\ovla-\frac{2}{(\rS)^2}\de \ovb &=\de h_2,\\
    \div^{\St,\#}\de \fb+\frac{2}{\rS}\de\ovla+\frac{2}{(\rS)^2}\de \ovb &=\de \hb_2,\\
    \left(\De^{\St,\#}+\frac{2}{(\rS)^2}\right)\de\ovla&=\de M_0+\sum_p\de\Mp\Jp+\frac{1}{2\rS}\left(\de\Cb_0+\sum_p\de \Cbp\Jp\right)\de h_3,\\
    \De^{\St,\#}\de\ovb-\frac{1}{2}\div^{\St,\#}(\de \fb-\de f)&=\de h_4,
\end{split}
\eea
and 
\bea 
(\div^{\St,\#}\de f)_{\ell=1}=\de \La,\qquad (\div^{\St,\#}\de\fb)_{\ell=1}=\de \Lab.
\eea 
See Appendix C in \cite{KS:Kerr1} for the structures of $(\de h_1,\de \hb_1, \de h_2,\de \hb_2, \de h_3, \de h_4)$. We can deduce the following analog of \eqref{hhhhh}:
\beq
\|\nu^\Jt_\#(\de h_1,\de\hb_1,\de h_2,\de\hb_2,\de h_4)\|_{\hk_{s}(S)}+r\|\nu^\Jt_\#(\de h_3)\|_{\hk_{s}(S)}\les \epg r^{-2}\|(\nu_\#^\Jt)^{\leq 1}(\Jt-\Jh)\|_{\hk_{s}(S)}.
\eeq 
We commute the GCM system \eqref{GCMde} with $\nu^\Jt_\#$ to obtain the following analog of \eqref{ffbdiese}:
\begin{align}
\begin{split}\label{nabnuj}
\left\|\nab^{\St,\#}_{\nu^\Jt_\#}(\de f,\de\fb)\right\|_{\hk_{s+1}(S)}&\les \epg r^{-1}\|\Jt-\Jh\|_{\hk_{s}(S)}+\epg r^{-1}\|\nu_\#^\Jt(\Jt-\Jh)\|_{\hk_{s}(S)}\\
&+\left|\left(\div^{\St,\#}\nab^{\S,\#}_{\nu^\Jt_\#}(\de f)\right)_{\ell=1}\right|+\left|\left(\div^{\St,\#}\nab^{\St,\#}_{\nu^\Jt_\#}(\de \fb)\right)_{\ell=1}\right|.
\end{split}
\end{align}
Recalling (B.25) in \cite{KS:Kerr1} for a similar structure of $(\de \La,\de \Lab)$, we can deduce that
\begin{align}
\begin{split}
|\nu^\Jt_\#(\de \La,\de\Lab)| &\les \dg \left\|\nab^{\St,\#}_{\nu^\Jt_\#}(\de f,\de\fb)\right\|_{\hk_s(S)}+\left\|(\de f,\de\fb)\right\|_{\hk_s(S)}\\
&+r^{-1}(r^{-1}+\epg)\left\|\nu_\#^\Jt(\de U,\de S)\right\|_{\hk_s(S)}+r^{-1}\left\|(\de U,\de S)\right\|_{\hk_s(S)}.    
\end{split}
\end{align}
Then, we proceed as in \eqref{l=1moxin} to obtain
\beaa
&&\left|\left(\div^{\St,\#}\nab^{\S,\#}_{\nu^\Jt_\#}(\de f)\right)_{\ell=1}\right|+\left|\left(\div^{\St,\#}\nab^{\St,\#}_{\nu^\Jt_\#}(\de \fb)\right)_{\ell=1}\right|\\
&\les&\epg\left\|\nab^{\St,\#}_{\nu^\Jt_\#}(\de f,\de\fb)\right\|_{\hk_s(S)}+\|(\de f,\de\fb)\|_{\hk_s(S)}+r^{-1}(r^{-1}+\epg)\left\|\nu_\#^\Jt(\de U,\de S)\right\|_{\hk_s(S)}+r^{-1}\left\|(\de U,\de S)\right\|_{\hk_s(S)}.    
\eeaa
Injecting it in \eqref{nabnuj} and applying \eqref{cor6.11}, for $\epg$ small enough we obtain
\bea 
\begin{split}\label{C6}
\left\|\nab^{\St,\#}_{\nu^\Jt_\#}(\de f,\de\fb)\right\|_{\hk_{s+1}(S)}&\les \epg r^{-1}\|\Jt-\Jh\|_{\hk_{s}(S)}+\epg r^{-1}\|\nu_\#^\Jt(\Jt-\Jh)\|_{\hk_{s}(S)}\\
&+r^{-1}(r^{-1}+\epg)\left\|\nu_\#^\Jt(\de U,\de S)\right\|_{\hk_s(S)}.    
\end{split}
\eea 
Then, relying \eqref{UUSS} and using \eqref{cor6.11}, we easily obtain the following analog of \eqref{USffb}, 
\bea \label{xingxing}
r^{-1}\|\dkb(\nu_\#^\Jt(\de U,\de S))\|_{\hk_s(S)}\les \epg r^{-1}\|\Jt-\Jh\|_{\hk_s(S)}+\left\|\nab^{\St,\#}_{\nu_\#^\Jt}(\de f,\de\fb)\right\|_{\hk_s(S)}.
\eea 
Then, \eqref{C6} and \eqref{xingxing} imply, for $\ovr$ large enough and $\epg$ small enough
\bea \label{xingxingxing}
\begin{split}
&r^{-1}\|\nu^\Jt_\#(\de U,\de S)\|_{\hk_{s+1}(S)}+\left\|\nab^{\St,\#}_{\nu^\Jt_\#}(\de f,\de\fb)\right\|_{\hk_{s+1}(S)}\\
\les&\; \epg r^{-1}\|\Jt-\Jh\|_{\hk_{s}(S)}+\epg r^{-1}\|\nu_\#^\Jt(\Jt-\Jh)\|_{\hk_{s}(S)}+r^{-1}\left|\nu^\Jt_\#(\de U,\de S)\big|_{\ga(s)}\right|.
\end{split}
\eea
Recalling \eqref{UUSS}, proceeding as \eqref{UJtSJtdg} and \eqref{nuUSgj}, we obtain the following estimate:
\bea \label{zuihou}
\nu^\Jt_\#(\de U,\de S)\big|_{\ga(s)}=O(\dg).
\eea 
Finally \eqref{xingxingxing} and \eqref{zuihou} yield \eqref{nucor6.11}. This concludes the proof of Lemma \ref{lemmaappb}.
\section{Proof of Lemma \ref{transportlemma}}\label{lemmaappd}
For any integer $k$, applying Lemma \ref{nonpo} we infer
\beaa 
\nu_\#\left(\int_S |\dkb^k h|^2\right)&=&(e_3+ b_\# e_4)\left(\int_S |\dkb^k h|^2\right)\\
&=&z\int_S\left( z^{-1}e_3(|\dkb^k h|^2)-z^{-1}\Omb e_4(|\dkb^k h|^2)+ z^{-1}\kab |\dkb^k h|^2-z^{-1}\Omb\ka |\dkb^k h|^2\right)\\
&+&\Omb\int_S \left(e_4(|\dkb^k h|^2)+\ka |\dkb^k h|^2\right)+b_\#\int_S \left(e_4 (|\dkb^k h|^2)+ \ka |\dkb^k h|^2\right).
\eeaa 
Recalling from \eqref{nudiesebdiese} that $b_\#+\Omb=\frac{z}{\Psi'(s)}$, we obtain
\beaa 
\nu_\#\left(\int_S |\dkb^k h|^2\right)&=&z\int_S\left( z^{-1}e_3(|\dkb^k h|^2)-z^{-1}\Omb e_4(|\dkb^k h|^2)+ z^{-1}\kab |\dkb^k h|^2-z^{-1}\Omb\ka |\dkb^k h|^2\right)\\
&+&z\int_S z^{-1}(b_\#+\Omb)\left(e_4(|\dkb^k h|^2)+\ka |\dkb^k h|^2\right)\\
&=&z\int_S z^{-1}\left(\nu_\#(|\dkb^k h|^2)+(\kab+b_\#\ka) |\dkb^k h|^2\right)\\
&=&z\int_S z^{-1}\nu_\#(|\dkb^k h|^2)+ O(r^{-1})\|h\|_{\hk_k(S)}^2.
\eeaa
Applying Lemmas \ref{commfor}, \ref{lemmaapp}, \ref{lemmaappa}, $\zc\in r\Ga_b$ and the  divergence theorem, we infer
\beaa 
z\int_S z^{-1}\nu_\#(|\dkb^k h|^2)&=&z\int_S z^{-1}(1+A)^{-1}\nu_\#^\Jt(|\dkb^k h|^2)-z\int_S z^{-1}(1+A)^{-1}B^a\pr_{y^a}(|\dkb^k h|^2)\\
&=&z\int_S z^{-1}(1+A)^{-1}\nu_\#^\Jt(|\dkb^k h|^2)+z\int_{S}\div\left(z^{-1}(1+A)^{-1}B\right)|\dkb^k h|^2\\
&=&z\int_S z^{-1}(1+A)^{-1}\nu_\#^\Jt(|\dkb^k h|^2)+O(\dg)\|h\|_{\hk_k(S)}^2.
\eeaa 
Combining the last two identities and applying Lemma \ref{commfor}, we infer
\beaa 
\left|\nu_\#\left(\int_S |\dkb^k h|^2\right)\right|&\les& \int_S\left|\nu_\#^\Jt(|\dkb^k h|^2)\right|+\|h\|_{\hk_k(S)}^2\les\left\|\nu_\#^\Jt(h)\right\|_{\hk_k(S)}\|h\|_{\hk_k(S)}+\|h\|_{\hk_k(S)}^2,
\eeaa 
which implies
\beaa 
\left|\nu_\#\left(\|h\|_{\hk_k(S)}\right)\right|&\les&\left\|\nu_\#^\Jt(h)\right\|_{\hk_k(S)}+\|h\|_{\hk_k(S)}.
\eeaa
Integrating it from $S=\ovS$ and recalling that $|\ovI|\les\epg$, we infer
\beaa 
\|h\|_{\hk_k(S)}-\|h\|_{\hk_k(\ovS)}&\les&\epg\sup_S\left\|\nu_\#^\Jt(h)\right\|_{\hk_k(S)}+\epg\sup_S\|h\|_{\hk_k(S)}.
\eeaa 
Taking the supremum over $S\subset\Si_\#$, for $\epg$ small enough, we have
\beaa 
\sup_{S\subset\Si_\#}\|h\|_{\hk_k(S)} \leq (1+O(\epg))\|h\|_{\hk_k(\ovS)}+\sup_{S\subset\Si_\#}\left\|\nu_\#^\Jt(h)\right\|_{\hk_k(S)}.
\eeaa
This concludes the proof of Lemma \ref{transportlemma}.
\section*{Declarations}
\addcontentsline{toc}{section}{Declarations}
{\bf Acknowledgements.} The author is very grateful to J\'er\'emie Szeftel for his support, discussions, encouragements and patient guidance.\\ \\
{\bf Funding.} This work was partially supported by the ERC grant ERC-2016 CoG 725589 EPGR.\\ \\
{\bf Competing Interest statements.} Conflict of interest does not exist in the manuscript.

\end{document}